\documentclass{aimsplain}
\usepackage{amsmath}
\usepackage[utf8]{inputenc}
\usepackage[T1]{fontenc}
\allowdisplaybreaks[4]
\usepackage{paralist}
\usepackage{graphicx}
\usepackage{mathtools}
\mathtoolsset{showonlyrefs,showmanualtags}
\usepackage{enumerate}
\usepackage{dsfont}
\usepackage{lmodern}
\usepackage{microtype}
 \usepackage[colorlinks=true]{hyperref}
\hypersetup{
    pdftitle={A non-local problem for the Fokker-Planck equation related to the Becker-Döring model},
    pdfauthor={Joseph G. Conlon, André Schlichting},
    pdfsubject={},
    pdfkeywords={},
    urlcolor=blue, citecolor=red
}

\def\currentvolume{X}

    \def\ppages{X--XX}
     \def\DOI{10.3934/xx.xx.xx.xx}

\newcommand{\cC}{\mathcal{C}}
\newcommand{\cD}{\mathcal{D}}
\newcommand{\cF}{\mathcal{F}}
\newcommand{\cG}{\mathcal{G}}
\newcommand{\cH}{\mathcal{H}}
\newcommand{\cM}{\mathcal{M}}
\newcommand{\cP}{\mathcal{P}}
\newcommand{\cS}{\mathcal{S}}
\newcommand{\cX}{\mathcal{X}}

\newcommand{\Z}{\mathbf{Z}}
\newcommand{\R}{\mathbf{R}}

\newcommand{\EX}[1][E]{\ensuremath{\mathds{#1}}}
\newcommand{\Prob}[1][P]{\ensuremath{\mathds{#1}}}

\DeclarePairedDelimiter{\abs}{\lvert}{\rvert}
\DeclarePairedDelimiter{\norm}{\lVert}{\rVert}
\DeclarePairedDelimiter{\bra}{(}{)}
\DeclarePairedDelimiter{\pra}{[}{]}
\DeclarePairedDelimiter{\set}{\{}{\}}
\DeclarePairedDelimiter{\skp}{\langle}{\rangle}

\DeclareMathAlphabet{\mathup}{OT1}{\familydefault}{m}{n}
\newcommand{\dx}{\mathop{}\!\mathup{d}}
\newcommand{\pderiv}[3][]{\frac{\mathup{d}^{#1}#2}{\mathup{d}#3^{#1}}}

\newcommand{\sig}{\sigma}

\newcommand{{\ba}}{\bf a}
\newcommand{\ve}{\varepsilon}
\newcommand{\la}{\lambda}
\newcommand{\La}{\Lambda}
\newcommand{\ga}{\gamma}
\newcommand{\Ga}{\Gamma}
\newcommand{\pa}{\partial}
\newcommand{\ra}{\rightarrow}

\newcommand{\del}{\delta}

\newcommand{\al}{\alpha}

\DeclareMathOperator*{\argmin}{arg\,min}
\newcommand{\eps}{\varepsilon}
\DeclareMathOperator{\Ent}{Ent}
\newcommand{\bfOne}{\mathbf{1}}
\newcommand{\eq}{{\ensuremath{\scriptscriptstyle\mathup{eq}}}}
\DeclareMathOperator{\LSI}{LSI}

\newtheorem{theorem}{Theorem}[section]
\newtheorem{corollary}[theorem]{Corollary}
\newtheorem{lemma}[theorem]{Lemma}
\newtheorem{proposition}[theorem]{Proposition}
\newtheorem{assumption}[theorem]{Assumption}
\theoremstyle{definition}
\newtheorem{definition}[theorem]{Definition}
\newtheorem{remark}[theorem]{Remark}
\numberwithin{equation}{section}

\title[A non-local problem for the Fokker-Planck equation]{A non-local problem for the Fokker-Planck equation related to the Becker-D\"{o}ring model}

\author[Joseph G. Conlon and André Schlichting]{}


\email{conlon@umich.edu}
\email{schlichting@iam.uni-bonn.de}

\thanks{$^*$ Corresponding author.}

\begin{document}
\maketitle

\centerline{\scshape Joseph G. Conlon}
\medskip
{\footnotesize
 \centerline{University of Michigan}
   \centerline{Department of Mathematics}
   \centerline{Ann Arbor, MI 48109-1109, USA}
}

\medskip

\centerline{\scshape Andr\'{e} Schlichting$^*$}
\medskip
{\footnotesize
 \centerline{Universit\"at Bonn}
   \centerline{Institut f\"ur Angewandte Mathematik}
   \centerline{Endenicher Allee 60, 53129 Bonn, Germany}
}

\bigskip


\begin{abstract}
This paper concerns a Fokker-Planck equation on the positive real line modeling nucleation and growth of clusters. The main feature of the equation is the dependence of the driving vector field and boundary condition on a non-local order parameter related to the excess mass of the system.

The first main result concerns the well-posedness and regularity of the Cauchy problem. The well-posedness is based on a fixed point argument, and the regularity on Schauder estimates. The first a priori estimates yield Hölder regularity of the non-local order parameter, which is improved by an iteration argument.

The asymptotic behavior of solutions depends on some order parameter $\rho$ depending on the initial data. The system shows different behavior depending on a value  $\rho_s>0$,  determined from the potentials and diffusion coefficient. For $\rho \leq \rho_s$, there exists an equilibrium solution $c^\eq_{(\rho)}$. If $\rho\le\rho_s$ the solution converges strongly to $c^\eq_{(\rho)}$, while if $\rho > \rho_s$ the solution converges weakly to $c^\eq_{(\rho_s)}$. The excess $\rho - \rho_s$ gets lost due to the formation of larger and larger clusters. In this regard, the model behaves similarly to the classical Becker-Döring equation.

The system possesses a free energy, strictly decreasing along the evolution, which establishes the long time behavior. In the subcritical case $\rho<\rho_s$ the entropy method, based on suitable weighted logarithmic Sobolev inequalities and interpolation estimates, is used to obtain explicit convergence rates to the equilibrium solution.

The close connection of the presented model and the Becker-Döring model is outlined by a family of discrete Fokker-Planck type equations interpolating between both of them. This family of models possesses a gradient flow structure, emphasizing their commonality.
\end{abstract}

\newpage

\section{Introduction}

\subsection{The model, its well-posedness and convergence to equilibrium}
In this paper we shall be concerned with a non-linear non-local problem associated to the Fokker-Planck equation on the half line $\R^+=[0,\infty)$,
\begin{equation}\label{A1}
\pa_t c(x,t) +\partial_x \bra[\big]{b(x,t)c(x,t)} \ = \ \pa_x^2 \bra[\big]{a(x)c(x,t)} \ , \quad x,t\in \R^+.
\end{equation}
We shall be primarily interested in the large time behavior of solutions to \eqref{A1}  with non-negative initial data and Dirichlet boundary condition at $x=0$. We shall assume that $a(\cdot)$ is differentiable and strictly positive, and that the drift $b(\cdot,\cdot)$ has the form
\begin{equation}\label{B1}
b(x,t) \ = \ a(x)\bra[\big]{\theta(t)W'(x)-V'(x)}  \ ,
\end{equation}
where $V(\cdot), W(\cdot)$ are $\cC^1$ functions and $\theta(\cdot)$ is continuous. We show in $\S6$ that with the choice (\ref{B1}) of the drift $b(\cdot,\cdot)$ together with suitable Dirichlet boundary condition and conservation law, the evolution (\ref{A1}) may be considered a continuous version of the discrete Becker-D\"{o}ring (BD) model~\cite{BD1935}.

At this point, equation~\eqref{A1} with drift \eqref{B1} is a Fokker-Planck equation with time and space dependent coefficients. In particular, if the function $\theta(\cdot)$ in \eqref{B1} is constant $\theta(\cdot)\equiv\theta$, then  $c(x,t)=c_\theta^\eq(x)$, where
\begin{equation}\label{C1}
c_\theta^\eq(x) \ = \ a(x)^{-1}\exp\bra*{-V(x)+\theta W(x)} \ ,
\end{equation}
 is a steady state solution of \eqref{A1}, \eqref{B1}.  Furthermore it is known, under fairly general assumptions on $a(\cdot),V(\cdot), W(\cdot), \theta$,  that the solution $c(\cdot,t)$ to \eqref{A1} with Dirichlet boundary condition $c(0,t)=c_\theta^\eq(0)$, converges as $t\ra\infty$ to $c_\theta^\eq(\cdot)$. In this case convergence follows by establishing a positive lower bound on the Dirichlet form \cite{fot} associated to (\ref{A1}).                         
 
 In the problem we study here $\theta(\cdot)$ is non-constant in time and is determined by the conservation law
\begin{equation}\label{D1}
\theta(t)+\int_{0}^\infty W(x)c(x,t)  \dx{x}  = \ \rho \ , \quad\text{where $\rho>0$ is constant.}
\end{equation}
By assuming that the inital data satisfies $\int_0^\infty W(x) c(x,0) \dx{x}< \infty$, the constraint~\eqref{D1} is proven to be satisfied for short time by a fixed point argument. In addition for global in time existence, a lower bound on $\theta(t)$ is provided, which uses the particular choice of~\eqref{B1}. 
In the application of this model to coarsening, $\theta$ models the gaseous phase or available monomer concentration and $c$ is the volume cluster density, the constraint~\eqref{D1} corresponds to the conservation of total mass. 
The constraint makes the Fokker-Planck equation \eqref{A1}, \eqref{B1} non-local and non-linear. Additionally, we impose a Dirichlet boundary condition which is consistent with the requirement that $c_\theta^\eq(x)$ is a stationary solution to \eqref{A1}, \eqref{B1}, but not necessarily satisfying the constraint~\eqref{D1}. Our Dirichlet condition is therefore given by
\begin{equation}\label{E1}
  c(0,t) \ = \ c_{\theta(t)}^\eq(0) \ = \ a(0)^{-1}\exp\bra*{-V(0)+\theta(t)W(0)} \  , \quad t>0 \ .
\end{equation}
It turns out that the above Dirichlet condition~\eqref{E1} is also thermodynamically consistent, 
since the system \eqref{A1}, \eqref{B1}, \eqref{D1}, \eqref{E1} has a free energy functional acting as 
Lyapunov function for the evolution (see §\ref{s:intro:rate}), which is the main tool for the investigation of the long-time limit. 
To specify the long-time limit, we observe that if $W$ is assumed to be a positive function such that 
$W(\cdot)a(\cdot)^{-1}\exp\pra{-V(\cdot)}$ is integrable on $(0,\infty)$, then $W(\cdot)c_\theta^\eq(\cdot)$ is integrable for $\theta\le 0$. 
Furthermore, the function $\theta\mapsto \theta+\|W(\cdot)c_\theta^\eq(\cdot)\|_1$ is strictly increasing and 
maps $(-\infty,0]$ to $(-\infty,\rho_s]$ where 
\begin{equation}\label{def:rhos}
  \rho_s= \|W(\cdot)c_0^\eq(\cdot)\|_1 = \int_0^\infty W(x) a(x)^{-1} \exp\bra*{-V(x)}  \dx{x}  \ .
\end{equation}
We denote by $\theta_{\eq}(\cdot)$ the inverse function with domain $(-\infty,\rho_s]$.  Evidently $\theta_\eq(\rho_s)=0$, and so we may extend  $\theta_\eq(\cdot)$ in a continuous way to have domain $\R$ by setting $\theta_\eq(\rho)=0$ for $\rho>\rho_s$. It is proven below that $c_{(\rho)}^{\eq} = c_{\theta^{\eq}(\rho)}^{\eq}$ with $\rho$ given as right hand side of~\eqref{D1} is the long-time limit of the evolution equation \eqref{A1}, \eqref{B1}, \eqref{D1}, \eqref{E1}.

The goal of the work is to establish well-posedness of the system, to investigate the long-time behavior and to obtain the rate for convergence in the subcritical regime $\rho < \rho_s$.

\bigskip

The above  model \eqref{A1}, \eqref{B1}, \eqref{D1}, \eqref{E1} bears similarities with the classical Becker-Döring model~\cite{BD1935}, and is also closely related to the Lifshitz-Slyozov-Wagner (LSW) model of coarsening~\cite{LS61,W61}. The Becker-Döring model is discrete with evolution determined by a countable set of ODEs, whereas  the LSW model is continuous with a nonlocal conservation law on $\R^+$. The close connection between the Becker-Döring and LSW models was investigated in several works \cite{P97,Collet99,vel,Collet00,np2,LM02,Collet02,N03,Collet04,np3,C10,Schlichting}. The models from these works closest to the one investigated here are those considered in~\cite{C10,vel}. There a diffusive LSW equation very similar to the equations \eqref{A1}, \eqref{B1}, \eqref{D1} is studied. However, instead of the boundary condition \eqref{E1}, a homogeneous Dirichlet boundary condition $c(0,t)=0$ is considered. The boundary condition is crucial, since it changes the stationary states in a nontrivial way. Another modified LSW equation was proposed in \cite{Collet99} and \cite[§4]{Collet02}, which is obtained by a formal second order expansion of the Becker-Döring equation. The difference between the models obtained in~\cite{Collet99,Collet02} and the system \eqref{A1}, \eqref{B1}, \eqref{D1}, \eqref{E1} is that the coupling of the constraint and the boundary condition is different. There has been no rigorous mathematical analysis done on the model presented in \cite{Collet99,Collet02}. Our present analysis may be applicable here also, given  the slight change of boundary conditions.

\bigskip

Let us specify the set of assumptions on the functions $a(\cdot)$, $V(\cdot)$, $W(\cdot)$ needed to obtain statements for the system  \eqref{A1}, \eqref{B1}, \eqref{D1}, \eqref{E1}.
\begin{assumption}\label{ass:VW:intro}\
\begin{enumerate}[ (a) ]
 \item $a(\cdot), \ V(\cdot), \ W(\cdot) \in \cC^2(\R^+)$ satisfy for some $C_0>0$
 \begin{equation}\label{F1:0:intro}
    \bra[\big]{ \abs{V''(x)} + \abs{W''(x)}} \; a(x) + \bra[\big]{ \abs{V'(x)} + \abs{W'(x)}} \;  \abs{a'(x)} + \abs{a''(x)} \leq C_0 \qquad\text{ for } x\in \R^+ .
 \end{equation}
 \item $W(\cdot)$ is an increasing function with $W(0)>0$ and $\inf_{x\in \R^+} a(x)\geq c_0 >0$.
 \item For any $\del>0$ there exists $x_\del>0$ such that
 \begin{equation}\label{F1:1:intro}
   |V'(x)| + |(\log a(x))'|  \ \le \ \del W'(x) \ ,  \quad \text{if } x\ge x_\del \ .
 \end{equation}
and
\begin{equation}\label{F1:2:intro}
\abs{V''(x)}  + |(\log a)''(x)| + \abs{W''(x)} \ \le \ \delta W'(x)^2 \ ,   \quad \text{if } x\ge x_\del \ .
\end{equation}
  \item There exists $C_0,c_0>0$ such that
\begin{equation}\label{G1:intro}
c_0 \leq a(x) W'(x)^2 \leq C_0 W(x) \ ,  \qquad \text{ for } x\in \R^+ \ .
\end{equation}
  \item The function  $W(\cdot) a(\cdot)^{-1} \exp\bra*{-V(\cdot)}$ is integrable on $(0,\infty)$.
\end{enumerate}
\end{assumption}
\begin{remark}\label{rem:ass}
  Assumption~\ref{ass:VW:intro}~(a) represents limits on the growth rates of the coefficients, and is needed to avoid any question or discussion of the explosion of the associated SDE. These assumptions are typical for establishing well-posedness of linear Fokker-Planck equations. The physical interpretation of the model associates with $c$ a cluster volume distribution and with $W$ a bulk energy per unit volume. Hence, we ask for a monotone relation between bulk energy and volume. In this interpretation the diffusion coefficient $a$ is an overall reactivity and so is strictly positive. The potential $V$ corresponds to surface energy per unit volume, and therefore has smaller growth at infinity than the bulk energy $W$.  This is expressed in part (c) of Assumption~\ref{ass:VW:intro}. Assumption~\ref{ass:VW:intro}~(e) ensures that  $\rho_s$ as defined in~\eqref{def:rhos}  satisfies  $\rho_s<\infty$. This is the physically interesting case, showing two regimes in the longtime limit. The case $\rho_s=\infty$ can be handled with small modifications. 
  
  Assumption~\ref{ass:VW:intro}~(d) has no direct physical interpretation, but is crucial for the qualitative and quantitative investigation of the long-time limit (see also condition~\eqref{G1:alpha} in Theorem~\ref{thm:QuantLongTime}). In particular, it implies that  $W(x) \to \infty$ as $x\to \infty$. To see this observe from (a)  of Assumption~\ref{ass:VW:intro} that $a''(\cdot)\le C_0$ yields the inequality  $\sqrt{a(x)}\le C_1[1+x]$ for some constant $C_1$. Hence, using also (d) of 
   Assumption~\ref{ass:VW:intro},  we have
  \[
    W(x) - W(0) = \int_0^x W'(y) \dx{y} \geq \int_0^x \frac{\sqrt{c_0}}{\sqrt{a(y)}} \dx{y} \geq \frac{\sqrt{c_0}}{C_1} \int_0^x \frac{\dx{y}}{1+y} =\frac{\sqrt{c_0}}{C_1}  \log\bra*{1+x} . 
  \]
  To illustrate the above set of assumptions, we take for $a,V,W$ power laws
  \[
    W(x)= (1+x)^{\kappa} \, \qquad a(x) = (1+x)^\alpha \qquad\text{and}\qquad  V(x) = (1+x)^\gamma .
  \]
  Then, the admissible range of exponents is
  \[
    0 < \kappa \leq 2 , \qquad \max\set{2-2\kappa,0} \leq \alpha \leq 2 -\kappa \qquad\text{and}\qquad 0 < \gamma < \min\set{2-\alpha, \kappa} .
  \]
\end{remark}
\begin{theorem}\label{thm:ExistUniqueConvergence}
Assume the functions $a(\cdot)$, $V(\cdot)$, $W(\cdot)$ satisfy Assumption~\ref{ass:VW:intro}.  Let $c(x,0), \ x>0,$ be a non-negative measurable function such that
\begin{equation}\label{AJ2}
\int_0^\infty W(x)c(x,0) \dx{x} \ < \ \infty \ .
\end{equation}
Then there exists a unique solution $c(\cdot,t), \ t>0,$ to the Cauchy problem \eqref{A1}, \eqref{B1}, \eqref{D1}, \eqref{E1} with initial condition $c(\cdot,0)$. \\[0.25\baselineskip]
For all $t>0$ the function $c(\cdot,t)\in \cC^1([0,\infty))$ and $\theta \in \cC^1([0,\infty))$.\\[0.25\baselineskip]
For any $L>0$ the solution $c(\cdot,t)$ converges uniformly on the interval $[0,L]$ as $t\ra\infty$ to the equilibrium $c_\theta^\eq(\cdot)$ with $\theta=\theta_\eq(\rho)$. If $\rho\le\rho_s$ then also
\begin{equation}\label{H1}
\lim_{t\ra\infty}\int_0^\infty W(x)|c(x,t)-c_\theta^\eq(x)| \dx{x} \ = \ 0 \ .
\end{equation}
\end{theorem}
The three statements on well-posedness, regularity and convergence to equilibrium are proven in the next three sections §\ref{s:ExistUnique}, §\ref{s:Reg} and §\ref{s:Conv}, respectively.
{  In  §\ref{s:ExistUnique}  we show that the solution $c(\cdot,t)$ exists in a weak sense (see Definition~\ref{def:weakSol}).  In particular,  the Borel measure $c(x,t) \dx{x}$ converges weakly as $t\ra 0$ to $c(x,0) \dx{x}$. Regularity properties of $c(x,t)$ are established in  §\ref{s:Reg}. For $t>0$ the function $c(\cdot,t)$ with domain $[0,\infty)$ is $\cC^1$ and the boundary condition \eqref{E1} holds.}

\subsection{Convergence rate to equilibrium (subcritical)}\label{s:intro:rate}

We derive in the subcritical case $\rho < \rho_s$ a quantified rate of convergence to equilibrium.
The proof relies on the entropy method and the convergence statement is shown with respect to a free energy, which is decreasing along the solution.
For the formal calculations with the free energy it is convenient to rewrite the set of equations \eqref{A1}, \eqref{B1}, \eqref{E1}. Observing from \eqref{B1}, \eqref{C1} that
\begin{equation*}
\pa_x[a(x)c(x,t)]-b(x,t)c(x,t) \ = \ a(x)c(x,t) \; \partial_x \log\frac{c(x,t)}{c_{\theta(t)}^\eq(x)} \  ,
\end{equation*}
we see that \eqref{A1}, \eqref{B1}, \eqref{E1} can be rewritten as
\begin{equation*}
  \partial_t c(x,t) = \partial_x\bra*{ a(x) c(x,t) \; \partial_x \log\frac{c(x,t)}{c_{\theta(t)}^\eq(x)}} \quad\text{with b.c.}\quad \log \frac{c(0,t)}{c_{\theta(t)}^\eq(0)} = 0 \ .
\end{equation*}
With the PDE and boundary condition in this form, together with equation~\eqref{D1} for $\theta(t)$,  the following energy dissipation estimate is formally deduced: 
\begin{equation}\label{e:EDI}
 \frac{\dx{}^+}{\dx{t}} \cG\bra*{c(\cdot,t), \theta(t) } \leq - \cD\bra*{ c(t),\theta(t) }
\end{equation}
where $\frac{\dx{}^+}{\dx{t}} f(t) = \limsup_{\delta \to 0^+} \frac{f(t+\delta)- f(t)}{\delta}$ and
\begin{align}
  \cG(c(\cdot,t),\theta(t)) &= \int \bra*{ \log c - 1} c(x) \dx{x} + \int \bra{V + \log a} c(x) \dx{x} + \frac{1}{2} \theta(t)^2 \label{e:FreeEnergy:FP}\\
  &\qquad\text{with}\quad \theta(t) = \rho - \int W(x) c(x,t) \dx{x} \notag \\
  \cD(c(\cdot),\theta) &= \int a(x) \bra*{ \partial_x \log \frac{c}{c_{\theta}^\eq}}^2 c(x) \dx{x} .\notag
\end{align}
{ The differential inequality \eqref{e:EDI} is rigorously established after proving sufficient regularity properties of the solution $c(\cdot,t)$.}
\begin{remark}[Relation to McKean-Vlasov dynamic]
  The term $\frac{1}{2} \theta^2$ in~\eqref{e:FreeEnergy:FP} is characteristic for free energies of McKean-Vlasov equations
  \begin{equation}\label{FreeEnergy:McKeanVlasov}
    \cF_{MV}(c) = \int c \log c + \int c \tilde V + \frac{1}{2} \iint K(x,y) c(x)c(y) \, \dx{x} \dx{y} ,
  \end{equation}
  with some kernel function $K$. For the product kernel $K(x,y)=W(x) W(y)$, the last term becomes $\bra{\int W c}^2$. By choosing $\tilde V =  V -1 + \log a - \rho W$, the free energy $\cF_{MV}$ agrees with $\cG$ from \eqref{e:FreeEnergy:FP} up to a constant. The connection becomes more apparent by noting that \eqref{A1}, \eqref{B1}, \eqref{D1}, \eqref{E1} is the formal gradient flow with respect to a Wasserstein metric including a boundary condition (see §\ref{s:GF:FP}). Hence, the presented Fokker-Planck equation has a close connection to the class of McKean-Vlasov equations with a product kernel, however with a non-local boundary condition.
  
  Let us point out that gradient flows with boundary condition are quite delicate and are first studied in~\cite{Figalli10} for the heat equation with Dirichlet boundary conditions. McKean-Vlasov equations with non-local interaction and some boundary are just recently studied in~\cite{Morales2016}. However, the particular boundary condition~\eqref{E1} together with the non-local constraint~\eqref{D1} has been to our knowledge not studied in the literature so far.
\end{remark}
The function $\cG$ is proven to be convex with a unique minimizer satisfying the constraint~\eqref{D1} (see Lemma~\ref{lem:CharactMinEnergy})
given by
\begin{equation*}
  \inf_c\set*{ \cG(c,\theta) : \theta + \int W(x) c(x) \dx{x} = \rho } = \cG\bra{c_{\theta_\eq}^\eq,\theta_\eq}
\end{equation*}
where $\theta_\eq = \theta_\eq(\rho)$ is uniquely determined by $\rho$ through the identity $\rho = \theta_\eq + \int W c_{\theta_\eq}^\eq$. This allows to define the normalized free energy functional
\begin{equation}\label{e:def:FreeEnergyNorm}
  \cF_\rho(c) = \cG(c,\theta) - \cG\bra{ c_{\theta_\eq}^\eq,\theta_\eq} \qquad\text{with}\qquad \theta = \rho - \int W(x) c(x) \dx{x} .
\end{equation}
Therewith, we can state the second main result on the rate of convergence to equilibrium.
  \begin{theorem}\label{thm:QuantLongTime}
  Let $\rho<\rho_s$. Assume the function $a(\cdot), V(\cdot), \ W(\cdot)$ to satisfy Assumption~\ref{ass:VW} and in addition for some $\beta \in (0,1]$ and constants $0<c_0<C_0<\infty$ holds
  the refinement of~\eqref{G1:intro}
  \begin{equation}\label{G1:alpha}
    c_0 W^{1-\beta}(x) \leq a(x) W'(x)^2 \qquad\text{for } x\in\R^+ .
  \end{equation}
  Let~$c$ be a solution to \eqref{A1}, \eqref{B1}, \eqref{D1}, \eqref{E1} with initial condition $c(\cdot,0)$ satisfying~\eqref{AJ2} and for some $C_0$ and $k > 0$ the moment condition
  \begin{equation}\label{e:quantLT:initial}
     \int W(x)^{1+ k \beta}  \, c(x,0) \dx{x} \leq C_0,
  \end{equation}
  Then there exists $\lambda$ and $C$ depending on $a,V,W,\theta_\eq,C_0,k$ such that for all $t\geq 0$
  \begin{equation*}
    \cF_\rho(c(t)) \leq \frac{1}{\bra{ C + \lambda t}^{k}} \ .
  \end{equation*}
  Moreover, if~\eqref{G1:alpha} holds with $\beta=0$, that is $c_0 W(x) \leq a(x) W'(x)^2 \leq C_0 W(x)$ for $x\in \R^+$, then there exists $C>0$ and $\lambda>0$ such that
 \begin{equation*}
    \cF_\rho(c(t)) \leq C e^{-\lambda t} .
  \end{equation*}
\end{theorem}
The proof of Theorem~\ref{thm:QuantLongTime} is contained in §\ref{s:Rate}.

By a suitable Pinsker inequality (see Corollary~\ref{cor:FreeEnergyRelEnt}), the convergence of Theorem~\ref{thm:QuantLongTime} also implies the quantified version of the statement~\eqref{H1} of Theorem~\ref{thm:GlobalExistence} as well as a quantified convergence statement for $\theta(t)$.
\begin{corollary}\label{cor:Pinsker}
  Under the assumptions of Theorem~\ref{thm:QuantLongTime} there exists for any $T>0$ an explicit constant $C>0$ such that
  \begin{equation}
    \bra*{\int W(x) \abs*{c(x,t)- c_{\theta^\eq}^\eq(x)} \dx{x}}^2 + \bra*{ \theta(t) - \theta^\eq}^2 \leq C \;\cF_\rho(c(\cdot,t)) \qquad\text{ for all } t\geq T.
  \end{equation}
\end{corollary}
\begin{remark}
  Firstly, let us emphasize, that the rates given in Remark~\ref{rem:ass} satisfy the refined assumption~\eqref{G1:alpha} with $\beta = \frac{2-\alpha -\kappa}{\kappa} \in [0,1]$.

  To explain the quantity~\eqref{e:quantLT:initial}, we introduce $\omega(x) = W(x) / \bra*{a(x) W'(x)^2}$. Then, by using~\eqref{G1:alpha}, we have $\omega(x)\leq c_0^{-1} W(x)^{\beta}$ and the moment condition~\eqref{e:quantLT:initial} gives a bound on
  \[
    \int \omega(x)^k \, W(x) \, c(x,0) \dx{x} \leq C_0 .
  \]
  The weight function $\omega(x)$ is essential for the derivation of suitable functional inequalities, in this case weighted logarithmic Sobolev inequalities (see §\ref{s:weightLSI}). Together with an interpolation argument, this is the essential ingredient to obtain a suitable differential inequality for the time-derivative of~$\cF_\rho$. We decided for the sake of presentation to use the slightly simplifying assumption~\eqref{G1:alpha}.

  In addition, it is possible by the same technique, to obtain subexponential convergence rates of the form $\exp\bra*{-\lambda t^{\nu}}$ for some $\nu\in (0,1)$ by using an interpolation argument with stretched exponential moments and suitable bounds on the initial data. We refer to~\cite{CEL15}, where this kind of result is obtained for the Becker-Döring equation.
\end{remark}

\subsection{Relation with the classical Becker-Döring model}
 The system \eqref{A1}, \eqref{B1}, \eqref{D1}, \eqref{E1} may be considered  a continuum analogue of the classical Becker-D\"{o}ring model for cluster evolution~\cite{BD1935}.  Letting $c_\ell(t), \ \ell=1,2,...,$ denote the density at time $t$ of clusters of volume $\ell$, the
cluster evolution is  determined by the equations
\begin{equation}\label{P1}
\pderiv{c_\ell(t)}{t} \ = \ J_{\ell-1}(t)-J_\ell(t) \ , \quad \ell=2,3,...
\end{equation}
 and the  mass conservation law
\begin{equation}\label{Q1}
\sum_{\ell=1}^\infty \ell c_\ell(t) \ = \ \rho  \ .
\end{equation}
It follows from \eqref{P1}, \eqref{Q1} that monomer evolution is determined by the equation
 \begin{equation}\label{R1}
 \pderiv{c_1(t)}{t} \ = \ -J_1-\sum_{\ell=1}^\infty J_\ell \ .
 \end{equation}
 The flux $J_\ell$ is the rate at which clusters of volume $\ell$ become clusters of volume $\ell+1$, and is given by the formula
 \begin{equation}\label{S1}
 J_\ell \ = \ a_\ell c_1c_\ell-b_{\ell+1}c_{\ell+1}  \ ,
 \end{equation}
 where $a_\ell,\ b_\ell>0$ are given functions of $\ell$.  Thus an $\ell$-cluster can combine with a monomer at rate $a_\ell$ to become an $\ell+1$-cluster.  An $\ell+1$-cluster can evaporate a monomer at rate $b_{\ell+1}$ to become an $\ell$-cluster.   
 One looks for equilibrium solutions to \eqref{P1}, \eqref{Q1} by solving the equations $J_\ell=0,  \ \ell=1,2,...$ From \eqref{S1} it follows that $c_{\ell+1}=a_\ell c_1c_\ell/b_{\ell+1},  \ \ell=1,2,..$. Hence $c_\ell=Q_\ell c_1^\ell, \ \ell=1,2,..$ is an equilibrium solution where
 \begin{equation}\label{T1}
 Q_\ell \ = \ \prod_{r=1}^{\ell-1} \frac{a_r}{b_{r+1}} \ , \qquad Q_1=1,
 \end{equation}
 provided \eqref{Q1} holds. 
 Making the assumption that there exists $z_s>0$ such that
 \begin{equation}\label{U1}
 \lim_{\ell\ra\infty}\frac{a_\ell}{b_\ell} \ = \ \frac{1}{z_s} \ ,
 \end{equation}
 it is clear there is a family of equilibria for every density $\rho$ with $0<\rho< \rho_s$, where
 \begin{equation}\label{V1}
 \rho_s \ =  \ \sum_{\ell=1}^\infty Q_\ell z_s^\ell \ .
 \end{equation}
If $\rho_s<\infty$ then the equilibrium corresponding to $\rho=\rho_s$ is critical: there are no equilibria with~$\rho>\rho_s$.

Global existence and uniqueness theorems for the Becker-Döring system \eqref{P1}, \eqref{Q1} were proven in the seminal paper of Ball, Carr and Penrose \cite{bcp} under fairly mild assumptions on the rates $a_\ell,b_\ell,\  \ell=1,2,..,$ the main ones being  that $a_\ell,b_\ell$ should grow at sub-linear rates in $\ell$  as  $\ell\ra\infty$.
It was also shown (Theorem~5.6 of \cite{bcp}) under stronger assumptions on $a_\ell,b_\ell$, in particular \eqref{U1},   that if $\rho\le \rho_s<\infty$, then the solution of \eqref{P1}, \eqref{Q1} converges strongly at large time to the
corresponding equilibrium solution, in the sense that
\begin{equation}\label{W1}
\lim_{t\ra\infty} \sum_{\ell=1}^\infty \ell|c_\ell(t)-Q_\ell c_1^\ell | \ = \ 0 \ .
\end{equation}
If $\rho>\rho_s$ then $c_\ell(t)$  converges weakly to the equilibrium solution with maximal
density,
\begin{equation}\label{X1}
\lim_{t\ra\infty}c_\ell(t) \ = \ Q_\ell z_s^\ell \ , \quad \ell=1,2,....
\end{equation}
Our first main result Theorem~\ref{thm:ExistUniqueConvergence} can be seen as the analog of the these results.
A key tool in the proof of Theorem 5.6 of \cite{bcp} is a Lyapunov function $\cG$ defined by
 \begin{equation}\label{Y1}
 \cG(c(\cdot)) \ =  \ \sum_{\ell=1}^\infty c_\ell\set*{\log\frac{c_\ell}{Q_\ell}-1}  \ ,
 \end{equation}
which is decreasing on solutions to \eqref{P1}, \eqref{Q1}.

To make the connection between the Becker-Döring model and the model \eqref{A1}, \eqref{B1}, \eqref{D1}, \eqref{E1} we  formulate,  following Vel\'{a}zquez \cite{vel}, equation \eqref{P1} as a discretization of a diffusion equation.  From \eqref{S1} we have that
 \begin{equation}\label{Z1}
 J_\ell \ = \ [b_\ell c_\ell-b_{\ell+1}c_{\ell+1}]+\{(a_\ell z_s-b_\ell)+\theta(t)a_\ell\}c_\ell  \ ,
 \end{equation}
 where the function $\theta(t)$ is given by
 \begin{equation}\label{AA1}
 c_1(t) \ = \ z_s+\theta(t) \  .
 \end{equation}
Setting $c(\ell,t)=c_\ell(t), \ \ell=1,2,\dots,$ and denoting by $D$ the forward difference operator with adjoint~$D^*$, we have from \eqref{P1} that
\begin{equation}\label{AB1}
\pa_t c(\ell,t) \ = \ -D^*D[b_\ell c(\ell,t)]+D^*[\{(a_\ell z_s-b_\ell)+\theta(t)a_\ell\}c(\ell,t)] \ , \quad \ell\ge 2.
\end{equation}
We also have from \eqref{AA1} the boundary condition
\begin{equation}\label{AC1}
c(1,t) \ = \ z_s+\theta(t) \ .
\end{equation}
The dynamics is now completed from the conservation law \eqref{Q1}. Evidently \eqref{AB1}, \eqref{AC1} can be considered a discretization on the positive integers of \eqref{A1}, \eqref{B1}, \eqref{E1}. In §\ref{s:Interpol} we describe for $\ve>0$ a family of models on $\ve\Z^+$, where $\Z^+$ denotes the non-negative integers. The Becker-Döring model is a member of this family corresponding to $\ve=1$, and the model \eqref{A1}, \eqref{B1}, \eqref{D1}, \eqref{E1} corresponds to the limiting case $\ve\ra0$. For each model there is a Lyapunov function analogous to \eqref{Y1}, which decreases  on solutions.  In the case of \eqref{A1}, \eqref{B1}, \eqref{D1}, \eqref{E1}  the Lyapunov function is consistent with~\eqref{e:FreeEnergy:FP}. Moreover, each model possesses a gradient flow structure for which the Lyapunov function is the driving free energy with respect to a Wasserstein metric with Dirichlet boundary condition.

An important example of $a_\ell, \ b_\ell$ is
 \begin{equation}\label{AE1}
 a_\ell \ = \ a_1\ell^{\alpha} \ , \quad b_{\ell} \ = \ a_\ell(z_s+q/\ell^{\gamma}) \ , \ \ell=1,2,... ,
 \end{equation}
 where $\alpha\in [0,1]$, $\gamma\in (0,1)$ and $q, z_s>0$. The classical theory of coarsening is obtained by the choice $\alpha= \gamma = \frac{1}{3}$. We refer to \cite{N03} for some heuristic derivation of these more general rates and their physical interpretation.
 In that case we have an asymptotic formula for $Q_\ell$ at large $\ell$ given by 
 \begin{equation}\label{AF1}
 Q_\ell \  \simeq \ z_s^{-(\ell-1)} \frac{1}{\ell^{\alpha}}\exp\pra*{-\frac{q}{z_s(1-\gamma)}\ell^{1-\gamma}} \ , \quad \ell\ra\infty \ .
 \end{equation}
 Hence there are equilibria for every density $\rho$ with $0<\rho\le \rho_s$. Using \eqref{AB1} to make a comparison between the Becker-Döring model with parameter values given by \eqref{AE1} and \eqref{A1}, \eqref{B1}, we formally obtain  by using $l(x) = 1+x$ that $a(x) = b_{1+x} \sim (1+x)^\alpha$, where $\sim$ neglects lower order terms for $x\gg 1$. Similarly, we have
 \begin{align*}
   W'(x) = \frac{a_{1+x}}{b_{1+x}} \sim \frac{1}{z_s} \qquad\text{and}\qquad  V'(x) = \frac{b_{1+x} - z_s a_{1+x}}{b_{1+x}} \sim \frac{q}{z_s} (1+x)^{-\gamma} . 
 \end{align*}
 Hence, we see that it is appropriate to set up to numerical constants
 \begin{equation}\label{AG1}
 a(x) \ = \ (1+x)^{\alpha} \ ,  \quad V(x) \ = (1+x)^{1-\gamma} \ ,  \quad  W(x) \ = \ 1+x \ ,
 \end{equation}
in \eqref{A1}, \eqref{B1}. Note that in view of Remark~\ref{rem:ass} the functions satisfy Assumption~\ref{ass:VW:intro} and convergence to equilibrium is obtained through Theorem~\ref{thm:QuantLongTime}. In particular the case $\alpha=1$ yields exponential convergence to equilibrium in analogy to the result of~\cite{CEL15} for the Becker-Döring equation.

\subsection{Transformation to \texorpdfstring{$a\equiv 1$}{a=1}}

For the rest of the paper and the sake of presentation, we are going to transform the equation by a change of variable to a Fokker-Planck equation of the same type but with constant diffusion constant $a\equiv 1$. We achieve this by making the change of variable
\begin{equation}\label{J1}
z(x)=\int_0^x \frac{\dx{x'}}{\sqrt{a(x')}} \qquad\text{and} \qquad \tilde{c}(z(x),t) \ = \ \sqrt{a(x)}c(x,t) \ .
\end{equation}
{ We have already observed that Assumption~\ref{ass:VW:intro}~(a) implies that $\sqrt{a(x)}\le C_1[1+x]$ for some constant $C_1$, whence $z(x) = \int_0^x \dx{x'} /\sqrt{a(x')} \to \infty$ as $x\to \infty$.} Moreover, from Assumption~\ref{ass:VW:intro}~(b) it follows that $0<z'(x) < \infty$ and thus $z : \R^+ \to \R^+$ is one-to-one and we write $x(z)$ for the inverse. Then \eqref{A1} becomes
\begin{multline} \label{K1}
\partial_t \tilde{c}(z,t) +\partial_z \pra[\big]{\tilde{b}(z,t)\tilde{c}(z,t)} \ = \ \partial_z^2\tilde{c}(z,t) \ , \quad 0<z<\infty, \ t>0, \\
{\rm where \ } \tilde{b}(z,t) \ = \ \frac{b(x,t)}{\sqrt{a(x)}}  -\frac{a'(x)}{2\sqrt{a(x)}} \ , \quad x=x(z) \ .
\end{multline}
Note that \eqref{B1} implies that
\begin{align} 
\tilde{b}(z,t) &= \theta(t)\tilde{W}'(z)-\tilde{V}'(z) \notag \\
&\quad \text{where}\quad \tilde{W}(z)= W(x(z)) \ \text{ and } \  \tilde{V}(z)= V(x(z))+\frac{1}{2}\log[ a(x(z))] \ . \label{L1}
\end{align}
Evidently the conservation law \eqref{D1} becomes
 \begin{equation}\label{M1}
\theta(t)+\int_{0}^\infty \tilde{W}(z)\tilde{c}(z,t) \dx{z} \ = \ \rho \ ,
\end{equation}
and the boundary condition \eqref{E1} becomes
 \begin{equation}\label{O1}
 \tilde{c}(0,t) \ =  \ \exp\bra*{-\tilde{V}(0)+\theta(t)\tilde{W}(0)} \ .
 \end{equation}
Moreover, it holds by writing shortly $x = x(z)$ and using that $\pderiv{x}{z} = \sqrt{a(x)}$
\begin{align}
 \pderiv{}{z} \tilde V(z) &= V'(x) \sqrt{a(x)} + \frac{a'(x)}{2\sqrt{a(x)}} \ ; \label{tildeVW:d1}  \\
 \pderiv[2]{}{z} \tilde V(z) &= V''(x) a(x) + \frac{1}{2} V'(x) a'(x) + \frac{a''(x)}{2} - \frac{\bra*{ a'(x)}^2}{4 a(x)} \ ;  \label{tildeVW:d2}  \\
 \pderiv{}{z}\tilde W(z) &= W'(x) \sqrt{a(x)} \ ; \label{tildeVW:d3} \\
 \pderiv[2]{}{z} \tilde W(z) &= W''(x) a(x) + \frac{1}{2} W'(x) a'(x) \ .  \label{tildeVW:d4} 
\end{align}
In particular, the functions $\tilde V$ and $\tilde W$ satisfy again Assumption~\ref{ass:VW:intro} with $a(\cdot) \equiv 1$ (see Lemma~\ref{lem:Ass:intro}). From now on, we drop the tilde-superscript, write again $x$ for $z$ and assume $a\equiv 1$ and $V, W$ to satisfy Assumption~\ref{ass:VW:intro}, which then take the following simple form:
\begin{assumption}[$a\equiv 1$]\label{ass:VW}\
\begin{enumerate}[ (a) ]
 \item $V(\cdot), W(\cdot) \in \cC^2(\R^+)$ satisfy satisfy for some $C_0>0$
 \begin{equation}\label{F1:0}
    \abs{V''(x)} + \abs{W''(x)}\leq C_0 \qquad\text{ for } x\in \R^+ \  .
  \end{equation}
 \item $W(\cdot)$ is an increasing function with $W(0)>0$.
 \item For any $\del>0$ there exists $x_\del>0$ such that
 \begin{equation}\label{F1:1}
   |V'(x)|  \ \le \ \del W'(x) \ ,  \quad\text{if } x\ge x_\del \ .
 \end{equation}
and
\begin{equation}\label{F1:2}
  \abs{V''(x)} + \abs{W''(x)} \ \le \ \delta W'(x)^2 \ ,   \quad\text{if } x\ge x_\del \ .
\end{equation}
  \item There exists $C_0,c_0>0$ such that
\begin{equation}\label{G1}
c_0 \leq W'(x)^2 \leq C_0 W(x) \ ,  \qquad \text{ for } x\in \R^+ \ .
\end{equation}
  \item The function  $W(\cdot) \exp\bra*{-V(\cdot)}$ is integrable on $(0,\infty)$.
\end{enumerate}
\end{assumption}

\bigskip

\section{Existence and uniqueness Theorems}\label{s:ExistUnique}

\subsection{Notion of solutions, existence and uniqueness results}

Now, we are going to prove that there exists a solution globally in time to \eqref{A1}, \eqref{B1}, \eqref{D1}, \eqref{E1} under Assumption~\ref{ass:VW} on the functions $V(\cdot),W(\cdot)$ as well as a $W$-moment of the initial data~\eqref{AJ2}.  In order to carry this out it will be helpful for us to consider solutions to the PDE adjoint to \eqref{A1},
 \begin{equation}\label{A2}
\mathcal{L}_{x,t} w (x,t) = \pa_t w(x,t) +  \pa^2_x w(x,t) +b(x,t)\;\pa_x w(x,t) = 0 \ .
\end{equation}
Note that (\ref{A2}) is to be solved \emph{backwards in time}, and is therefore parabolic (see  page 26 of \cite{fried}). 
We shall be interested in solutions $w(x,t)$ to \eqref{A2} in domains $\{(x,t) \ : \ x>0, \ t<T\}$ for $T>0$  with zero Dirichlet boundary condition at $x=0$ and terminal condition at $t=T$. That is
\begin{equation}\label{B2}
w(x,T) \ = \ w_0(x), \  x> 0, \quad w(0,t) \ = \ 0, \ t<T,
\end{equation}
where $w_0(\cdot)$ is a given function.  We assume that 
\begin{equation} \label{C2}
\begin{split}
 b:[0,\infty)\times [0,T]\ra\R, \quad &\frac{\pa b}{\pa x}:[0,\infty)\times [0,T]\ra\R \ \   \text{ are continuous }, \\
  &{\rm and } \  \sup_{x\ge 0, \ 0\le t\le T} \frac{|b(x,t)|}{1+x}<\infty \ .
\end{split}
\end{equation}
Then the terminal-boundary value problem \eqref{A2}, \eqref{B2}   has a unique classical solution in $[0,T]\times [0,\infty)$  provided $w_0(\cdot)$ is a continuous function satisfying 
\[ \sup_{x\ge 0} |w_0(x)|\exp(-Ax)<\infty \quad\text{ for some $A\ge 0$.}  \]
That is there is a unique function $w:(0,\infty)\times (0,T)\ra\R$ such that $w(x,t)$ is $\cC^2$ in $x,  \ \cC^1$ in~$t$, denoted by $w\in \cC_{t,x}^{2,1}(\R^+ \times [0,T])$, and satisfies the PDE \eqref{A2}.  In addition, $w$ extends to a continuous function  on the set $\{(x,t)\in [0,\infty)\times [0,T] \ : \ (x,t)\ne (0,T) \ \}$ and satisfies \eqref{B2}.  This follows from the Green's function estimates in \cite[Lemma 3.4]{cg}. In this section we shall always assume that \eqref{C2} holds.

We shall be concerned mostly with solutions to \eqref{A2} where the drift $b(\cdot,\cdot)$ is given by \eqref{B1} with $a(\cdot)\equiv 1$, that is
\begin{equation}\label{A*2}
b(x,t) \ = \ \theta(t)W'(x)-V'(x) \ , \quad x\ge 0 \ .
\end{equation}
In order for the condition \eqref{C2} to hold we shall need $W(\cdot),V(\cdot)$ to be $\cC^2$ on $[0,\infty)$ with bounded second derivative, which is~\eqref{F1:0} of Assumption~\ref{ass:VW}, and the function $\theta(\cdot)$ to be continuous.

Suppose now that the continuous function $\theta:[0,T]\ra\R$ is given. If the function $c(x,t), \ x>0,0<t<T,$ is  a classical solution to \eqref{A1}, \eqref{E1}, and $w(x,t), \ x>0,t<T,$ a classical  solution to \eqref{A2}, \eqref{B2} then we have that
{
\begin{equation}\label{AH2}
\pderiv{}{t} \int_0^\infty w(x,t)c(x,t) \dx{x} \ = \  \pa_x w(0,t) \, \exp\pra*{-V(0)+\theta(t)W(0)}  \ .
\end{equation}
}
We use an integrated form of the identity \eqref{AH2} (see~\eqref{AI2} below) to construct a unique measure valued weak solution to the initial-boundary value problem \eqref{A1}, \eqref{E1}. We recall some well known properties of spaces of measures \cite{ bil,prok}. Let $\mathcal{B}(\R^+)$ be the $\sig-$algebra of Borel sets on $\R^+=[0,\infty)$, and  $\mathcal{M}(\R^+)$ be the Banach space of real finite Borel measures $\mu:\mathcal{B}(\R^+)\ra\R$. Let $\cC_0(\R^+)$ be the Banach space of continuous functions $f:[0,\infty)\ra\R$ which  vanish at $\infty$  with norm $\|f\|=\sup_{x\in\R^+}|f(x)|$. Then by the Riesz representation theorem  $\mathcal{M}(\R^+)$ is the dual of $\cC_0(\R^+)$.
\begin{definition}\label{def:weakSol}
{
  Let $T_0>0$. A family $[0,T_0]\ni T \mapsto c(\cdot,T) \in \cM(\R^+)$ is a solution to \eqref{A1}, \eqref{E1} with~$b$ satisfying~\eqref{A*2} and $\theta \in \cC([0,T_0],\R^+)$ if for all $T\in (0,T_0)$ and all solutions $w\in \cC_{x,t}^{2,1}(\R^+\times [0,T])$ } to the adjoint equation~\eqref{A2}, \eqref{B2} with $w_0\in  \cC_0(\R^+)$  the following identity holds: 
  \begin{equation}\label{AI2}
\int_0^\infty w_0(x) c(x,T) \dx{x} \ = \   \int_0^\infty w(x,0)c(x,0) \dx{x} + \int_0^T \pa_x w(0,t)\, e^{-V(0)+\theta(t)W(0)} \dx{t}  \ .
\end{equation}
\end{definition}
If the initial data $c(x,0) \  dx$ is a finite Borel measure then it follows from the maximum principle for solutions to ~\eqref{A2}, \eqref{B2} that the first integral on the RHS of \eqref{AI2} is bounded by a constant times $\|w_0\|_\infty$.  The second integral on the RHS of \eqref{AI2}, which incorporates the boundary condition \eqref{E1}, is also bounded   by a constant times $\|w_0\|_\infty$.  This follows again from the maximum principle and Lemma~\ref{lem:uniformFluxBound} below. Note that if $w_0(\cdot)$ is non-negative then the maximum principle implies that $w(\cdot,t)$ is non-negative for $t<T$. In particular we have that $\pa_x w(0,t)\ge 0$. It follows now from \eqref{AI2} that, if the measure $c(x,0) \ dx$ is non-negative,  $c(x,T) \ dx$ is also non-negative. 

We prove in the sense of Definition~\ref{def:weakSol} a local existence and uniqueness theorem for the nonlinear system \eqref{A1}, \eqref{D1},  \eqref{E1} by using a fixed point argument.
\begin{proposition}\label{prop:ShortTimeExistence}
Suppose that the functions $V(\cdot),W(\cdot)$ satisfy Assumption~\ref{ass:VW} and that the initial data $c(x,0)\ge 0, \ x\ge 0$, satisfies~\eqref{AJ2} with $\theta(0)=\theta_0$ so that \eqref{D1} holds at $t=0$. Then there exists a unique solution to \eqref{A1}, \eqref{B1}, \eqref{D1}, \eqref{E1} in some interval $0\le t\le T$, where the time period $T$ is bounded below by a constant depending only on $\theta_0$.
\end{proposition}
Next, a global existence and uniqueness of solutions to \eqref{A1}, \eqref{B1}, \eqref{D1}, \eqref{E1} for non-negative initial data $c(x,0), \ x>0,$ satisfying \eqref{AJ2} is obtained by a further control on the time period $T$. This requires certain uniform in time estimates for the adjoint problem controlling both terms of the RHS of~\eqref{AI2}.
\begin{theorem}\label{thm:GlobalExistence}
Suppose Assumption~\ref{ass:VW} holds and $c(0,\cdot)$ satisfies~\eqref{AJ2}, then the local in time solution to \eqref{A1}, \eqref{B1}, \eqref{D1}, \eqref{E1} constructed in Proposition~\ref{prop:ShortTimeExistence} can be extended to all times~$T>0$.
\end{theorem}

\subsection{A priori estimates for local existence and proof of Proposition~\ref{prop:ShortTimeExistence}}

\begin{lemma}\label{lem:Adjoint:Estimate}
Assume that $W:[0,\infty)\ra\R$ is a $\cC^2$ increasing function which satisfies $W(0)>0$ and  $\sup_{x\ge 0} W(x)\exp(-Ax)<\infty$ for some $A\ge 0$. In addition assume that $V:[0,\infty)\ra\R$ is $\cC^2$ and there exists  $x_0,C_0>0$ such that
\begin{equation}\label{D2}
W'(x)^2+|W''(x)| \ \le \ C_0W(x),  \quad |V'(x)| \ \le \ C_0W'(x) \quad\text{for } x\ge x_0 \ .
\end{equation}

Let $\theta:[0,T]\ra\R$ be a continuous function and  $w(x,t), \ x>0,t<T,$ be the solution to \eqref{A2},  \eqref{B2}, \eqref{A*2} with $w_0(x)=W(x), \ x>0$. Then for any $T,\theta_\infty>0$ and $\theta(\cdot)$ satisfying  $\|\theta(\cdot)\|_\infty\le \theta_\infty$, there is a constant $C(T,\theta_\infty)$ such that $w(x,0)\le C(T,\theta_\infty)W(x)$ for all $x>0$.  The constant $C(T,\theta_\infty)$ satisfies $\lim_{T\ra0} C(T,\theta_\infty)=1$ for any $\theta_\infty>0$.
\end{lemma} 
\begin{proof}
For any $\la\in\R$ let $w_\la(x,t)=e^{\la(T-t)}W(x), \ x>0,t<T$. Letting $\mathcal{L}_{x,t}$ denote the partial differential operator of \eqref{A2}, with $b$ given by \eqref{A*2} then
\begin{equation}\label{E2}
e^{-\la(T-t)}\mathcal{L}_{x,t} w_\la(x,t) \ = \  -\la W(x)+W''(x)+[\theta(t)W'(x)-V'(x)]W'(x) \  .
\end{equation}
We see from \eqref{D2}, \eqref{E2} that $\mathcal{L}_{x,t} w_\la(x,t)\le 0, \ x>0,t<T,$ provided $\la$ is taken sufficiently large independently of $T$.  In addition $w_\la(x,T)=W(x), \ x\ge 0,$ and $w_\la(0,t)>0, \ t<T$. We conclude from the maximum principle that $w(x,t)\le w_\la(x,t), \ x>0,t<T$.
\end{proof}
\begin{lemma}\label{lem:uniformFluxBound}
Assume that $b(\cdot,\cdot)$ satisfies \eqref{C2} and $x_1>0$.  Let $w(x,t), \ 0<x<x_1, \ 0\le t<T,$ be the solution to \eqref{A2} with terminal data $w(x,T)=w_0(x), \ 0<x<x_1,$ and Dirichlet boundary conditions $w(0,t)=0, w(x_1,t)=w_1(t), \ t< T$, where $w_0:[0,x_1]\ra\R$ is $\cC^1$ and $w_1:[0,T]\ra\R$ are continuous.  Setting $\|b(\cdot,\cdot)\|_\infty=\sup\{|b(x,t)| \ : \ 0<x<x_1, \ \max[0,T-1]<t<T\}$, and
$\|\pa b(\cdot,\cdot)\|_\infty=\sup\{|\pa b(x,t)/\pa x| \ : \ 0<x<x_1, \ \max[0,T-1]<t<T\}$, then  there exists a constant $C>0$ depending only on $x_1$ and
$\|b(\cdot,\cdot)\|_\infty, \ \|\pa b(\cdot,\cdot)\|_\infty$, such that
\begin{equation}\label{F*2}
\abs*{\pa_x w(x,t)} \ \le \ \frac{C}{\sqrt{T-t}}\pra[\big]{ \|w_0(\cdot)\|_\infty+(T-t)\|w_1(\cdot)\|_\infty } \ ,
\end{equation}
for $0<x<x_1/2 , \ \ \max[0,T-1]<t<T$. In the case when $x=x_1/2$ there is the stronger inequality, for all $\max[0,T-1]<t<T$ it holds
\begin{equation}\label{F2}
\abs[\big]{ \pa_x w(x_1/2,t)} \ \le \ C\pra[\big]{\|w'_0(\cdot)\|_\infty+\|w_0(\cdot)\|_\infty+(T-t)\|w_1(\cdot)\|_\infty} \ ,
\end{equation}
where $C$ depends only on $x_1, \ \|b(\cdot,\cdot)\|_\infty, \ \|\pa b(\cdot,\cdot)\|_\infty$. 
\end{lemma}
\begin{proof}
We follow the standard perturbative approach pioneered by Schauder \cite{fried}. We consider the terminal value problem for \eqref{A2}  in the interval $0<x<x_1,t<T$ with terminal data and Dirichlet boundary conditions given by
\begin{equation}\label{I2}
w(x,T) \ = \ w_0(x), \ \ 0<x<x_1, \quad w(0,t) \ = \ w(x_1,t) \ = \ 0, \ \ t<T.
\end{equation}
The solution to \eqref{A2}, \eqref{I2} can be represented in terms of the Dirichlet Green's function $G_D$ as
\begin{equation}\label{J2}
w(x,t) \ = \ \int_0^{x_1} G_D(x,x',t,T) w_0(x') \dx{x'} \ .
\end{equation}
The function $w(x,t)$ in the statement of the lemma can also be represented in terms of $G_D$ as
\begin{equation}\label{K2}
w(x,t) \ = \ \ \int_0^{x_1} G_D(x,x',t,T) w_0(x') \dx{x'} - \int_t^T \frac{\pa G_D}{\pa x'}(x,x_1,t,s) w_1(s) \dx{s} \ .
\end{equation}

If $b(\cdot,\cdot)\equiv 0$ then $G_D$ can be obtained from the method of images~\cite[p.~84]{fried}. For $x'$ in the interval $[0,x_1]$ let $x'_0=x'$, and $x'_1=-x'$ be the reflection of $x'$ in the boundary $0$ of the interval.  More generally we denote by $x'_m, m=0,1,2,..,$ all multiple reflections of $x'$ in the boundaries $0,x_1$. Letting
\begin{equation}\label{L2}
G(x,t) \ = \ \frac{1}{\sqrt{4\pi t}}\exp\pra*{-\frac{x^2}{4t}} \ ,
\end{equation}
we define the function $K(x,x',t,T)$ by 
\begin{equation}\label{M2}
K(x,x',t,T) \ = \ \sum_{m=0}^\infty (-1)^{p(m)}G(x-x'_m,T-t) \ ,
\end{equation}
where $p(m)=0,1$, depending on the number of reflections to obtain $x'_m$.  It is easy to see that the series \eqref{M2} converges and that $G_D(x,x',t,T)=K(x,x',t,T)$ when $b(\cdot,\cdot)\equiv 0$.  In the case of nontrivial $b$ one can obtain $G_D$ by perturbation expansion.  Let $\mathcal{L}_{x,t}$ denote the operator on the LHS of \eqref{A2}, so \eqref{A2} is $\mathcal{L}_{x,t}w=0$. Then
\begin{align}
G_D(x,x',t,T) &= K(x,x',t,T) +\sum^\infty_{n=0} \; v_n(x,x',t,T),  \notag \\
v_n(x,x',t,T) &= \int^{T}_t\dx{s}\int^{x_1}_{0} \dx{x''} \  K(x,x'',t,s) g_n(x'', x', s, T),  \label{N2} \\
g_0(x,x',t,T) &= \mathcal{L}_{t,x} K(x,x',t,T), \notag \\
g_{n+1}(x,x',t,T) &= \int^{T}_t\dx{s}\int^{x_1}_{0} \dx{x''} \ \mathcal{L}_{t,x} K(x,x'',t,s) g_n(x'', x', s, T). \notag
\end{align}
One easily obtains from \eqref{N2} the estimate for $n=0,1,2,\dots$,
\begin{equation}\label{O2}
|g_n(x,x',t,T)| \ \le \  \frac{C^{n+1}}{\Ga((n+1)/2))}\|b(\cdot,\cdot)\|_\infty^{n+1}(T-t)^{n/2-1/2}G(x-x',2(T-t)) \ ,
\end{equation}
where $\Ga(\cdot)$ is the Gamma function and $C$ is a constant depending only on $x_1$. It follows from \eqref{N2}, \eqref{O2} that there are constants $C_1,C_2$ depending only on $x_1$ such that for $\max\set{0,T-1}<t<T$, $0<x,x'<x_1$ it holds
\begin{equation} \label{P2}
0 \ \le \ G_D(x,x',t,T) \ \le \ C_1\exp\pra*{C_2\|b(\cdot,\cdot)\|^2_\infty(T-t)}G(x-x',2(T-t))
\end{equation}
Similarly we have that
\begin{equation}\label{Q2}
\abs*{\frac{\pa G_D(x,x',t,T)}{\pa x} } \ \le \ \frac{C_1}{\sqrt{T-t}}\exp\pra*{C_2\|b(\cdot,\cdot)\|^2_\infty(T-t)}  \ G(x-x',2(T-t)) \  ,
\end{equation}
where $C_1,C_2$ depend only on $x_1$. Estimates on other derivatives of $G_D$ involve $\|\pa b(\cdot,\cdot)\|_\infty$ as well as
$\|b(\cdot,\cdot)\|_\infty$.  Following the argument of \cite[Lemma 3.4]{cg}, we  have that
\begin{multline} \label{R2}
\left|\frac{\pa G_D(x,x',t,T)}{\pa x'} \right|+\sqrt{T-t}\left|\frac{\pa^2 G_D(x,x',t,T)}{\pa x\pa x'} \right| \ \le \ \frac{C_1}{\sqrt{T-t}}  \ G(x-x',2(T-t))
  \\ \times \set*{ 1+\sqrt{T-t}\left[\|b(\cdot,\cdot)\|_\infty+\|\pa b(\cdot,\cdot)\|_\infty\sqrt{T-t}\right]
 \exp\pra*{C_2\|b(\cdot,\cdot)\|^2_\infty(T-t)}} \  \  ,
\end{multline}
where $C_1,C_2$ depend only on $x_1$.

The  inequality \eqref{F*2} evidently follows by differentiating \eqref{K2} with respect to $x$ and using the bounds \eqref{Q2}, \eqref{R2}.  To obtain the  inequality \eqref{F2}  it is sufficient to bound the derivative
\begin{equation}\label{S2}
\frac{\pa}{\pa x} \int_0^{x_1}K(x,x',t,T) w_0(x') \dx{x'} \
\end{equation}
by a constant. The reason for this is that  \eqref{R2} implies that the derivative of the second  integral in~\eqref{K2} is bounded by a constant, and  one also easily sees that the derivative of the higher terms in the perturbation series \eqref{N2}  for the first integral in \eqref{K2} are bounded by a constant.  To bound the integral \eqref{S2} we observe that the derivative $\pa/\pa x$ can be replaced by $\pa/\pa x'$ in the expansion~\eqref{M2}. Hence on integration by parts in the integral \eqref{S2} we obtain a uniform bound.
\end{proof}
\begin{lemma}\label{lem:PostiveFlux}
Assume that the functions $V,W$ are $\cC^2$ on $[0,\infty)$, satisfy $\|W''\|_\infty+\|V''\|_\infty<\infty$ and the conditions of Lemma~\ref{lem:Adjoint:Estimate}.

 Let $w(x,t), \ x>0,t<T,$ be the solution to \eqref{A2},  \eqref{B2}, \eqref{A*2} with $w_0(x)=W(x), \ x>0$. Then for any $x_1>0$  there is a constant $C$ depending only on $x_1, \|\theta(\cdot)\|_\infty$ such that
\begin{equation}\label{U2}
0 \ \le \ \frac{\pa w(x,t)}{\pa x} \ \le \  C\sqrt{W(x)} \quad\text{for } x>x_1/2, \ \max[0,T-1]<t<T \ .
\end{equation}
\end{lemma}
\begin{proof}
Setting $u(x,t)=\pa w(x,t)/\pa x$, we see from \eqref{A2} that
\begin{equation}\label{V2}
\frac{\pa u(x,t)}{\pa t} +  \frac{\pa^2 u(x,t)}{\pa x^2} +b(x,t)\frac{\pa u(x,t)}{\pa x} +\frac{\pa b(x,t)}{\pa x} u(x,t) =0 \ . 
\end{equation}
We argue now as in Lemma~\ref{lem:Adjoint:Estimate}. Thus let $w_\la(x,t)=e^{\la(T-t)}\sqrt{W(x)}$ for any $\la\in\R$, $x>0$, $t<T$. Letting $\mathcal{L}$ denote the partial differential operator of \eqref{V2}, then
\begin{multline} \label{Y2}
e^{-\la(T-t)}\sqrt{W(x)} \ \mathcal{L}w_\la(x,t) \ = \\
 \left[-\la+\theta(t)W''(x)-V''(x)\right] W(x) \\
 +\frac{1}{2}\left\{W''(x)-\frac{W'(x)^2}{2W(x)}+[\theta(t)W'(x)-V'(x)]W'(x)\right\} \ .
\end{multline}
It follows from \eqref{Y2}, in view of \eqref{D2} and the boundedness of the second derivatives of $V(\cdot),W(\cdot)$, that there exists $\la>0$ depending on $\|\theta(\cdot)\|_\infty$  such that $ \mathcal{L}w_\la(x,t) <0$ for $x>0,t<T$. We also have that $\mathcal{L}u(x,t) =0$ for $x>0,t<T$, and from \eqref{D2}  there is a constant $C_1$ such that  $0\le u(x,T)=W'(x)\le C_1\sqrt{W(x)}$ for $x>x_1/2$.  Also from Lemma~\ref{lem:Adjoint:Estimate}, Lemma~\ref{lem:uniformFluxBound} it follows that
the constant $C_1$ can be chosen sufficiently large so that $u(x_1/2,t)\le C_1w_\la(x,t)$ for $\max[0,T-1]<t<T$.   We conclude from  the maximum principle that $u(x,t)\le C_1w_\la(x,t)$ for $x>x_1/2, \ \max[0,T-1]<t<T$.  
\end{proof}
\begin{lemma}\label{lem:Adjoint:Comparison}
Assume that the functions $V(\cdot),W(\cdot)$ are $\cC^2$ on $[0,\infty)$,  satisfy the conditions of Lemma~\ref{lem:Adjoint:Estimate}, and also $\|W''(\cdot)\|_\infty+\|V''(\cdot)\|_\infty<\infty$. For $i=1,2$ let $\theta_i:[0,T]\ra\R$ be continuous functions and $w_i(x,t), \ x>0,t<T,$ the corresponding solutions to \eqref{A2}, \eqref{B2},\eqref{A*2} with $\theta(\cdot)=\theta_i(\cdot)$ and terminal data $W(\cdot)$. Then there is a constant $C$ depending on $\|\theta_1(\cdot)\|_\infty,\|\theta_2(\cdot)\|_\infty$ such that

\begin{equation}\label{Z2}
{ |w_2(x,t)-w_1(x,t)| \ \le \  C\|\theta_2(\cdot)-\theta_1(\cdot)\|_\infty } \sqrt{T-t}  \ W(x) \quad\text{for } x>0, 
\end{equation}
\begin{equation}\label{AQ2}
{ \left|\frac{\pa w_2(0,t)}{\pa x}-\frac{\pa w_1(0,t)}{\pa x}\right| \ \le \ C\|\theta_2(\cdot)-\theta_1(\cdot)\|_\infty  \ ,  }
\end{equation}
provided $\max[0,T-1]<t<T$.
\end{lemma}
\begin{proof}
For $0\le\mu\le 1$ let $w_\mu(x,t), \ x>0,t<T,$ be the solution to \eqref{A2}, \eqref{B2}, \eqref{A*2} with $\theta(\cdot)=(1-\mu)\theta_1(\cdot)+\mu\theta_2(\cdot)=\theta_\mu(\cdot)$ and terminal data $W(\cdot)$. Then
\begin{equation}\label{AB2}
{ w_2(x,t)-w_1(x,t)} \ = \  \int_0^1 \frac{\pa w_\mu(x,t)}{\pa \mu} \dx\mu \ = \  \int_0^1 v_\mu(x,t) \dx\mu \ ,
\end{equation}
where $v_\mu(x,t), \ x>0,t<T,$ is the solution to the PDE
 \begin{multline} \label{AC2}
\frac{\pa v_\mu(x,t)}{\pa t} +  \frac{\pa^2 v_\mu(x,t)}{\pa x^2} +b_\mu(x,t)\frac{\pa v_\mu(x,t)}{\pa x} \ = \ [\theta_1(t)-\theta_2(t)]\, W'(x)\frac{\pa w_\mu(x,t)}{\pa x} \ , \\
\text{for } 0<x<\infty, \ t<T,  \quad \text{and } b_\mu(x,t) \ = \ \theta_\mu(t)W'(x)-V'(x) \ .
\end{multline}
The function $v_\mu$ also satisfies the terminal and boundary conditions
\begin{equation}\label{AD2}
v_\mu(x,T) \ = \ 0, \ x\ge 0, \quad v_\mu(0,t)=0, \ t \le T. 
\end{equation}
The solution to \eqref{AC2}, \eqref{AD2} can be represented as
\begin{equation}\label{AE2}
v_\mu(x,t) \ = \ \int_t^T{ [\theta_2(s)-\theta_1(s)]}h_\mu(x,t;s) \dx{s} \ ,
\end{equation}
where $w(x,t)=h_\mu(x,t;s)$ is the solution to \eqref{A2}, \eqref{A*2} for $t<s,x>0$
with $\theta(\cdot)=\theta_\mu(\cdot)$ and terminal and boundary conditions given by 
\begin{equation}\label{AF2}
h_\mu(x,s;s)  =  W'(x)\frac{\pa w_\mu(x,s)}{\pa x}, \quad h_\mu(0,t;s)=0, \ t<s.
\end{equation}

Observe from \eqref{D2}, Lemma~\ref{lem:uniformFluxBound} and Lemma~\ref{lem:PostiveFlux} that there is a constant $C$ depending on $\|\theta_1(\cdot)\|_\infty,\|\theta_2(\cdot)\|_\infty$ such that
\begin{equation}\label{AG2}
|h_\mu(x,s;s)| \ \le \ \frac{CW(x)}{\sqrt{T-s}} \ ,  \quad\text{for } x>0, \ \max[0,T-1]<s<T \ .
\end{equation}
Hence from \eqref{AG2} and Lemma~\ref{lem:Adjoint:Estimate} we have that
\begin{equation}\label{AR2}
|h_\mu(x,t;s)| \ \le \ \frac{CW(x)}{\sqrt{T-s}} \ ,  \quad\text{for } x>0, \ \max[0,T-1]<t<s<T \ ,
\end{equation}
 for some constant $C$ depending on $\|\theta_1(\cdot)\|_\infty,\|\theta_2(\cdot)\|_\infty$. The inequality \eqref{Z2} follows now from \eqref{AB2}, \eqref{AE2}, \eqref{AR2}. To prove \eqref{AQ2} we observe from \eqref{AR2} and  Lemma~\ref{lem:uniformFluxBound} that
 \begin{equation}\label{AS2}
\left|\frac{\pa h_\mu(0,t;s)}{\pa x}\right| \ \le \ \frac{C}{\sqrt{(s-t)(T-s)}} \ ,  \quad\text{for } \max[0,T-1]<t<s<T,
\end{equation}
where again $C$  depends on $\|\theta_1(\cdot)\|_\infty,\|\theta_2(\cdot)\|_\infty$. The inequality \eqref{AQ2} follows now from \eqref{AB2}, \eqref{AE2}, \eqref{AS2}.
\end{proof}

\begin{proof}[Proof of Proposition~\ref{prop:ShortTimeExistence}]
For $T>0$ and $\del>0$ we define the metric space
\begin{equation*}
 X_{T,\del} =\set[\Big]{ \theta \in \cC([0,T],\R) \ : \  \theta(0) = \theta_0\; , \  \theta_0 - \delta \leq \theta(t) \leq \rho \ \ \forall t\in [0,T]}  \ ,
\end{equation*}
together with the distance $d_\infty(\theta_1,\theta_2) = \norm{\theta_1 - \theta_2}_\infty$. Given $\theta(\cdot)\in X_{T,\del}$, we define the function
\begin{equation}\label{AK2}
\mathcal{B}\theta:[0,T]\ra\R \qquad\text{by} \qquad \mathcal{B}\theta(t)+ \int_0^\infty W(x)c(x,t) \dx{x} \ = \ \rho \ ,
\end{equation}
where $c(\cdot,t), \ 0\le t\le T,$ is the solution to \eqref{A1}, \ \eqref{E1} with the given initial data satisfying \eqref{AJ2} and $\theta(\cdot)\in X_{T,\del}$.

We  show that $\mathcal{B}\theta(\cdot)\in X_{T,\del}$ provided $T$ is sufficiently small. By the non-negativity of the function~$c(\cdot,\cdot)$ we have from \eqref{AK2} that $\mathcal{B}\theta(\cdot)\le \rho$. From \eqref{AI2} we have that
\begin{equation}\label{AL2}
\mathcal{B}\theta(T)-\theta_0 \ = \  \int_0^\infty [W(x)-w(x,0)]c(x,0) \dx{x} - \int_0^T dt \ \frac{\pa w(0,t)}{\pa x} e^{-V(0)+\theta(t)W(0)}  \ ,
\end{equation}
where $w(\cdot,t)$ is the solution to \eqref{A2}, \eqref{B2}, \eqref{A*2} with terminal data $w_0(\cdot)=W(\cdot)$.    From Lemma~\ref{lem:Adjoint:Estimate} we have that
\begin{align*}
 \int_0^\infty [W(x)-w(x,0)]c(x,0) \dx{x} &\ \ge\  -C(T,\theta_0,\del) \int_0^\infty W(x)c(x,0) \dx{x}\\
 &\ = \ C(T,\theta_0,\del)\bra*{\theta_0-\rho} \ ,
\end{align*}
for some constant $C(T,\theta_0,\del)$ which satisfies $\lim_{T\ra 0} C(T,\theta_0,\del)=0$. We have also from Lemma~\ref{lem:Adjoint:Estimate} and Lemma~\ref{lem:uniformFluxBound}, applied with $x_1=1$, that
\begin{align}\label{AN2}
0 \ &\le \ \int_0^T \dx{t} \ \frac{\pa w(0,t)}{\pa x}\exp[-V(0)+\theta(t)W(0)] \\
\ &\le \ \exp\pra*{ - V(0) + \rho W(0)} \int_0^T \frac{C}{\sqrt{T-t}} \bra[\bigg]{ \sup_{0\leq x \leq 1}\abs{W(x)} + (T-t)\sup_{0\leq t \leq T} \abs{w(1,t)}} \\
\ &\le \ C(\theta_0,\del)  \ \sqrt{T} \ , \quad 0<T\le 1,
\end{align}
for some constant $C(\theta_0,\del)$.  It follows from \eqref{AL2}--\eqref{AN2} that there exists a positive constant $\ga(\theta_0,\del)<1$ such that $\mathcal{B}\theta(T)\ge \theta_0-\del$ if $T\le \ga(\theta_0,\del)$. A similar argument also implies that $\mathcal{B}\theta(t)\ge \theta_0-\del$ for $0\le t\le T$. We have therefore shown that   $\mathcal{B}\theta(\cdot)\in X_{T,\del}$ provided  $T\le \ga(\theta_0,\del)$.

Next we show that $T$ can be taken sufficiently small, depending only on $\theta_0,\del$, so that the mapping $\mathcal{B}:X_{T,\del}\ra X_{T,\del}$ is a contraction with respect to $d_\infty$.  For $i=1,2$ let  $\theta_i(\cdot)\in X_{T,\del}$ and $w_i(x,t), \ x>0,t<T,$ be the solution to \eqref{A2}, \eqref{B2}, \eqref{A*2}   with $\theta(\cdot)=\theta_i(\cdot)$ and terminal data $w_i(\cdot,T)=W(\cdot)$. From \eqref{AL2} we have that
\begin{align} \label{AO2}
\mathcal{B}\theta_1(T)-\mathcal{B}\theta_2(T) &=  \int_0^\infty \bra*{w_2(x,0)-w_1(x,0)} \; c(x,0) \dx{x} \\
&\qquad + \int_0^T\bra*{\frac{\pa w_2(0,t)}{\pa x}-\frac{\pa w_1(0,t)}{\pa x}} e^{-V(0)+\theta_2(t)W(0)}  \dx{t} \nonumber \\
&\qquad + \int_0^T \frac{\pa w_1(0,t)}{\pa x} \bra*{ e^{-V(0)+\theta_2(t)W(0)}- e^{-V(0)+\theta_1(t)W(0)} } \dx{t} \ . \nonumber
\end{align}
We can bound the first two integrals on the RHS of \eqref{AO2} by using Lemma~\ref{lem:Adjoint:Comparison}. Thus from \eqref{Z2} we have that
\begin{equation}\label{AP2}
\left|\int_0^\infty [w_2(x,0)-w_1(x,0)]c(x,0) \dx{x}\right| \ \le \ C(\theta_0,\del)\sqrt{T} \ \|\theta_2(\cdot)-\theta_1(\cdot)\|_\infty \ .
\end{equation}
From \eqref{AQ2} we have that
\begin{equation*}
\abs*{\int_0^T \bra*{\frac{\pa w_2(0,t)}{\pa x}-\frac{\pa w_1(0,t)}{\pa x}} e^{-V(0)+\theta_2(t)W(0)} \dx{t}\; } \le C(\theta_0,\del)\; T \; \|\theta_2(\cdot)-\theta_1(\cdot)\|_\infty \ .
\end{equation*}
The third integral on the RHS of \eqref{AO2} can be bounded just as in \eqref{AN2}, yielding also a bound  $C(\theta_0,\del)\sqrt{T} \ \|\theta_2(\cdot)-\theta_1(\cdot)\|_\infty$.

Since the argument of the previous paragraph also applies for $\mathcal{B}\theta_1(t)-\mathcal{B}\theta_2(t), \ 0<t<T,$ we conclude that
\begin{equation}\label{AU2}
\|\mathcal{B}\theta_1(\cdot)-\mathcal{B}\theta_2(\cdot)\|_\infty \ \le \  C(\theta_0,\del)\sqrt{T} \ \|\theta_1(\cdot)-\theta_2(\cdot)\|_\infty
\end{equation}
provided $T\le C(\theta_0,\del)$.   Hence by taking $T$ small enough, depending only on $\theta_0,\del$, the mapping $\mathcal{B}$ is a contraction on $X_{T,\del}$.  It follows that there is a unique solution to
 \eqref{A1}, \eqref{B1}, \eqref{D1}, \eqref{E1} in the interval $0\le t\le T$.
\end{proof}

\subsection{A priori estimates for global existence and proof of Theorem~\ref{thm:GlobalExistence}}

\begin{lemma}\label{lem:adjoint:Wbound}
Assume that the functions $V(\cdot),W(\cdot)$ are $\cC^2$ on $[0,\infty)$,  satisfy the conditions of Lemma~\ref{lem:Adjoint:Estimate} and in addition the inequality
\begin{equation}\label{AV2}
W''(x) \ \le \ C_0W'(x)^2 \ ,  \quad\text{for } x\ge x_0 \ ,
\end{equation}
where $C_0,x_0$ can be taken to be the same as in \eqref{D2}. Let $\theta:[0,T]\ra\R$ be a continuous function and  $w(x,t), \ x>0,t<T,$ be the solution to \eqref{A2}, \eqref{B2}, \eqref{A*2} with $w_0(x)=W(x), \ x>0$.  Then there exists $C_\infty>0, \ \theta_\infty<0$, independent of $T$,   such that $w(x,0)\le C_\infty W(x)$ for $x\ge 0$, provided $\sup_{0\le t\le T}\theta(t)\le \theta_\infty$.
\end{lemma}
\begin{proof}
We  use the representation of the solution to \eqref{A2}, \eqref{B2} as an expectation value. Thus if $X(\cdot)$ is the solution to the SDE
\begin{equation}\label{AW2}
dX(t) \ = \  b(X(t),t) \dx{t} + \sqrt{2} \dx{B(t)}  \ ,
\end{equation}
then
\begin{equation}\label{AX2}
w(x,0) \ =  \ \EX\pra[\Big]{ w_0(X(T));  \ \inf_{0\le s\le T} X(s)>0 \ \Big| \ X(0)=x \ } \ .
\end{equation}
It follows from \eqref{AW2} and the It\^{o} calculus that
\begin{equation} \label{AY2}
dW(X(t)) \ = \ \left[W'(X(t))b(X(t),t)+W''(X(t))\right] \dx{t}  +W'(X(t))\sqrt{2} \dx{B(t)} \ .
\end{equation}
Let $\psi:[0,\infty)\ra\R$ be a $\cC^\infty$ convex decreasing function with the properties
\begin{equation}\label{AZ2}
\psi(0)=1, \quad \psi(x)=0 \text{ for } x\ge x_0+1,  \quad \psi'(x)<0 \text{ for } 0\le x<x_0+1.
\end{equation}
Then there exists $\del>0$ such that the function $\Phi(x)=W(x)-\del\psi(x)$ has the property $\Phi(x)\ge W(x)/2, \ x\ge 0$.
Setting $Y(t)=\Phi(X(t))$, then \eqref{AW2} yields an evolution equation
\begin{equation}\label{BJ2}
dY(t) \ = \ \mu(Y(t),t) \dx{t}+\sig(Y(t)) \dx{B(t)} \ ,
\end{equation}
with
\begin{align*}
  \mu(y,t) = \Phi'(y)b(y,t) + \Phi''(y) \qquad\text{and}\qquad \sigma(y) = \sqrt{2} \; \Phi'(y) .
\end{align*}
Then we see from \eqref{AX2} that the solution to \eqref{A2}, \eqref{B2} with $w_0(\cdot)=W(\cdot)$  satisfies the inequality
\begin{equation}\label{BA2}
w(x,0) \ \le  \ 2\; \EX\pra[\Big]{ Y(T) ,  \ \inf_{0\le s\le T} Y(s)>\Phi(0) \ \Big| \ Y(0)=\Phi(x)\  } \qquad\text{for } x>0 \ .
\end{equation}

With $b(\cdot,\cdot)$ given by  \eqref{A*2}, it follows from \eqref{D2}, \eqref{AV2}, \eqref{AY2}, \eqref{BJ2} that for $\theta_\infty$ sufficiently negative the functions $\mu(\cdot,\cdot), \ \sig(\cdot)$ in \eqref{BJ2} satisfy the inequalities
\begin{equation}\label{BB2}
\mu(y,t) \ \le \  -\sig(y)^2 \ , \quad \sig(y) \ \le \ \sqrt{2C_0y} \quad\text{for }  y\ge W(x_0+1), \ 0\le t\le T.
\end{equation}
Let $\phi:[\Phi(0),\infty)\ra\R$ be given by
\begin{equation}\label{BC2}
\phi(y) \ = \ K\bra[\big]{\exp[-\xi \Phi(0)]-e^{-\xi y} }+\exp[Ay]-\exp[A\Phi(0)] \ ,
\end{equation}
where $K,\xi,A>0$. Then
\begin{multline} \label{BD2}
\mu(y,t)\pderiv{\phi(y)}{y}+\frac{\sig(y)^2}{2}\pderiv[2]{\phi(y)}{y} \ = \   \\
\mu(y,t)\bra[\big]{K\xi e^{-\xi y}+A \exp[Ay]} +\frac{\sig(y)^2}{2}\bra[\big]{-K\xi^2 e^{-\xi y}+A^2\exp[Ay]} \ .
\end{multline}
In view of \eqref{BB2}, we see that if $A<2$ the RHS of \eqref{BD2} is negative for all $y\ge W(x_0+1),K,\xi\ge 0$.  Since the function $\psi(\cdot)$ is convex and decreasing,  it follows from \eqref{AV2}  that $\Phi''(x)\le C_0W'(x)\Phi'(x)$ if $x\ge x_0$. We conclude that $\mu(y,t)\le 0$ for $\Phi(x_0)\le y\le W(x_0+1)$ provided $\theta_\infty$ is sufficiently negative. For any $\xi>1$ let $K(\xi)$ be such that $K(\xi)\xi\exp[-\xi y]\ge 2A\exp[Ay]$ for $y=W(x_0+1)$. Then if $K\ge K(\xi)$ the RHS of \eqref{BD2} is negative for $\Phi(x_0)\le y\le W(x_0+1)$. We can also see that there are constants $C_1,c_1>0$ such that $\mu(y,t)\le C_1$ and $\sig(y)\ge c_1$ for $\Phi(0)\le y\le \Phi(x_0)$. Hence there exists $\xi>1$ such that
the RHS of \eqref{BD2} is negative for all $\Phi(0)\le y\le \Phi(x_0)$ provided $K\ge K(\xi)$. 

It follows from the argument of the previous paragraph that for any $L>W(x_0+1)$ there exists positive $K,\xi,A$ with $\xi>1,A<2$ such that the RHS of \eqref{BD2} is negative for all $\Phi(0)\le y\le L$. We can use the corresponding function $\phi(\cdot)$ of \eqref{BC2} to obtain a bound on the expectation in \eqref{BA2}.  For $L>W(x_0+1)$ and $\Phi(0)<y<L$ let $\tau$ be the first exit time from the interval $[\Phi(0),L]$ for the diffusion $Y(t), \ t\ge 0$. We have from the maximum principle that
\begin{equation}\label{BF2}
\Prob(Y(\tau)=L\  | \  Y(0)=y) \ \le \ \frac{\phi(y)}{\phi(L)} \ .
\end{equation}
Let $A_{\ell}$ be the event $A_{\ell} = \set{ \ell + \Phi(0) \leq \sup_{0\leq s\leq T} Y(s) \leq \ell+1 + \Phi(0)}$. Then, for any integer $\ell_0$ which satisfies $\ell_0+\Phi(0)> \max[y, W(x_0+1)]$, we can estimate
\begin{align}
\MoveEqLeft{\EX\pra[\Big]{ Y(T);  \ \inf_{0\le s\le T} Y(s)>\Phi(0) \ \Big| \ Y(0)=y \ }}\\
&\leq \ell_0 + \Phi(0) + \EX\pra[\Big]{ Y(T) ; \ \bigcup\nolimits_{\ell \geq \ell_0} A_\ell ;  \ \inf_{0\le s\le T} Y(s)>\Phi(0) \ \Big| \ Y(0)=y \ }\\
&\le \ell_0+\Phi(0)+\phi(y)\sum_{\ell=\ell_0}^\infty \frac{\ell+1+\Phi(0)}{\phi(\ell+\Phi(0))}  \ , \label{BG2}
\end{align}
where the last bound follows from \eqref{BF2}. Since the sum in \eqref{BG2} is bounded, we conclude that for any $x_1>0$ there is a constant $C(x_1)$, independent of $T$ such that
\begin{equation}\label{BI2}
\sup_{0\le t\le T} w(x_1,t) \ \le  \ C(x_1) \ .
\end{equation}
To complete the proof of the lemma we now argue as in Lemma~\ref{lem:Adjoint:Estimate}. Thus we see from \eqref{D2}, \eqref{E2}, \eqref{AV2} that $\theta_\infty$ can be chosen sufficiently negative so that $\mathcal{L}w_\la(x,t)\le 0$ for $x\ge x_0, \ \la\ge 0$. It follows then from \eqref{BI2} and the maximum principle that $w(x,0)\le [C(x_0)/W(x_0)+1]W(x)$ for~$x\ge x_0$.
\end{proof}
\begin{lemma}\label{lem:adjoint:fluxBound}
Assume that the functions $V(\cdot),W(\cdot)$ are $\cC^2$ on $[0,\infty)$, satisfy the conditions of Lemma~\ref{lem:Adjoint:Estimate}, and in addition  the inequality
\begin{equation}\label{BK2}
W'(x)^2 \ \ge \ c_0,  \quad\text{for } x\ge x_0,
\end{equation}
where $c_0$ is a positive constant. 
 Let $\theta:[0,T]\ra\R$ be a continuous function and  $w(x,t), \ x>0,t<T,$ be the solution to \eqref{A2}, \eqref{B2}, \eqref{A*2} with $w_0(x)=W(x), \ x>0$.  Then there exists $\theta_\infty<0$ and $C_\infty,\del_\infty>0$, independent of $T$,   such that
 \begin{equation}\label{BL2}
0 \ \le \  \frac{\pa w(0,t)}{\pa x} \ \le  C_\infty \max\set*{ \frac{1}{\sqrt{T-t}} , 1} e^{-\del_\infty(T-t)} \quad\text{for } 0\le t\le T.
 \end{equation}
 provided $\sup_{0\le t\le T}\theta(t)\le \theta_\infty$.
\end{lemma}
\begin{proof}
From \eqref{BK2} there exists $\theta_\infty<0$ such that the drift $b(\cdot,\cdot)$ of \eqref{A*2} satisfies the inequality $b(x,t)\le -\theta_\infty\sqrt{c_0}/2, \ x\ge x_0,0\le t\le T$.  From \eqref{D2} we have that
\begin{equation}\label{BM2}
\frac{W'(x)}{\sqrt{W(x)}} \ \le C_0 \quad\text{for } x\ge x_0 \ ,
\end{equation}
whence we conclude that $W(x)\le C_1[1+x^2], \ x\ge 0,$ for some constant $C_1$.  Let $v(x,t), \ x\ge 0,t<T$ be the solution to the PDE
\begin{equation}\label{BN2}
\frac{\pa v(x,t)}{\pa t}+\frac{\pa^2 v(x,t)}{\pa x^2}+b(x)\frac{\pa v(x,t)}{\pa x} \ = \ 0, \quad x>0,t<T,
\end{equation}
with terminal and boundary conditions given by
\begin{equation}\label{BO2}
v(x,T) \ = \ C_1[1+x^2],   \ x>0, \quad v(0,t)=0, \ t< T.
\end{equation}
Then, provided that $\sup_{0\le t\le T}b(x,t)\le b(x)$ for $x\ge 0$, one has that $w(x,t)\le v(x,t)$ for $(x,t)\in \R^+ \times (0,T)$.
   
 For any $a>0$ let $\Phi(\cdot)$ be the function  $\Phi(x)=1-e^{-ax}+ ae^{-ax_0}x, \ x\ge 0$. Making the change of variable $y=\Phi(x), \ v(x,t)=\tilde{v}(y,t)$ in \eqref{BN2} we obtain the PDE
  \begin{multline} \label{BP2}
\frac{\pa \tilde{v}(y,t)}{\pa t}+\tilde{a}(y)\frac{\pa^2 \tilde{v}(y,t)}{\pa y^2}+\tilde{b}(y)\frac{\pa \tilde{v}(y,t)}{\pa y} \ = \ 0, \quad y>0,t<T, \\
\tilde{a}(y) \ = \ \Phi'(x)^2 \ , \quad  \tilde{b}(y) \ = \ \Phi'(x)b(x)+\Phi''(x) \ .
\end{multline} 
It is easy to see that $a>0$ can be chosen sufficiently large, depending on $\theta_\infty$,  so that
\begin{equation}\label{BQ2}
\tilde{b}(y) \ \le \  -ae^{-ax_0}\theta_\infty\sqrt{c_0}/2 \ , \quad \tilde{a}(y) \ \le 4a^2 \ \quad\text{for } y\ge 0 \ .
\end{equation}
Let $Y(\cdot)$ be the solution to the SDE
\begin{equation}\label{BR2}
dY(s) \ = \ \tilde{b}(Y(s)) \dx{t}+\sqrt{2\tilde{a}(Y(s))}  \dx{B(s)} \ ,
\end{equation}
and for $y>0$ let $\tau_{y,t}$ be the first exit time from $[0,\infty)$ for the solution $Y(s), \ s\ge t,$ to \eqref{BR2} with $Y(t)=y$.  Then
\begin{equation}\label{BS2}
\tilde{v}(y,t) \ = \  \EX\pra*{ \ \tilde{v}(Y(T),T) H(\tau_{y,t}-T) \ | \ Y(t)=y \ } \ ,
\end{equation}
where $H:\R\ra\R$ is the Heaviside function $H(s)=0,  s<0, \ H(s)=1,s>0$. We conclude from \eqref{BS2} and the Schwarz inequality that
\begin{equation}\label{BT2}
\tilde{v}(y,t) \ \le \  \EX\pra*{ \ \tilde{v}(Y(T),T)^2 \ | \ Y(t)=y \ }^{1/2} \Prob(\tau_{y,t}>T)^{1/2} \ .
\end{equation}

We can estimate the first expectation on the RHS of \eqref{BS2} by following the argument of Lemma~\ref{lem:adjoint:Wbound} and using \eqref{BQ2}.  Thus there is a constant $C(y)$, independent of $T$, such that
\begin{equation}\label{BU2}
\sup_{0\le t\le T} \EX\pra*{ \ \tilde{v}(Y(T),T)^2 \ | \ Y(t)=y \ } \ \le \  C(y) \ .
\end{equation}
To estimate $\Prob(\tau_{y,t}>T)$ we use the fact that the function $\phi(y) \ = \ \EX\pra*{\exp\{\del\tau_{y,0}\} \ }, \ y>0$,  is the solution to the boundary value problem
\begin{equation}\label{BV2}
\del\phi(y)+\tilde{b}(y)\pderiv{\phi(y)}{y}+\tilde{a}(y)\pderiv[2]{\phi(y)}{y} \ =  \  0,  \quad \phi(0)=1.
\end{equation}
From \eqref{BQ2} and the maximum principle we see that there exists $\del,\eta>0$ such that $\phi(y)\le e^{\eta y}, \ y>0$.  We conclude that $\Prob(\tau_{y,t}>T)\le \exp[-\del(T-t)+\eta y]$. Now \eqref{BT2} yields a bound on the function $\tilde{v}$, and hence on $w$. We see that for any $x_1>0$ there is a constant $C(x_1)$, independent of $T$ such that
\begin{equation}\label{BW2}
w(x,t) \ \le \ C(x_1)e^{-\del(T-t)/2} \quad\text{if }  0<x<x_1, \ 0\le t\le T.
\end{equation}

To complete the proof of \eqref{BL2} we argue as in Lemma~\ref{lem:uniformFluxBound}. Thus there is a constant $C$ depending only on $\theta_\infty$ such that
\begin{align}\label{BX2}
 \frac{\pa w(0,t)}{\pa x} \ \le C \max\set*{ \frac{1}{\sqrt{T-t}} , 1} \bra*{\sup_{0<x<1} w(x,t+1)+\sup_{t\le s\le t+1} w(1,s)}
\end{align}
The inequality \eqref{BL2} follows
 from \eqref{BW2}, \eqref{BX2}
on  taking $\del_\infty=\del/2$.
\end{proof}
\begin{remark}
The exponential bound \eqref{BL2} is sufficient for our purposes. However one could obtain more accurate asymptotics as $T\ra\infty$ by using the theory of large deviations \cite{fw,var}. In \cite{ds} precise asymptotics for $\Prob(\tau_{y,0}>T)$ have been obtained for drifts $\tilde{b}(y)=-\beta/y^p$ with $\beta>0, 0<p<1$.
\end{remark}
The above Lemmas ensure that the local in time solution of Proposition~\ref{prop:ShortTimeExistence} can be prolonged to all times.
\begin{proof}[Proof of Theorem~\ref{thm:GlobalExistence}]
For any $T>0$, let $c(\cdot,t), \theta(t), \ 0\le t\le T,$ be the solution to \eqref{A1}, \eqref{B1}, \eqref{D1}, \eqref{E1} constructed in Proposition~\ref{prop:ShortTimeExistence}.
We have then from \eqref{D1}, \eqref{AI2}  that
\begin{equation}\label{BY2}
\rho-\theta(T) \ = \   \int_0^\infty w(x,0)c(x,0) \dx{x} + \int_0^T \dx{t} \ \frac{\pa w(0,t)}{\pa x}\exp[-V(0)+\theta(t)W(0)]  \ ,
\end{equation}
where $w$ is the solution to \eqref{A2}, \eqref{B2}, \eqref{A*2} with $w_0(\cdot)\equiv W(\cdot)$. Assume now that $\sup_{0\le t\le T}\theta(t)\le \theta_\infty$, where $\theta_\infty$ is as in the statements of Lemma~\ref{lem:adjoint:Wbound} and Lemma~\ref{lem:adjoint:fluxBound}. Then we conclude from Lemmas \ref{lem:uniformFluxBound}, \ref{lem:adjoint:Wbound}, \ref{lem:adjoint:fluxBound} and \eqref{BY2} that
\begin{multline} \label{BZ2}
\rho-\theta(T) \ \le \  C_\infty[\rho-\theta(0)]+ \\
\exp[-V(0)+\theta_\infty W(0)]\left\{C_\infty\int_0^{\max[0,T-1]}e^{-\del_\infty(T-t)} \dx{t}+C\int_{\max[0,T-1]}^T \frac{\dx{t}}{\sqrt{T-t}} \right\} \ ,
\end{multline}
for some constants $C_\infty,C$ independent of $T, c(\cdot,\cdot)$.  The inequality \eqref{BZ2} implies a lower bound on $\theta(T)$ of the form
\begin{equation}\label{CA2}
\theta(T) \ \ge  C_\infty\theta(0)+C'_\infty, \quad C_\infty,C'_\infty \text{ independent of } T,c(\cdot,\cdot) \ .
\end{equation}
Since \eqref{CA2} yields a lower bound on $\theta(\cdot)$ uniform for all time, we  conclude from the local existence result that the solution to \eqref{A1}, \eqref{B1}, \eqref{D1}, \eqref{E1} can be extended for all time.
\end{proof}

\bigskip

\section{Regularity Theory}\label{s:Reg}

\subsection{Strategy and results}
In this section we shall prove various regularity results for the solution to \eqref{A1}, \eqref{B1}, \eqref{D1}, \eqref{E1} assuming that $a(\cdot)\equiv 1$ and the initial data satisfies \eqref{AJ2}. Our main goal is to show that the function $\theta(\cdot)$, which occurs as part of the solution of  \eqref{A1}-\eqref{E1} is $\cC^1$. In order to do this we begin by assuming that 
$\theta(\cdot)$ is continuous, as was established in the proof of Proposition \ref{prop:ShortTimeExistence}. We write $\theta(t)$ in terms of the integral of $c(\cdot,t)$ as given by \eqref{D1}. It follows then from  boundary regularity properties for solutions to the parabolic PDE problem \eqref{A1}, \eqref{E1} with initial data satisfying \eqref{AJ2}  that $\theta(\cdot)$ is H\"{o}lder continuous up to order $1/2$. Now by bootstrapping the boundary regularity results, we obtain after two iterations that $\theta(\cdot)$ is $\cC^1$.

Our boundary regularity results essentially show for solutions to \eqref{A1}, \eqref{E1} that 
the function $t\ra \lim_{x\ra 0}\pa_xc(x,t)$ has the same regularity as a half derivative of the boundary data \eqref{E1}.  It is natural to expect this since the function $t\ra \lim_{x\ra 0}\pa_tc(x,t)$ has the same regularity as a full derivative of the boundary data.  The proof below of this result is quite delicate, requiring a careful analysis of the first two terms in the perturbation expansion \eqref{N2} of the Green's function.  Our result does not seem to already be in the literature although there are substantial boundary regularity results, even for parabolic systems in many dimensions \cite{bdm}. 

We shall initially just be concerned with the solution to the PDE \eqref{A1} with general drift $b(\cdot,\cdot)$ satisfying \eqref{C2}.  In~§\ref{s:ExistUnique} we defined the solution to the initial value problem for \eqref{A1} with boundary data \eqref{E1} and integrable initial  data by \eqref{AI2} (see Definition~\ref{def:weakSol}). This uniquely determines the function $c(\cdot,T)$ for $T>0$ as a positive measure on $[0,\infty)$, but yields no regularity properties. The first step is to establish some interior regularity for $c(\cdot,\cdot)$.
\begin{lemma}\label{lem:reg:C1}
Assume that for some $T_0>0$ the function $\theta:[0,T_0]\ra\R$ is continuous and that $b(\cdot,\cdot)$ satisfies \eqref{C2} with $T=T_0$. If the initial data $c(x,0), \ x>0,$ for \eqref{A1}, \eqref{E1} is integrable, then the function $c:[0,\infty)\times (0,T_0]\ra\R$ determined by \eqref{AI2} is continuous and satisfies \eqref{E1}.  In addition, the function $c(\cdot,T)$ is $\cC^1$ in $(0,\infty)$ for all $0<T\le T_0$.
\end{lemma}
We suppose now that the drift for \eqref{A1} with $a(\cdot)\equiv1$ is given by \eqref{A*2}, and that the initial data for \eqref{A1}, \eqref{E1} satisfies \eqref{AJ2}. Provided $\theta(\cdot)$ is continuous and $V,W$ satisfy the conditions of Lemma~\ref{lem:Adjoint:Estimate} we may define a function $\Theta(T), \ T\ge 0,$ in terms of the solution to \eqref{A1}, \eqref{E1} by
\begin{equation}\label{E*4}
\Theta(T)+\int_0^\infty  W(x)\, c(x,T) \dx{x} \ = \ \rho \ .
\end{equation}
Our goal will be to show that the function $\Theta(\cdot)$ is more regular by roughly a half derivative than the function $\theta(\cdot)$ which enters the boundary condition \eqref{E1}.
\begin{lemma}\label{lem:Hoelder}
Assume that for some $T_0>0$ the function $\theta:[0,T_0]\ra\R$ is continuous, $V,W$ satisfy the conditions of Lemma~\ref{lem:Adjoint:Estimate}, and  $b(x,t)=\theta(t)W'(x)-V'(x)$ satisfies \eqref{C2} with $T=T_0$. If the initial data $c(x,0), \ x>0,$ for \eqref{A1}, \eqref{E1} satisfies \eqref{AJ2}  then for any positive $\al,\beta$ with $\al<1/2, \ \beta<T_0,$ there is a constant $C$ depending on $\al,\beta$  such that
\begin{equation}\label{F4}
|\Theta(T)-\Theta(T')| \ \le \ C|T-T'|^{\al} \quad\text{for } \beta \le T,T'\le T_0 \ .
\end{equation}
\end{lemma}
This basic regularity estimate follows from a comparison principle.
To obtain more detailed properties, we take an alternative approach to the proof of Lemma~\ref{lem:Hoelder} by formally differentiating~\eqref{E*4} with respect to $T$.   Then, using \eqref{A1} and twice formally integrating by parts with respect to~$x$, we obtain the identity
\begin{multline} \label{M4}
\pderiv{\Theta(T)}{T} \ = \  W(0)\frac{\pa c(0,T)}{\pa x}-\left[W(0)b(0,T)+W'(0)\right]c(0,T) \\
-\int_0^\infty \left[W'(x)b(x,T)+W''(x)\right]c(x,T) \dx{x} \ .
\end{multline}
From  \eqref{D2} we see that the LHS of \eqref{M4} minus the first term on the RHS is bounded. We therefore expect from Lemma~\ref{lem:Hoelder} that the function $T\ra\pa c(0,T)/\pa x$ has roughly the same regularity as the derivative of a  H\"{o}lder continuous function of order $1/2$.  To see this we consider for any $L>0$ the perturbation expansion \eqref{N2} for the Dirichlet Green's function $G_D$ on the interval $0<x<L$. We define $c_{0,L}(x,T), \ c_{1,L}(x,T), \ 0<x<L, T>0$ by
 \begin{align}
 c_{0,L}(x,T) \ &= \ \int_0^T \dx{t} \ \frac{\pa K_L(0,x,t,T)}{\pa x'}\exp[-V(0)+\theta(t)W(0)] \ ,   \label{N4} \\
 \quad c_{1,L}(x,T) \ &= \ \int_0^T \dx{t}  \ \frac{\pa v_{0,L}(0,x,t,T)}{\pa x'}\exp[-V(0)+\theta(t)W(0)] \ , \nonumber
 \end{align}
 where $K_L$ is given by \eqref{M2} and $v_{0,L}$ by \eqref{N2}.

 The next Lemma, will show that the functions $c_{0,L}$ and $c_{1,L}$ after passing to the limit $L\to \infty$ yield the leading order contribution in terms of regularity of a solution.
\begin{lemma}\label{lem:Schauder1}
Assume  $\theta(\cdot), \ b(\cdot,\cdot), \ c(\cdot,0)$ satisfy the conditions of Lemma~\ref{lem:reg:C1} and let $c(x,T), \ x>0, \ 0<T\le T_0,$ be the solution of \eqref{A1}, \eqref{E1} determined by \eqref{AI2}. Then  the function $(x,T)\ra [c(x,T)-c_{0,\infty}(x,T)-c_{1,\infty}(x,T)]$ is $\cC^1$ in $x$ in the region $0\le x<\infty, \  0<T\le T_0$. Moreover, it holds $\lim_{x\ra0}c_{0,\infty}(x,T)=\exp[-V(0)+\theta(T)W(0)], \ 0<T\le T_0$ and $\lim_{x\ra0}c_{1,\infty}(x,T)=0, \ 0<T\le T_0$.
\end{lemma}
Next we prove a regularity result for $c_{1,\infty}(x,T)$ as $x\ra 0$, which is consistent with Lemma~\ref{lem:Hoelder} and the formula \eqref{M4}.
\begin{lemma}\label{lem:Schauder2}
Assume that for some $T_0>0$ the function $\theta:[0,T_0]\ra\R$ is continuous and that $b(\cdot,\cdot)$ satisfies \eqref{C2} with $T=T_0$.
Then for any positive $\al,\beta$ with $\al<1, \ \beta<T_0,$ there is a constant $C$ depending on $\al,\beta$  such that
\begin{equation}\label{AA4}
\sup_{0<x<1}\left|\int_{T'}^T \frac{\pa c_{1,\infty}(x,t)}{\pa x} \dx{t} \right| \ \le \ C|T-T'|^{\al} \quad\text{for } \beta \le T,T'\le T_0 \ .
\end{equation}
If $s\ra b(0,s)$ is locally H\"{o}lder continuous in the interval $0<s\le T_0$, then the function $(x,T)\ra c_{1,\infty}(x,T)$  is $\cC^1$ in $x$ in the region $0\le x<\infty, \  0<T\le T_0$.
\end{lemma}
The proof of regularity for $c_{0,\infty}(x,T)$ as $x\ra 0$, consistent with Lemma~\ref{lem:Hoelder} and the formula~\eqref{M4}, is much simpler than for $c_{1,\infty}(x,T)$.
\begin{lemma}\label{lem:Schauder3}
Assume that for some $T_0>0$ the function $\theta:[0,T_0]\ra\R$ is continuous, and locally H\"{o}lder continuous on the interval $0<T\le T_0$ with order $\ga, \ 0\le \ga<1/2$.
Then for any positive $\beta\le T_0$ there is a constant $C$ depending on $\beta,\ga$  such that
\begin{equation}\label{CI4}
\sup_{0<x<1}\left|\int_{T'}^T \frac{\pa c_{0,\infty}(x,t)}{\pa x} \dx{t}\right| \ \le \ C|T-T'|^{\al} \quad\text{for } \beta \le T,T'\le T_0 \ ,
\end{equation}
where $\al=\ga+1/2$.
If $T\ra \theta(T)$ is locally H\"{o}lder continuous in the interval $0<T\le T_0$ of order  $\ga>1/2$, then the function $(x,T)\ra c_{0,\infty}(x,T)$  is $\cC^1$ in $x$ in the region $0\le x<\infty, \  0<T\le T_0$.
\end{lemma}
We use Lemmas \ref{lem:Schauder1}-\ref{lem:Schauder3} to obtain a sharper version of Lemma~\ref{lem:Hoelder}.
\begin{proposition}\label{prop:improveHoelder}
Assume the conditions of Lemma~\ref{lem:Hoelder} and in addition that  the function $\theta:[0,T_0]\ra\R$ is  locally H\"{o}lder continuous on the interval $0<T\le T_0$ with order $\ga, \ 0\le \ga<1/2$. Then the inequality \eqref{F4} holds with  $\al=\ga+1/2$.
If $\theta(\cdot)$ is locally H\"{o}lder continuous in the interval $0<T\le T_0$ of order  $\ga>1/2$, then the function $\Theta(T), \ 0<T< T_0$  is $\cC^1$ and the identity \eqref{M4} holds.
\end{proposition}
\begin{proof}
Let $\phi:[0,\infty)\ra\R^+$ be a $\cC^\infty$ function which has the property that $\phi(x)=1, \ 0\le x\le 1,$ and $\phi(x)=0, \ x\ge 2$.
For $0<\ve\le 1$ and $L\ge 1$ we define the function $\Theta_{\ve,L}(T), \ 0<T\le T_0$ by
\begin{equation}\label{CW4}
\Theta_{\ve,L}(T)+\int_\ve^\infty W(x)\phi(x/L)\,c(x,T) \dx{x} \ = \ \rho \ .
\end{equation}
Then $\Theta_{\ve,L}(T), \ 0<T<T_0,$ is $\cC^1$ and
\begin{multline} \label{CX4}
\pderiv{\Theta_{\ve,L}(T)}{T} \ = \  W(\ve)\frac{\pa c(\ve,T)}{\pa x}-\left[W(\ve)b(\ve,T)+W'(\ve)\right]c(\ve,T)  \\
-\int_\ve^\infty \left[W'(x)\phi(x/L)+L^{-1}W(x)\phi'(x/L)\right]b(x,T)c(x,T) \dx{x} \\
-\int_\ve^\infty \left[W''(x)\phi(x/L)+2L^{-1}W'(x)\phi'(x/L)+L^{-2} W(x)\phi''(x/L)\right]c(x,T) \dx{x} \ .
\end{multline}
To establish \eqref{CX4} we first assume that the drift $b(\cdot,\cdot)$ has sufficient regularity that $c(x,T),\ x>0, \ 0<T<T_0$ is a classical solution to \eqref{A1}, \eqref{D1}. Then \eqref{CX4} follows from integration by parts. The identity continues to hold for $b(\cdot,\cdot)$ just satisfying \eqref{C2} by an approximation argument.  Next we define $\Theta_\ve(T)$ by
\begin{equation}\label{CY4}
\Theta_\ve(T)+\int_\ve^\infty W(x)\,c(x,T) \dx{x} \ = \ \rho \ .
\end{equation}
so $\Theta_\ve(T)=\lim_{L\ra\infty}\Theta_{\ve,L}(T)$.  It follows from \eqref{C2} that $|L^{-1}W'(x)\phi'(x/L)|\le C, \ x\ge 0, \ L\ge 1,$ for some constant $C$. Hence on taking $L\ra\infty$ in \eqref{CX4} we conclude that $\Theta_\ve(T), \ 0<T<T_0$ is $\cC^1$ and satisfies the identity
\begin{multline} \label{CZ4}
\pderiv{\Theta_\ve(T)}{T} \ = \  W(\ve)\frac{\pa c(\ve,T)}{\pa x}-\left[W(\ve)b(0,T)+W'(\ve)\right]c(\ve,T) \\
-\int_\ve^\infty \left[W'(x)b(x,T)+W''(x)\right]c(x,T) \dx{x} \ .
\end{multline}
To conclude the proof we observe from \eqref{E*4}, \eqref{CY4} that $\Theta(T)=\lim_{\ve\ra 0}\Theta_\ve(T), \ 0<T<T_0$.  The result then follows from \eqref{CZ4} and Lemmas \ref{lem:Schauder1}-\ref{lem:Schauder3}.
\end{proof}
The smoothness result of Theorem~\ref{thm:ExistUniqueConvergence} follows now by iteratively applying the improvement of the Hölder exponent of Proposition~\ref{prop:improveHoelder}.
\begin{theorem}\label{thm:C1}
Assume the conditions of Proposition~\ref{prop:ShortTimeExistence} hold, and let $c(x,t), \ x\ge 0, \ 0<t<T_0$ be the corresponding local in time solution to \eqref{A1},  \eqref{B1}, \eqref{D1}.  Then $c(x,t)$ and $\pa c(x,t)/\pa x$ are continuous in the region $x\ge 0, \ 0<t<T_0,$ and the function $\theta(t), \ 0<t<T_0$ is $\cC^1$.  Furthermore the identity \eqref{M4} with $\Theta(\cdot)=\theta(\cdot)$ holds for $0<t<T_0$ .
\end{theorem}
\begin{proof}
Proposition~\ref{prop:ShortTimeExistence} shows that the function $\theta(\cdot)$ is continuous in the interval $[0,T_0]$. Since $\Theta(\cdot)=\theta(\cdot)$ we can conclude from Lemma~\ref{lem:Hoelder} that $\theta(\cdot)$ is locally H\"{o}lder continuous of order $\ga$ for any $\ga<1/2$. Now Proposition~\ref{prop:improveHoelder} implies that $\theta(\cdot)$ is  locally H\"{o}lder continuous of order $\ga$ for any $\ga<1$.  Applying Proposition~\ref{prop:improveHoelder} again we conclude that $\theta(\cdot)$ is $\cC^1$. The regularity of the function $c(\cdot,\cdot)$ follows from  Lemma~\ref{lem:reg:C1} and Lemmas \ref{lem:Schauder1}-\ref{lem:Schauder3}.
\end{proof}

We close this section with two more regularity properties. The first one is a parabolic regularization property of the solution, which shows that starting from $c(\cdot,0)$ satisfying~\eqref{AJ2}, the solution will be bounded for any positive time and is in addition equicontinuous.
\begin{lemma}\label{lem:stability:sup}
Let $(c(\cdot,t),\theta(t)), \ 0<t\le T_0,$ be the solution of \eqref{A1}, \eqref{B1}, \eqref{D1}, \eqref{E1}  with initial data satisfying \eqref{AJ2} constructed in Proposition~\ref{prop:ShortTimeExistence}.  Then for any $t_0$ satisfying $0<t_0<T_0$, there exists $N>0$ depending only on $t_0,T_0$ and $\|\theta(\cdot)\|_\infty$ such that
\[
  \sup c(\cdot,t)\le N \qquad\text{ for }\qquad t_0\le t\le T_0 \ .
\]
In addition, for any $L>0$  the family of functions $c(\cdot,t), \ t_0\le t\le T_0,$ is equicontinuous on the interval $[0,L]$.
\end{lemma}
The second result is a tightness property for the solution on time intervals, where $\theta$ is negative.
\begin{lemma}\label{lem:tight}
  Assume that the functions $V(\cdot),W(\cdot)$ are $\cC^2$ on $[0,\infty)$, satisfy the conditions of Lemma~\ref{lem:Adjoint:Estimate} and in addition~\eqref{F1:2} of Assumption~\ref{ass:VW}. Let $(c(\cdot,t),\theta(t)), \ t>0,$ be the solution of \eqref{A1}, \eqref{B1}, \eqref{D1}, \eqref{E1}  with initial data satisfying \eqref{AJ2} constructed in Theorem~\ref{thm:GlobalExistence} and $\Theta_\infty \le \theta(t)\le \theta_\infty$ for all $t>0$ and some $\Theta_\infty \le \theta_\infty<0$. Then, for any $\del>0$ there exists $M(\del)>0$ such that
\begin{equation}\label{eq:tight}
\int_{M(\del)}^\infty W(x) c(x,t) \dx{x} \ < \ \del \quad\text{for all } t\ge 0 \ .
\end{equation}
\end{lemma}

\medskip

\subsection{Basic regularity Lemmas~\ref{lem:reg:C1} and~\ref{lem:Hoelder}}

\begin{proof}[Proof of Lemma~\ref{lem:reg:C1}]
For $L\ge 1, \ 0<T\le T_0,$ let $w_L(x,t), \ 0\le x\le L, \ t\le T,$ be the solution to~\eqref{A2} on the interval $0<x<L, t<T,$ with terminal and boundary conditions given by \eqref{I2} where $x_1=L$.  We define the function $c_L(x,T), \ 0<x<L, \ 0<T\le T_0,$ similarly to \eqref{AI2} by
\begin{equation}\label{A4}
\int_0^L \! w_0(x) c_L(x,T) \dx{x} \;= \;  \int_0^L \! w_L(x,0)c(x,0) \dx{x} \; +\; \int_0^T  \! \pa_x w_L(0,t) \; e^{-V(0)+\theta(t)W(0)} \dx{t} \ .
\end{equation}
Observe that if $w_0$ is a continuous function of compact support then the RHS of \eqref{A4} converges to the RHS of \eqref{AI2} as $L\ra\infty$. In Lemma~\ref{lem:uniformFluxBound} we constructed the Green's function $G_{D,L}$ for the solution to the terminal-boundary value problem for $w_L$. It follows then from \eqref{A4} that
\begin{equation}\label{B4}
c_L(x,T)  =  \int_0^L \!\!\! G_{D,L}(x',x,0,T)c(x',0) \dx{x'} + \int_0^T\!\!\! \pa_{x'} G_{D,L}(0,x,t,T) \; e^{-V(0)+\theta(t)W(0)} \dx{t} \ .
\end{equation}
From \eqref{B4} and the Green's function estimates in \cite[Lemma 3.4]{cg} we see that  the function $c_L$ with domain $(0,L)\times (0,T_0]$ extends to a continuous function $c_L:[0,L]\times (0,T_0]\ra\R$ which satisfies the boundary conditions
\begin{equation}\label{C4}
c_L(0,T) \ = \ \exp[-V(0)+\theta(T)W(0)] \ , \quad c_L(L,T) \ = \ 0, \quad 0<T\le T_0 \  .
\end{equation}

We wish to show that the function $c_L(\cdot,\cdot)$ given by \eqref{B4} converges as $L\ra\infty$ to a continuous function, whence it will follow that the function $c(\cdot,\cdot)$ defined by \eqref{AI2} is continuous. To do this we observe that for any $0<L'< L$  then
\begin{align}
c_L(x,T) &\ = \  \int_0^{L'} G_{D,L'}(x',x,0,T)c(x',0) \dx{x'}  \nonumber \\
&\qquad +  \int_0^T \pa_{x'} G_{D,L}(0,x,t,T) \; e^{-V(0)+\theta(t)W(0)} \dx{t} \label{D4} \\
&\qquad -  \int_0^T \pa_{x'} G_{D,L}(L',x,T',T)\; c_L(L',T')  \dx{T'}  \ ,  \ \  0<x<L', \ 0<T\le T_0 \ . \nonumber
\end{align}
The identity \eqref{D4} can easily be established  for a drift $b(\cdot,\cdot)$ with sufficient regularity that $c_L(x,T)$ is a classical solution to \eqref{A1}, \eqref{C4} in the interval  $0<x<L,0<T<T_0$. This requires greater regularity on $b(\cdot,\cdot)$ than \eqref{C2}.  From the Green's function estimates in \cite[Lemma 3.4]{cg} we can then infer by a limiting argument that \eqref{D4} continues to hold just under the assumption \eqref{C2}. To show that $c_L(x,T)$ converges as $L\ra\infty$, we consider a continuous function $w_1:(0,\infty)\ra\R$ of compact support with integral equal to $1$. We then multiply \eqref{D4} by $w_1(L')$ and integrate with respect to $L'$.  After integration, the contribution of the first two integrals on the RHS of \eqref{D4} to $c_L(x,T)$ continue to be independent of $L$.

To estimate the contribution of the third integral on the RHS of \eqref{D4} we use \eqref{A4} with $T'$ in place of $T$ and $w_0$ given by the formula
\begin{equation}\label{E4}
w_0(L') \ = \ \frac{\pa G_{D,L'}(L',x,T',T)}{\pa x'}w_1(L') \ .
\end{equation}
Hence the integral  with respect to $L'$ of $w_1(L')$ times the third integral on the RHS of \eqref{D4} can be written in terms of an integral in which $c_L(L',T')$ is replaced by $c(L',0)$, plus a boundary term corresponding to the second integral on the RHS of \eqref{A4}.  Assuming that $x>0$ in \eqref{E4} lies to the left of the support of $w_1$, we see that the function $w_0$ in \eqref{E4} is continuous and uniformly bounded for $0< T'<T$.  Hence we can control the limit as $L\ra\infty$ of the integral involving $c(L',0)$, and we can similarly control the boundary term.  We conclude that $\lim_{L\ra\infty} c_L(x,T)=c(x,T)$,  and the limit is uniform in any rectangle $\{0\le x\le x_0, \ T_1\le T\le T_0\}$  with $x_0,T_1>0$. It follows that the function $c:[0,\infty)\times (0,T_0]\ra\R$ determined by \eqref{AI2} is continuous and satisfies \eqref{E1}. Upon differentiating \eqref{D4} with respect to $x$, we  similarly see from the Green's function estimates in \cite[Lemma 3.4]{cg}  that $\pa c_L(x,T)/\pa x$ also converges uniformly as $L\ra\infty$  in rectangles $\{x_1\le x\le x_0, \ T_1\le T\le T_0\}$ with $x_1,T_1>0$. Hence $\pa c(x,T)/\pa x$ exists and is continuous for $x>0, \ 0<T\le T_0$.
\end{proof}

\bigskip

\begin{proof}[Proof of Lemma~\ref{lem:Hoelder}]
Since $\beta>0$ we may assume from Lemma~\ref{lem:reg:C1} that the initial data $c(\cdot,0)$ is continuous on $[0,\infty)$. Hence it will be sufficient for us prove \eqref{F4} assuming this and  $T'=0$.  From \eqref{AI2}, \eqref{E*4} we have that
\begin{equation}\label{G4}
\Theta(T)-\Theta(0) \ = \  \int_0^\infty \bra[\big]{ W(x)-w(x,0) } c(x,0) \dx{x} - \int_0^T \pa_x w(0,t)\;  e^{-V(0)+\theta(t)W(0)} \dx{t}   \ ,
\end{equation}
where $w(\cdot,t), \ t<T,$ is the solution to \eqref{A2}, \eqref{B2} with terminal data $w_0(\cdot)=W(\cdot)$. From Lemma~\ref{lem:uniformFluxBound} we see that the second integral on the RHS of \eqref{G4} is bounded by $C_1T^{1/2}$ for some constant $C_1$.  To bound the first integral on the RHS of \eqref{G4} we use the argument of Lemma~\ref{lem:Adjoint:Estimate}. In particular, since from \eqref{A2} it follows $\mathcal{L}_{x,t} w_\la(x,t)\le 0$ for $x\ge 0, \ t<T,$ if $\la>0$ is sufficiently large, we see that the first integral is bounded below as
\begin{equation}\label{H4}
 \int_0^\infty [W(x)-w(x,0)]c(x,0) \dx{x} \ \ge \ -C_2T\int_0^{\infty} W(x) c(x,0) \dx{x} \ ,
\end{equation}
for some constant $C_2$.

To obtain an upper bound we first show that for any positive $\al<1/2$ there is a positive constant~$C_3$ depending only on~$\al$ such that
\begin{equation}\label{I4}
w(T^{\al},t) \ \ge \ W(0)\left[1-C_3\exp\left\{- \frac{1}{4T^{1-2\al}}\right\} \ \right] \quad\text{for } 0\le t<T \ .
\end{equation}
To prove \eqref{I4} we note that, since $W(\cdot)$ is an increasing function, it is sufficient to show that the diffusion $X(\cdot)$ defined by \eqref{AW2} satisfies the inequality
\begin{equation}\label{J4}
\Prob\bra*{\inf_{t\le s\le T} X(s)<0\  \bigg| \  X(t)=T^{\al}} \ \le \ C_3\exp\pra*{ -\frac{1}{4T^{1-2\al}}}  \qquad\text{for } 0\le t<T \ .
\end{equation}
Let $\tau_{t,T}$ be the exit time from the interval $[0,2T^\al]$ for $X(s), s\ge t,$ started at $X(t)=T^\al$. Then the LHS of \eqref{J4} is bounded above by $\Prob(\tau_{t,T}\le T)$. Since the drift for $X(\cdot)$ in the interval $[0,2T^\al]$ is bounded by a constant, we can compare $\Prob(\tau_{t,T}\le T)$ with the exit probability of a pure diffusion from the interval $[0,2T^\al]$. In this way, we conclude that $\Prob(\tau_{t,T}\le T)$ is bounded by the RHS of \eqref{J4} for a suitable constant $C_3$.

Next we use the fact that  $\mathcal{L}_{x,t} w_\la(x,t)\ge 0$ for $x\ge 0, \ 0\le t<T,$ if $\la<0$ is sufficiently small. It follows from \eqref{I4} and the maximum principle that for such $\la$,
\begin{multline} \label{K4}
w(x,t) \ \ge \ c(T)w_\la(x,t) \quad\text{for } x\ge T^\al, \ 0\le t<T \  , \\
 \quad\text{where } c(T)=\frac{W(0)}{W(T^\al)}\left[1-C_3\exp\pra*{- \frac{1}{4T^{1-2\al}}} \ \right]  \ .
\end{multline}
From \eqref{K4} we see that
\begin{multline} \label{L4}
 \int_0^\infty [W(x)-w(x,0)]c(x,0) \dx{x} \\
 \le \ \int_0^{T^\al} W(x)c(x,0) \dx{x} \  + \ \left[1-c(T)e^{\la T}\right]\int_0^\infty W(x)c(x,0) \dx{x} \ .
\end{multline}
Since $c(\cdot,0)$ is continuous the RHS of \eqref{L4} is bounded above by a constant times $T^\al$. We conclude that \eqref{F4} holds for $T'=0$.
\end{proof}

\medskip

\subsection{Schauder Lemmas~\ref{lem:Schauder1}-\ref{lem:Schauder3}}

\begin{proof}[Proof of Lemma~\ref{lem:Schauder1}]
As in Lemma~\ref{lem:Hoelder} we may assume that the initial data $c(\cdot,0)$ is continuous on $[0,\infty)$. Choosing $L\ge 1$, we have similarly to \eqref{D4}  the representation
\begin{align} 
c(x,T) &\ = \  \int_0^{L}  G_{D,L}(x',x,0,T)\; c(x',0) \dx{x'} \nonumber \\
&\qquad + \int_0^T  \pa_{x'} G_{D,L}(0,x,t,T)\; e^{-V(0)+\theta(t)W(0)} \dx{t} \label{O4} \\
&\qquad - \int_0^T  \pa_{x'} G_{D,L}(L,x,T',T)\; c(L,T') \dx{T'} \  ,  \quad 0<x<L, \ 0<T\le T_0 \ .
\end{align}
From \eqref{R2} we see that the derivatives with respect to $x$ of the first and third integrals on the RHS of \eqref{O4} exist and are continuous in $(x,T)$ for $0\le  x<L, \ 0<T\le T_0$. Hence we are left to estimate the derivative with respect to $x$ of the second integral on the RHS of \eqref{O4}. To show this we first observe that, in addition to \eqref{R2}, the Green's functions estimates in \cite[Lemma 3.4]{cg} also imply that
\begin{multline} \label{P4}
\left|\frac{\pa ^2G_{D,L}(x',x,t,T)}{\pa x\pa x'}-\frac{\pa ^2K_L(x',x,t,T)}{\pa x\pa x'}-\frac{\pa ^2v_{0,L}(x',x,t,T)}{\pa x\pa x'}\right| \\
 \le \ C\, G(x-x',2(T-t)) \ , \quad \text{ where $C$ depends only on $L,T$.}
\end{multline}
It follows from \eqref{N4} and \eqref{P4} that the function $(x,T)\ra [c(x,T)-c_{0,L}(x,T)-c_{1,L}(x,T)]$ is~$\cC^1$ in~$x$ in the region $0\le x<L, \  0<T\le T_0$.

Finally we need to estimate the differences between the second derivatives of $K_L,\ v_{0,L}$ and $K_\infty, \ v_{0,\infty}$ respectively.  It is evident from  the representation \eqref{M2} for $K_L$ that for some constant $C$ depending only on $L,T$, it holds for $0<x,x'<L/3$, $0<t<T$
\begin{equation} \label{Q4}
\left|\frac{\pa ^2K_L(x',x,t,T)}{\pa x\pa x'}-\frac{\pa ^2K_\infty(x',x,t,T)}{\pa x\pa x'}\right| \le \  C\exp\pra*{-\frac{L^2}{4(T-t)}} \  .
\end{equation}
It follows from \eqref{N4}, \eqref{Q4} that the function $(x,T)\ra [c_{0,L}(x,T)-c_{0,\infty}(x,T)]$ is $\cC^1$ in $x$ in the region $0\le x<L/3, \  0<T\le T_0$.

To estimate the difference between the second derivatives of $v_{0,L}$ and $v_{0,\infty}$ we first note from~\eqref{N2} that
\begin{equation}\label{R4}
 v_{0,L}(x',x,t,T)   \ = \ \int_t^T I_L(x',x,t,s,T)  \dx{s} \ , \qquad  0<x',x<L, \ \ 0<t<T\le T_0 \ ,
\end{equation}
where  $I_L(x',x,t,s,T)$ is defined by
\begin{equation}\label{S4}
I_L(x',x,t,s,T) \ = \ \int_0^L K_L(x',x'',t,s) \ b(x'',s) \ \frac{\pa K_L(x'',x,s,T)}{\pa x''} \dx{x''} \ .
\end{equation}
Since $K_L(x'',x,s,T)=0$ for $x''=0,L$, we have on integration by parts in \eqref{S4} that $I_L$ is also  given by the expression
\begin{multline} \label{T4}
I_L(x',x,t,s,T) \ = \ - \int_0^L K_L(x',x'',t,s) \ \frac{\pa b(x'',s)}{\pa x''} K_L(x'',x,s,T) \dx{x''} \\
- \int_0^L  \frac{\pa K_L(x',x'',t,s)}{\pa x''} \ b(x'',s) \  K_L(x'',x,s,T) \dx{x''} \ .
\end{multline}
We may estimate the second mixed derivative of $v_{0,L}(x',x,t,T)$ with respect to $x,x'$ by using the representation \eqref{S4} for $I_L$ in the integration \eqref{R4} over $t<s<(T+t)/2$,  and the representation \eqref{T4} over $(T+t)/2<s<T$.  We conclude from \eqref{C2} that there is a constant $C$ depending only on $L,T$ such that
\begin{equation}\label{U4}
\left|\frac{\pa ^2v_{0,L}(x',x,t,T)}{\pa x\pa x'}\right| \ \le \  \frac{C}{\sqrt{T-t}} \  G(x-x',2(T-t)) \ , \qquad 0<x,x'<L, \ 0<t<T \  .
\end{equation}

We use the same method as in the previous paragraph to  estimate the difference between the second derivatives of $v_{0,L}$ and $v_{0,\infty}$. Similarly to how we obtained \eqref{Q4} we see from \eqref{S4} that there is a constant $C$ depending only on $L,T$ such that
\begin{multline} \label{V4}
\left|\frac{\pa ^2I_L(x',x,t,s,T)}{\pa x\pa x'}-\frac{\pa ^2I_\infty(x',x,t,s,T)}{\pa x\pa x'}\right| \ \le \  \\
C\left[\frac{1}{(s-t)^{1/2}}+\frac{1}{(T-s)}\right]\exp\pra*{-\frac{L^2}{16(T-t)}}\ , \ \ 0<x,x'<L/3 \ , \ 0<t<s<T \ .
\end{multline}
From \eqref{T4} we also have that
\begin{multline} \label{W4}
\left|\frac{\pa ^2I_L(x',x,t,s,T)}{\pa x\pa x'}-\frac{\pa ^2I_\infty(x',x,t,s,T)}{\pa x\pa x'}\right| \ \le \  \\
C\left[\frac{1}{(s-t)}+\frac{1}{(T-s)^{1/2}}\right]\exp\pra*{-\frac{L^2}{16(T-t)}}\ , \ \ 0<x,x'<L/3 \ , \ 0<t<s<T \ .
\end{multline}
It follows from \eqref{V4}, \eqref{W4} and \eqref{R4} that for $0<x,x'<L/3$, $0<t<T\le T_0$
\begin{equation} \label{X4}
\left|\frac{\pa ^2v_{0,L}(x',x,t,T)}{\pa x\pa x'}-\frac{\pa ^2v_{0,\infty}(x',x,t,T)}{\pa x\pa x'}\right| \ \le \
C\exp\pra*{-\frac{L^2}{16(T-t)}}\ ,
\end{equation}
for some constant $C$ depending only on $L,T_0$. We conclude from \eqref{X4} that the function $(x,T)\ra [c_{1,L}(x,T)-c_{1,\infty}(x,T)]$ is $\cC^1$ in $x$ in the region $0\le x<L/3, \  0<T\le T_0$.

Taking $L\ra\infty$ in \eqref{M2} we have that $K_\infty(x',x,t,T)=G(x-x',T-t)-G(x+x',T-t)$, whence
\begin{equation}\label{Y4}
\frac{\pa K_\infty(x',x,t,T)}{\pa x'} \ \Bigg|_{x'=0} \ = \  \frac{x}{\sqrt{4\pi}}\frac{1}{(T-t)^{3/2}}\exp\pra*{-\frac{x^2}{4(T-t)}} \ , \quad x>0, \ t<T\  .
\end{equation}
From \eqref{Y4} it is easy to compute the integral
\begin{equation}\label{Z4}
\int_{-\infty}^T \frac{\pa K_\infty(x',x,t,T)}{\pa x'} \dx{t} \ \Bigg|_{x'=0} \ = \ 1.
\end{equation}
Since the integration in \eqref{Z4} is concentrated on the scale $T-t\simeq x^2$, it follows from \eqref{N4} that $\lim_{x\ra0}c_{0,\infty}(x,T)=\exp[-V(0)+\theta(T)W(0)], \ 0<T\le T_0$.  We can also easily see from \eqref{N4}, \eqref{R4}, \eqref{S4} that $\lim_{x\ra0}c_{1,\infty}(x,T)=0, \ 0<T\le T_0$.
\end{proof}

\bigskip

\begin{proof}[Proof of Lemma~\ref{lem:Schauder2}]
We define $\tilde{I}_\infty$ similarly to \eqref{S4} with $L=\infty$  by
\begin{equation}\label{AB4}
\tilde{I}_\infty(x',x,t,s,T) \ = \ \int_0^\infty  K_\infty(x',x'',t,s) \ b(x',s) \ \frac{\pa K_\infty(x'',x,s,T)}{\pa x''}  \dx{x''} \ .
\end{equation}
Corresponding to \eqref{R4}, \eqref{N4} we aso define
\begin{equation}\label{AC4}
 \tilde{v}_{0,\infty}(x',x,t,T)   \ = \ \int_t^T \tilde{I}_\infty(x',x,t,s,T) \dx{s} \ , \qquad 0<x',x<\infty, \ \ 0<t<T\le T_0 \ ,
\end{equation}
\begin{equation}\label{AD4}
\tilde{c}_{1,\infty}(x,T) \ = \ \int_0^T \ \frac{\pa \tilde{v}_{0,\infty}(0,x,t,T)}{\pa x'}\exp[-V(0)+\theta(t)W(0)] \dx{t}\ .
\end{equation}

We first show that the function $(x,T)\ra [c_{1,\infty}(x,T)-\tilde{c}_{1,\infty}(x,T)]$ is $\cC^1$ in $x$ in the region $0\le x<\infty, \  0<T\le T_0$.
To see this we write
\begin{equation}\label{AE4}
I_\infty(x',x,t,s,T) -\tilde{I}_\infty(x',x,t,s,T) \ = \ \hat{I}_{1,\infty}(x',x,t,s,T)+\hat{I}_{2,\infty}(x',x,t,s,T) \  ,
\end{equation}
where
\begin{equation}\label{AF4}
\hat{I}_{2,\infty}(x',x,t,s,T)  \ = \ \int_2^\infty  K_\infty(x',x'',t,s) \ [b(x'',s)-b(x',s)] \ \frac{\pa K_\infty(x'',x,s,T)}{\pa x''} \dx{x''} \ .
\end{equation}
Using integration by parts as in \eqref{T4} we have that
\begin{align}
\hat{I}_{1,\infty}(x',x,t,s,T) \ &= \ - \int_0^2 K_\infty(x',x'',t,s) \ \frac{\pa b(x'',s)}{\pa x''} K_\infty(x'',x,s,T)\dx{x''} \label{AG4} \\
&\quad - \int_0^2 \biggl( \frac{\pa K_\infty(x',x'',t,s)}{\pa x''} \ [b(x'',s)-b(x',s)] \  K_\infty(x'',x,s,T) \notag \\
&\qquad \quad  +K_\infty(x',2,t,s) \ [b(2,s)-b(x',s)] \  K_\infty(2,x,s,T) \biggr) \dx{x''} \ . \notag
\end{align}
From \eqref{C2}, \eqref{AF4} we easily  see that there is a constant $C$ depending only on $T$ such that
\begin{equation}\label{AH4}
\left|\frac{\pa^2\hat{I}_{2,\infty}(x',x,t,s,T)}{\pa x\pa x'}\right|  \ \le \   C\exp\pra*{-\frac{1}{8(T-t)}} \ , \quad 0<x,x'<1, \ 0<t<s<T \ .
\end{equation}
From \eqref{C2}, \eqref{AG4} we see there is a constant $C$ depending only on $T$ such that
\begin{align}
\MoveEqLeft{\left|\frac{\pa^2\hat{I}_{1,\infty}(x',x,t,s,T)}{\pa x\pa x'}\right|}  \\
&\le \frac{C}{(s-t)^{1/2}(T-s)^{1/2}}\int_{-\infty}^\infty G(x'-x'',2(s-t))G(x''-x,2(T-s)) \dx{x''} \notag \\
  &\qquad  +C\exp\pra*{-\frac{1}{8(T-t)}} \ , \qquad 0<x,x'<1, \ 0<t<s<T \ . \label{AI4}
\end{align}
Evidently \eqref{AI4} implies that
\begin{equation}\label{AJ4}
\left|\frac{\pa^2\hat{I}_{1,\infty}(x',x,t,s,T)}{\pa x\pa x'}\right|  \ \le  \frac{C}{(s-t)^{1/2}(T-s)^{1/2}}G(x'-x,2(T-t))
\end{equation}
for some constant $C$ depending only on $T$. Using \eqref{AH4}, \eqref{AJ4} in the integrations \eqref{R4}, \eqref{AC4} we conclude that
\begin{equation}\label{AK4}
\left|\frac{\pa^2v_{0,\infty}(x',x,t,T)}{\pa x\pa x'}-\frac{\pa^2\tilde{v}_{0,\infty}(x',x,t,T)}{\pa x\pa x'}\right|  \ \le  CG(x'-x,2(T-t))
\end{equation}
for some constant $C$ depending only on $T$. It follows from \eqref{N4}, \eqref{AD4}, \eqref{AK4} that the function $(x,T)\ra [c_{1,\infty}(x,T)-\tilde{c}_{1,\infty}(x,T)]$ is $\cC^1$ in $x$ in the region $0\le x<\infty, \  0<T\le T_0$.

We consider next the regularity of the function  $\tilde{c}_{1,\infty}(x,T)$ defined by \eqref{AD4}. To understand the degree of regularity one might expect we first look at the case when the function $b(\cdot,\cdot)$ is constant say $b(\cdot,\cdot)\equiv 1$. In that case we observe that $w(z,t)=\tilde{v}_{0,\infty}(z,x,t,T)$ is the solution to the terminal-boundary value problem
\begin{align} \label{AL4}
\frac{\pa w(z,t)}{\pa t}+\frac{\pa^2 w(z,t)}{\pa z^2}+ \frac{\pa K_\infty(z,x,t,T)}{\pa z} \ &= \ 0 \ , \quad z>0, \ t<T \ , \\
w(0,t) \ = \ 0,  \ \ t<T \ , \quad w(z,T) \ &= \ 0 \ , \ \ \ z>0 \ . \notag
\end{align}
Evidently $w(z,t)=(x-z) K_\infty(z,x,t,T)/2$ is the solution to \eqref{AL4}.  It follows that
\begin{equation}\label{AM4}
\frac{\pa \tilde{v}_{0,\infty}(0,x,t,T)}{\pa x'} \ = \ \frac{x}{2}\frac{\pa K_\infty(0,x,t,T)}{\pa x'} \ .
\end{equation}
If we use \eqref{AM4} in \eqref{AD4} we have from \eqref{Y4} that
\begin{equation}\label{AN4}
\tilde{c}_{1,\infty}(x,T) \ = \ \frac{x^2}{4\sqrt{\pi}}\int_0^T   \ \frac{1}{(T-t)^{3/2}}\exp\pra*{-\frac{x^2}{4(T-t)}} \exp[-V(0)+\theta(t)W(0)] \dx{t} \ .
\end{equation}
Letting $x\ra 0$ in \eqref{AN4} we have using identities similar to \eqref{Z4} that
\begin{equation}\label{AO4}
\lim_{x\ra 0} \tilde{c}_{1,\infty}(x,T)  \ = \ 0 \ , \qquad \lim_{x\ra 0} \frac{\pa \tilde{c}_{1,\infty}(x,T)}{\pa x} \ = \  \frac{1}{2}\exp[-V(0)+\theta(T)W(0)] \ .
\end{equation}
It follows from \eqref{AO4} that in the case when $b(0,s)$ is independent of $s$, the function $(x,T)\ra \tilde{c}_{1,\infty}(x,T)$ of \eqref{AD4}  is $\cC^1$ in $x$ in the region $0\le x<\infty, \  0<T\le T_0$.

The argument in the previous paragraph can be extended to include functions $b(\cdot,\cdot)$ which have the property that $s\ra b(0,s)$ is H\"{o}lder continuous in the interval $0<s\le T_0$.  To see this we observe from \eqref{AB4} that
\begin{align} \label{AP4}
\frac{\pa \tilde{I}_\infty(x',x,t,s,T)}{\pa x'} \Bigg|_{x'=0} \ &= \  b(0,s) \int_0^\infty \frac{\pa K_\infty(0,x'',t,s)}{\pa x'} \ \frac{\pa K_\infty(x'',x,s,T)}{\pa x''} \dx{x''} \\
&= \ b(0,s)J_1(x,t,s,T) \ . \notag
\end{align}
Using integration by parts as in \eqref{T4} we see there is a universal constant $C$ such that
\begin{equation}\label{AR4}
\left|\frac{\pa J_1(x,t,s,T)}{\pa x}\right| \  \le \ \frac{C\; G(x,2(T-t))}{(s-t)^{1/2}(T-s)^{1/2}}\; \min\set*{\frac{1}{(s-t)^{1/2}}, \frac{1}{(T-s)^{1/2}}}  \ .
\end{equation}
It follows from \eqref{AC4} and \eqref{AP4}, \eqref{AR4} that if the function $s\ra b(0,s)$ is H\"{o}lder continuous of order $\ga>0$ at $s=T$ then there is a constant $C_\ga$ such that
\begin{equation} \label{AS4}
\left| \frac{\pa^2 \tilde{v}_{0,\infty}(0,x,t,T)}{\pa x'\pa x}-b(0,T)\int_t^T\frac{\pa J_1(x,t,s,T)}{\pa x} \dx{s}\right|     \le  C_\ga(T-t)^{-1/2+\ga} \ G(x,2(T-t)) \ .
\end{equation}
From \eqref{AM4} we have that
\begin{equation}\label{AT4}
\int_t^T J_1(x,t,s,T) \dx{s} \ = \ \frac{x}{2}\frac{\pa K_\infty(0,x,t,T)}{\pa x'} \ .
\end{equation}
We conclude from \eqref{AS4}, \eqref{AT4} and the argument of the previous paragraph that when $s\ra b(0,s)$ is H\"{o}lder continuous in the interval $0<s\le T_0$, then the function $(x,T)\ra \tilde{c}_{1,\infty}(x,T)$ of \eqref{AD4}  is $\cC^1$ in $x$ in the region $0\le x<\infty, \  0<T\le T_0$.

Finally we wish  to establish \eqref{AA4}  assuming only that $b(\cdot,\cdot)$ satisfies \eqref{C2}. To prove this it will be sufficient to consider the function $\tilde{c}_{2,\infty}$ defined by
\begin{equation}\label{BC4}
\tilde{c}_{2,\infty}(x,T) \ = \ \int_0^T \Ga(x,t,T) \exp[-V(0)+\theta(t)W(0)] \dx{t} \ ,
\end{equation}
where $\Ga$ is the function
\begin{equation}\label{BD4}
\Ga(x,t,T) \ = \ \int_t^T b(0,s) J_1(x,t,s,T) \dx{s} \  .
\end{equation}
If $b(0,\cdot)\equiv 1$ then the RHS of \eqref{BC4} is given by the RHS of \eqref{AN4}.
We observe from \eqref{AP4} that
\begin{align} \label{AU4}
J_1(x,t,s,T) \ &= \ \frac{1}{8\pi(s-t)^{3/2}(T-s)^{3/2}}\int_0^\infty z \exp\pra*{-\frac{z^2}{4(s-t)}} \\
&\qquad\qquad \times \ \left\{(x-z)\exp\pra*{-\frac{(x-z)^2}{4(T-s)}}+(x+z)\exp\pra*{-\frac{(x+z)^2}{4(T-s)}} \right\} \dx{z}  \\
&= \ \frac{\exp\pra*{-x^2/4(T-s)}}{4\pi(s-t)^{3/2}(T-s)^{3/2}} \ J_2(x,t,s,T) \ ,
\end{align}
where $J_2(x,t,s,T)$ is given by the formula
\begin{multline} \label{AV4}
J_2(x,t,s,T) \ = \ \int_0^\infty z \exp\pra*{-\frac{(T-t)z^2}{4(s-t)(T-s)}} \\
 \times \  \left\{x\cosh\pra*{\frac{zx}{2(T-s)}}-z\sinh\pra*{\frac{zx}{2(T-s)}}\right\} \dx{z}   \ .
\end{multline}
Using integration by parts in \eqref{AV4} we have that
\begin{align*} 
J_2(x,t,s,T) &\ = \ \frac{2x(T-s)^2(s-t)}{(T-t)^2} \\
&\qquad +\ \left\{\frac{x^2(T-s)(s-t)}{(T-t)^2}-\frac{2(T-s)(s-t)}{(T-t)}\right\} \ J_3(x,t,s,T) \ , 
\end{align*}
where $J_3$ is given by the formula
\begin{equation}\label{AX4}
J_3(x,t,s,T) \ = \ \int_0^\infty \exp\left[-\frac{(T-t)z^2}{4(s-t)(T-s)}\right] \  \sinh\left[\frac{zx}{2(T-s)}\right] \dx{z} \ .
\end{equation}
Differentiating \eqref{AX4} with respect to $x$ and  integrating by parts we have that
\begin{equation}\label{AY4}
\frac{\pa J_3(x,t,s,T) }{\pa x} \ = \ \frac{s-t}{T-t}+\frac{(s-t)x}{2(T-t)(T-s)} \ J_3(x,t,s,T) \ .
\end{equation}
Integrating the differential equation in \eqref{AY4}, we see that
\begin{equation}\label{AZ4}
J_3(x,t,s,T) \ = \ \left(\frac{s-t}{T-t}\right)\exp\left[\frac{(s-t)x^2}{4(T-t)(T-s)}\right] \int_0^x \exp\left[-\frac{(s-t)z^2}{4(T-t)(T-s)}\right] \dx{z} \ .
\end{equation}
We conclude from  \eqref{AU4}--\eqref{AZ4} that
\begin{multline} \label{BA4}
J_1(x,t,s,T) \ = \ \frac{\exp\left[-x^2/4(T-t)\right]}{2\pi(T-t)^2} \Bigg\{ x\left(\frac{s-t}{T-s}\right)^{-1/2}\exp\left[-\frac{(s-t)x^2}{4(T-t)(T-s)}\right] \\
+\left(\frac{x^2}{2(T-t)}-1\right)\left(\frac{s-t}{T-s}\right)^{1/2} \int_0^x \exp\left[-\frac{(s-t)z^2}{4(T-t)(T-s)}\right] \dx{z}\ \Bigg\} \ .
\end{multline}
Note that it follows from \eqref{Y4}, \eqref{AT4}, \eqref{BA4} that
\begin{multline} \label{BB4}
\frac{x^2\sqrt{\pi(T-t)}}{2} \ = \ \int_t^T \Bigg\{ x\left(\frac{s-t}{T-s}\right)^{-1/2}\exp\left[-\frac{(s-t)x^2}{4(T-t)(T-s)}\right] \\
+\left(\frac{x^2}{2(T-t)}-1\right)\left(\frac{s-t}{T-s}\right)^{1/2} \int_0^x  \ \exp\left[-\frac{(s-t)z^2}{4(T-t)(T-s)}\right] \dx{z} \ \Bigg\} \dx{s} \ .
\end{multline}
From \eqref{BD4}, \eqref{BA4} we have that
\begin{equation}\label{BF4}
\Ga(x,t,T) \ = \ \frac{\exp\left[-x^2/4(T-t)\right]}{2\pi(T-t)^2}\left\{ \frac{x^2}{2(T-t)}\Ga_1(x,t,T)+\Ga_2(x,t,T)\right\} \ ,
\end{equation}
where $\Ga_1(x,t,T)$ is given by
\begin{align} \label{BG4}
\Ga_1(x,t,T) \ &= \ \int_t^T b(0,s) \ \left(\frac{s-t}{T-s}\right)^{1/2} \int_0^x \exp\left[-\frac{(s-t)z^2}{4(T-t)(T-s)}\right] \dx{z}  \dx{s} \\
&= \ \sqrt{T-t}\int_t^T b(0,s)\int_0^{ (s-t)^{1/2}x/(T-t)^{1/2}(T-s)^{1/2}} \exp\left[-\frac{y^2}{4}\right] \dx{y} \dx{s} \ . \notag
\end{align}
It is easy to see from \eqref{BG4} that
\begin{equation}\label{BH4}
\left|\frac{\pa \Ga_1(x,t,T)}{\pa x} \right| \ \le \  C(T-t) \ , \quad 0<t<T\le T_0 \ ,
\end{equation}
for some constant $C$. Hence if we define the function $\tilde{c}_{3,\infty}$  by
\begin{equation}\label{BI4}
\tilde{c}_{3,\infty}(x,T) \ = \ x^2\int_0^T \frac{\exp\left[-x^2/4(T-t)\right]}{4\pi(T-t)^3} \Ga_1(x,t,T) \exp[-V(0)+\theta(t)W(0)] \dx{t} \ ,
\end{equation}
we conclude from \eqref{BH4}  that $\lim_{x\ra 0}\tilde{c}_{3,\infty}(x,T)=0$ and $\pa \tilde{c}_{3,\infty}(x,T)/\pa x$ remains bounded as $x\ra 0$. In view of the continuity of the functions $\theta(\cdot), \ b(0,\cdot)$, we may further conclude that $\lim_{x\ra 0} \pa \tilde{c}_{3,\infty}(x,T)/\pa x$ exists and its dependence on $\theta(\cdot), \ b(0,\cdot)$, is only through the values~$\theta(T)$ and~$b(0,T)$.

From \eqref{BA4} we see that
\begin{multline} \label{BJ4}
\frac{\pa \Ga_2(x,t,T)}{\pa x} \ = \ x\frac{\pa \Ga_3(x,t,T)}{\pa x}+ \\ \int_t^T b(0,s)\left\{\left(\frac{s-t}{T-s}\right)^{-1/2}-\left(\frac{s-t}{T-s}\right)^{1/2}\right\} \exp\left[-\frac{(s-t)x^2}{4(T-t)(T-s)}\right] \dx{s}  \ ,
\end{multline}
where $\Ga_3(x,t,T)$ is given by
\begin{equation}\label{BK4}
\Ga_3(x,t,T) \ = \ \int_t^T b(0,s) \ \left(\frac{s-t}{T-s}\right)^{-1/2}\exp\left[-\frac{(s-t)x^2}{4(T-t)(T-s)}\right] \dx{s} \ .
\end{equation}
We see from \eqref{BK4} that
\begin{equation}\label{BL4}
\left|\frac{\pa \Ga_3(x,t,T)}{\pa x} \right| \ \le \  Cx \ , \quad 0<t<T\le T_0 \ ,
\end{equation}
for some constant $C$. Hence if we define the function $\tilde{c}_{4,\infty}$  by
\begin{equation}\label{BM4}
\tilde{c}_{4,\infty}(x,T) \ = \ \int_0^T \int_0^x  \frac{\exp\bra*{-x^2/4(T-t)}}{2\pi(T-t)^2} x'\frac{\pa \Ga_3(x',t,T)}{\pa x'}  \, e^{-V(0)+\theta(t)W(0)}  \dx{x'} \dx{t}\ ,
\end{equation}
we conclude from \eqref{BL4}  that $\lim_{x\ra 0}\tilde{c}_{4,\infty}(x,T)=0$ and $\pa \tilde{c}_{4,\infty}(x,T)/\pa x$ remains bounded as $x\ra 0$. We further conclude as with $\tilde{c}_{3,\infty}$  that $\lim_{x\ra 0} \pa \tilde{c}_{4,\infty}(x,T)/\pa x$ exists and its dependence on $\theta(\cdot), \ b(0,\cdot)$, is only through the values $\theta(T), \ b(0,T)$.

We define $\Ga_4(x,t,T)$ by $\Ga_4(0,t,T) =  0$ and 
\begin{equation} \label{BN4} 
\frac{\pa \Ga_4(x,t,T)}{\pa x} =   \int_t^T \! b(0,s)\bra*{\bra*{\frac{T-s}{s-t}}^{1/2}\!\!\! - \bra*{\frac{s-t}{T-s}}^{1/2}}
\exp\pra*{-\frac{(s-t)x^2}{4(T-t)(T-s)}}  \dx{s} \  . \notag
\end{equation}
We see from \eqref{BN4} that
\begin{equation}\label{BO4}
\left|\frac{\pa \Ga_4(x,t,T)}{\pa x} \right| \ \le \  C(T-t) \ , \quad 0<t<T\le T_0 \ ,
\end{equation}
for some constant $C$.  We define the function $\tilde{c}_{5,\infty}$  by
\begin{equation}\label{BP4}
\tilde{c}_{5,\infty}(x,T) \ = \ \int_0^T   \frac{\exp\left[-x^2/4(T-t)\right]}{2\pi(T-t)^2} \Ga_4(x,t,T)  \exp[-V(0)+\theta(t)W(0)] \dx{t}\ .
\end{equation}
The bound \eqref{BO4} implies that $\tilde{c}_{5,\infty}(x,T)$ is continuous for $x\ge 0$, $0<T\le T_0$ and satisfies $\lim_{x\ra 0}\tilde{c}_{5,\infty}(x,T)=0$, but $\pa \tilde{c}_{5,\infty}(x,T)/\pa x$ may diverge as $x\ra 0$. Let $f(x,T)$ be the function
\begin{align} \label{BQ4}
f(x,T) \ &= \  \int_0^T \  \frac{\exp\left[-x^2/4(T-t)\right]}{2\pi(T-t)^2} \frac{\pa \Ga_4(x,t,T)}{\pa x}  \exp[-V(0)+\theta(t)W(0)] \dx{t}  \\
&= \ \int_{0<t<s<T} \ b(0,s)\exp[-V(0)+\theta(t)W(0)]  \notag \\
&\qquad\qquad\qquad\quad \times \frac{(T+t-2s)}{2\pi(T-t)^2(s-t)^{1/2}(T-s)^{1/2}} \exp\left[-\frac{x^2}{4(T-s)}\right] \dx{t} \dx{s} \ . \notag
\end{align}
We can see as in the previous paragraph that  $\lim_{x\ra0} [\pa\tilde{c}_{5,\infty}(x,T)/\pa x-f(x,T)]$ exists.  Hence in order to complete the proof of \eqref{AA4} it will be sufficient to show for $0<\al<1, \ 0<\beta<T_0$, there is a constant $C_{\al,\beta}$ such that
\begin{equation}\label{BR4}
\sup_{0<x<1}\left|\int_{T_1}^{T_2} f(x,T) \dx{T} \right| \ \le \ C_{\al,\beta}|T_2-T_1|^{\al} \quad\text{for } \beta \le T_1<T_2\le T_0 \ .
\end{equation}

To prove  \eqref{BR4} we first observe for fixed $s,t$ with $0<t<s<T_1$ that
\begin{equation}\label{BS4}
\int_{T_1}^{T_2} \frac{(T+t-2s)}{2\pi(T-t)^2(s-t)^{1/2}(T-s)^{1/2}} \exp\left[-\frac{x^2}{4(T-s)}\right] \dx{T} \ = \ g(x,s-t,t) \ ,
\end{equation}
 where the function $g(x,\tau,t)$ is given by
 \begin{equation}\label{BT4}
 g(x,\tau,t) \ = \ \frac{1}{2\pi\tau}\int_{1-2\tau/(T_1-t)}^{1-2\tau/(T_2-t)} \frac{z}{\sqrt{1-z^2}}\exp\left[-\frac{x^2(1-z)}{4\tau(1+z)}\right] \dx{z}  \ .
 \end{equation}
 In obtaining \eqref{BS4}, \eqref{BT4} we have used the substitution $z=(T+t-2s)/(T-t)$.  Note also that the integration in \eqref{BT4} is over a subinterval of $\{-1<z<1\}$.
 To estimate the integral in \eqref{BT4} we first assume that $0<\tau<(T_1-t)/2$ in which case the integrand is positive and
 \begin{align} \label{BU4}
 0 \ &< \  g(x,\tau,t) \ \le \  g(0,\tau,t)\\
 &=  \frac{1}{\pi\sqrt{\tau(T_1-t)}}\left\{\left[1-\frac{\tau}{T_1-t}\right]^{1/2}-\left(\frac{T_1-t}{T_2-t}\right)^{1/2}\left[1-\frac{\tau}{T_2-t}\right]^{1/2} \ \right\} \ .
\end{align}
We have from \eqref{BU4} that
\begin{equation}\label{BV4}
g(0,\tau,t)\ \le \   \frac{1}{\pi\sqrt{\tau}}\left\{\frac{1}{\sqrt{T_1-t}}-\frac{1}{\sqrt{T_2-t}}\right\} \ .
\end{equation}
It follows from \eqref{BV4} that
\begin{equation}\label{BX4}
\int_0^{(T_1-t)/2}g(0,\tau,t) \dx{\tau} \ \le \  \frac{\sqrt{2}}{\pi}\left[1-\left(\frac{T_1-t}{T_2-t}\right)^{1/2}\right] \ \le \ \frac{C(T_2-T_1)}{T_2-t} \
\end{equation}
for some constant $C$.  Next we show that there is a constant $C$ such that
\begin{equation}\label{BY4}
|g(x,\tau,t)| \ \le \ \frac{C(T_2-T_1)}{(T_2-t)\sqrt{(T_1-t)(T_1-t-\tau)}}\  , \qquad (T_1-t)/2<\tau<T_1-t \ .
\end{equation}
Evidently \eqref{BY4} holds if $(T_2-T_1)/(T_2-t)\ge 1/4$ so we shall assume that $(T_2-T_1)/(T_2-t)< 1/4$, which implies that $2\tau/(T_1-t)-2\tau/(T_2-t)<1/2$.  In this case one has from \eqref{BT4} that
\begin{equation}\label{BZ4}
|g(x,\tau,t)|  \ \le \ \frac{C_1}{T_1-t}\int_{1-2\tau/(T_1-t)}^{1-2\tau/(T_2-t)}\frac{\dx{z}}{\sqrt{1+z}} \ , \quad\text{where $C_1$ is a constant.}
\end{equation}
The integration in \eqref{BZ4} can be computed to give
\begin{equation}\label{CA4}
\left[2-\frac{2\tau}{(T_2-t)}\right]^{1/2}-\left[2-\frac{2\tau}{(T_1-t)}\right]^{1/2} \ \le  \ \frac{(T_2-T_1)\sqrt{2(T_1-t)}}{(T_2-t)\sqrt{T_1-t-\tau}} \ ,
\end{equation}
whence \eqref{BY4} follows. Now \eqref{BY4} implies that
\begin{equation}\label{CB4}
\int^{(T_1-t)}_{(T_1-t)/2}|g(x,\tau,t)| \dx{\tau} \ \le \ \frac{C(T_2-T_1)}{T_2-t} \quad \text{for some constant $C$.}
\end{equation}
 Letting $h(x,t,s,T)$ denote the integrand of the integral on the RHS of \eqref{BQ4}, it follows from \eqref{BX4}, \eqref{CB4} that
 \begin{equation}\label{CC4}
\int_{T_1}^{T_2}\dx{T} \int_{0<t<s<T_1}\dx{t} \dx{s} \ |h(x,t,s,T)| \ \le \  C(T_2-T_1)[|\log(T_2-T_1)|+1]
 \end{equation}
 for some constant $C$.

 We consider next the situation where $0<t<T_1$ and $T_1<s<T$ and proceed similarly to the previous paragraph. The integration with respect to $T$ in \eqref{BS4} is now replaced by integration over the interval $s<T<T_2$.   Correspondingly, in the definition of the function $g$ the lower integration limit $1-2\tau/(T_1-t)$ in \eqref{BT4} is replaced by $-1$, while the upper limit is still given by $1-2\tau/(T_2-t)$. Since $T_1-t<\tau<T_2-t$, it follow that the maximum value of the upper limit is $-1+2(T_2-T_1)/(T_2-t)$. Assuming that $(T_2-T_1)/(T_2-t)\le 1/2$, we have from \eqref{BT4} that
  \begin{equation}\label{CD4}
 |g(x,\tau,t)| \ \le \  |g(0,\tau,t)| \ = \
  \frac{1}{\pi\sqrt{\tau(T_2-t)}}\left[1-\frac{\tau}{T_2-t}\right]^{1/2} \ .
 \end{equation}
 It is easy to see from \eqref{CD4} that
 \begin{equation}\label{CE4}
 \int_{T_1-t}^{T_2-t} |g(x,\tau,t)| \dx{\tau} \ \le \  \frac{C(T_2-T_1)^{3/2}}{(T_2-t)^{3/2}}  \quad\text{for some constant $C$.}
 \end{equation}
 
 In the case when $(T_2-T_1)/(T_2-t)> 1/2$ we can easily see from (\ref{CD4}) that if $x=0$ then the LHS of (\ref{CE4}) is bounded by a constant. However when $x>0$ the inequality in (\ref{CD4}) only holds for $\tau\ge (T_2-t)/2$.  It does not hold for $T_1-t<\tau<(T_2-t)/2$ since the integrand of (\ref{BT4}) changes sign in the interval $-1<z<1-2\tau/(T_2-t)$.   To bound $g(x,\tau,t)$ in this case we consider the situation when $\tau\le x^2$. For $n\ge 1$ an integer and $x^2/(n+1)<\tau\le x^2/n$ we have that
 \begin{equation} \label{DA4}
 |g(x,\tau,t)| \ \le \ \frac{1}{2\pi\tau}\int_{-1}^1 \frac{|z|}{\sqrt{1-z^2}}\exp\left[-\frac{n(1-z)}{4(1+z)}\right] \dx{z} \ \le \ \frac{C}{\tau\sqrt{n}}  \ ,
 \end{equation}
for some constant $C$. Hence the integral of $|g(x,\tau,t)|$ over the interval $0<\tau<x^2$ is bounded by a constant independent of $x$. To estimate $|g(x,\tau,t)|$ for $\tau>x^2$ we split the integral representation and use the cancellation properties of the function $z\ra z/\sqrt{1-z^2}$. Thus for $n\ge 1$ an integer and $nx^2<\tau<(n+1)x^2$  let $\al(n)=\min\{1-1/n, 1- 2\tau/(T_2-t)\}$. Then we have that
 \begin{equation} \label{DB4}
 |g(x,\tau,t)| \ \le \
  \frac{1}{2\pi\tau}\int_{\{-1<\tau<-\al(n), \ \al(n)<\tau<1\}} \frac{|z|}{\sqrt{1-z^2}} \dx{z} +
 \frac{x^2}{8\pi\tau^2}\int_{-\al(n)}^{\al(n)} \frac{\dx{z}}{(1+z)^{3/2}}   \ .
 \end{equation}
The first integral on the RHS of \eqref{DB4} is bounded by $C\left[1/\tau\sqrt{n}+1/\sqrt{\tau(T_2-t)}\right]$, and the second by $Cx^2\sqrt{n}/\tau^2$ for some constant $C$. Therefore the integral of $|g(x,\tau,t)|$ over the interval $\min\{x^2,T_2-t\}<\tau<T_2-t$ is bounded by a constant independent of $x$.  We have shown then that  the integral of $|g(x,\tau,t)|$ over the interval $0<\tau<T_2-t$ is bounded by a constant independent of $x$, whence \eqref{CE4} holds for $(T_2-T_1)/(T_2-t)> 1/2$. It easily follows from \eqref{CE4} that
 \begin{equation}\label{CF4}
\int_{T_1}^{T_2} \int_{0<t<T_1<s<T} |h(x,t,s,T)| \dx{t} \dx{s}  \dx{T} \ \le \  C(T_2-T_1)
 \end{equation}
 for some constant $C$.

 Finally we consider the situation where $T_1<t<T_2$ and $t<s<T$. The function $g(x,\tau,t)$ is defined as in the previous paragraph but now we need to estimate the integral of $g(x,\tau,t)$ over the interval $0<\tau<T_2-t$. We have already established that  the integral is bounded by a constant independent of $x$. This implies that
 \begin{equation}\label{CH4}
\int_{T_1}^{T_2} \int_{T_1<t<s<T}\ |h(x,t,s,T)| \dx{t} \dx{s}  \dx{T} \ \le \  C(T_2-T_1)
 \end{equation}
 for some constant $C$.  The inequality \eqref{BR4} follows now from \eqref{CC4}, \eqref{CF4}, \eqref{CH4}.
\end{proof}

\bigskip

\begin{proof}[Proof of Lemma~\ref{lem:Schauder3}]
For the proof, we introduce the notation $\varTheta_0(t) =\exp[-V(0)+\theta(T)W(0)]$. From \eqref{N4}, \eqref{Y4} we have that
\begin{equation}\label{CJ4}
c_{0,\infty}(x,T) \ = \ \frac{x}{\sqrt{4\pi}}\int_0^T \frac{1}{(T-t)^{3/2}}\exp\pra*{[-\frac{x^2}{4(T-t)}} \, \varTheta_0(t) \dx{t}\ .
\end{equation}
Hence we have that
\begin{align}
\frac{c_{0,\infty}(x,T)-c_{0,\infty}(0,T)}{x} &= \frac{1}{\sqrt{4\pi}}\int_0^T \!\! \frac{1}{(T-t)^{3/2}}\exp\pra*{-\frac{x^2}{4(T-t)}} \;  \bra[\big]{ \varTheta_0(t) - \varTheta_0(T)}  \dx{t} \notag \\
&\quad -\frac{\varTheta_0(T)}{\sqrt{4\pi}}\int_{-\infty}^0 \frac{1}{(T-t)^{3/2}}\exp\pra*{-\frac{x^2}{4(T-t)}} \dx{t} \ .  \label{CK4}
\end{align}
It is evident from \eqref{CK4} that if $\theta(\cdot)$ is H\"{o}lder continuous at $T$ of order  $\ga>1/2$ then
\begin{equation}\label{CL4}
\lim_{x\ra 0} \frac{c_{0,\infty}(x,T)-c_{0,\infty}(0,T)}{x} \ = \ \frac{\pa c_{0,\infty}(0,T)}{\pa x} \quad\text{exists.}
\end{equation}
To prove the continuity of $\pa c_{0,\infty}(x,T)/\pa x$ as $x\ra 0$ we observe from \eqref{CJ4}  that
\begin{equation}\label{CM4}
\frac{\pa c_{0,\infty}(x,T)}{\pa x} \ = \ f(x,T)-\frac{\sqrt{2}\; \varTheta_0(T)}{x\sqrt{\pi}}\int_0^{x/\sqrt{2T}} (1-z^2)e^{-z^2/2}  \dx{z}  \ ,
\end{equation}
where $f(x,T)$ is given by the formula 
\begin{equation}  \label{CN4}
f(x,T) =
 \frac{1}{\sqrt{4\pi}}\int_0^T\! \bra*{1-\frac{x^2}{2(T-t)}} \frac{1}{(T-t)^{3/2}}\exp\pra*{-\frac{x^2}{4(T-t)}} \, \bra[\big]{ \varTheta_0(t) - \varTheta_0(T)} \dx{t} \ .
\end{equation}
We see from \eqref{CM4}, \eqref{CN4},  similarly to \eqref{CL4}, that $\lim_{x\ra 0}\pa c_{0,\infty}(x,T)/\pa x$ exists provided $\theta(\cdot)$ is H\"{o}lder continuous at $T$ of order  $\ga>1/2$.

To prove \eqref{CI4} when $\ga<1/2$,  we first note from \eqref{CN4} that
\begin{equation}\label{CO4}
f(x,T) \ = \ \int_0^T g(x,t,T) \dx{t} \ , \quad\text{where } |g(x,t,T)|\le C/(T-t)^{\ga-3/2}
\end{equation}
for some constant $C$.  Evidently  for $0<T_1<T_2$, 
\begin{equation}\label{CP4}
\int_{T_1}^{T_2} \int_{0}^{T_1} (T-t)^{\ga-3/2}\dx{t}  \dx{T} \ \le \ C(T_2-T_1)^{\ga+1/2}
\end{equation}
for some constant $C$, so we are left to estimate the integral of $g(x,t,T)$ over the region $\{T_1<t<T, \ T_1<T<T_2\}$.  Letting $h(x,s)$ be the function
\begin{equation}\label{CQ4}
h(x,s) \ = \ \frac{1}{\sqrt{4\pi}}\left\{1-\frac{x^2}{2s}\right\}\frac{1}{s^{3/2}}\exp\pra*{-\frac{x^2}{4s}} \ , 
\end{equation}
we see that
\begin{equation} \label{CR4}
\int_{T_1}^{T_2} \int_{T_1}^{T} g(x,t,T)  \dx{t} \dx{T} = \int_{T_1}^{T_2} \varTheta_0(T) \bra*{\int_0^{T_2-T}h(x,s) \dx{s}-\int_0^{T-T_1}h(x,s) \dx{s} } \dx{T} \ .
\end{equation}
It follows from \eqref{CR4} that
\begin{align}
\int_{T_1}^{T_2} \int_{T_1}^{T} g(x,t,T)  \dx{t} \dx{T} &= \ \int_{T_1}^{(T_1+T_2)/2}  \varTheta_0(T)\int_{T-T_1}^{T_2-T}h(x,s) \dx{s} \dx{T} \nonumber \\
&\qquad -\int^{T_2}_{(T_1+T_2)/2}  \varTheta_0(T)\int^{T-T_1}_{T_2-T}h(x,s) \dx{s} \dx{T}  \ .\label{CS4}
\end{align}
Making the change of variable $T=(T_1+T_2)/2-\tau$ in the first integral on the RHS of \eqref{CS4}, and  $T=(T_1+T_2)/2+\tau$ in the second integral, we see that
\begin{multline}\label{CT4}
\int_{T_1}^{T_2} \int_{T_1}^{T} g(x,t,T)  \dx{t} \dx{T} \ = \ \int_0^{(T_2-T_1)/2} \ \Big\{\varTheta_0\bra[\big]{(T_1+T_2)/2-\tau}   \\
 -\varTheta_0\bra[\big]{(T_1+T_2)/2+\tau}\Big\}  \int_{(T_2-T_1)/2-\tau}^{(T_2-T_1)/2+\tau}h(x,s) \dx{s} \dx{\tau} \ .
\end{multline}
Using the fact that
\begin{equation}\label{CU4}
\int_0^{(T_2-T_1)/2} \int_{(T_2-T_1)/2-\tau}^{(T_2-T_1)/2+\tau} \frac{\dx{s}}{s^{3/2}} \dx{\tau} \ \le \ C(T_2-T_1)^{1/2}
\end{equation}
for some constant $C$, we conclude from \eqref{CT4} that
\begin{equation}\label{CV4}
\left|\int_{T_1}^{T_2} \int_{T_1}^{T} g(x,t,T)  \dx{t} \dx{T}\right| \ \le \ C(T_2-T_1)^{\ga+1/2}
\end{equation}
for some constant $C$.
\end{proof}

\subsection{Parabolic regularization (Lemma~\ref{lem:stability:sup}) and tightness (Lemma~\ref{lem:tight})}

\begin{proof}[Proof of Lemma~\ref{lem:stability:sup}]
It follows from  Proposition~\ref{prop:ShortTimeExistence}  that  the drift $b$ for \eqref{A1} is uniformly Lipschitz. Thus there is a constant $A$ depending on $\|\theta(\cdot)\|_\infty$ such that
\begin{equation}\label{AK3}
\left|\pa b(x,t)/\pa x\right| \ \le \ A \quad\text{for } x>0, \ 0\le t\le T_0 \ .
\end{equation}
We note that if $\sup c(\cdot,0)<\infty$  then   $\sup  c(\cdot,t)<\infty$ for $t>0$. This follows by the maximum principle from \eqref{E1}, \eqref{AK3}.  In fact we have that
\begin{equation}\label{AL3}
\sup c(\cdot,t) \ \le \ e^{At}\sup c(\cdot,0)+\sup_{0<s<t} \big\{\exp[A(t-s)-V(0)+\theta(s)W(0)] \big\}\ .
\end{equation}
Hence to establish boundedness of $\sup c(\cdot,t), \ t_0\le t\le T_0$, it will be sufficient to show that if $c(\cdot,0)$ satisfies \eqref{AJ2}, then $\sup c(\cdot,T)<\infty$ for $T>0$ sufficiently small.

Let $x_0\ge 3$ and $x(\cdot)$ be the solution to the ODE terminal value problem
\begin{equation}\label{AO3}
\pderiv{x(t)}{t} \ = \ b(x(t),t) \ , \quad t<T, \ \  x(T)=x_0 \ .
\end{equation}
Then if $w_{x_0}(x,t), \ x>0, \ t<T,$ denotes the solution to \eqref{A2}, \eqref{B2} with $w_0(x)=\del(x-x_0),\ x>0$, we show that for $T>0$ sufficiently small,
\begin{equation}\label{AM3}
\sup_{x>0}w_{x_0}(x,0)\le C_1(T),  \qquad \sup_{\{|x-x(t)|\ge 1, \ 0\le t<T\}}w_{x_0}(x,t)\le C_2(T),
\end{equation}
where $C_1(T), \ C_2(T) $ are constants depending on $T$, but not on $x_0$.  To see this we approximate the Dirac delta function at $x_0$ by bounded functions. Thus let $\phi:\R\ra\R$ be a continuous function with compact support in the interval $[-1,1]$ and with integral equal to $1$. We then take
\begin{equation}\label{AN3}
w_0(x) \ = \ \frac{1}{\ve}\phi\left(\frac{x-x_0}{\ve}\right) \ , \quad x>0, \ 0<\ve\le 1 \ ,
\end{equation}
so the function on the RHS converges to $\del(x-x_0)$ as $\ve\ra 0$.
Next we use the representation \eqref{AW2}, \eqref{AX2} for solutions to \eqref{A2}, \eqref{B2}. If $X(\cdot)$ is a solution to \eqref{AW2} then $Y(\cdot)=X(\cdot)-x(\cdot)$ is a solution to the SDE
\begin{equation}\label{AP3}
dY(t) \ = \ [b(Y(t)+x(t),t)-b(x(t),t)] \dx{t} +\sqrt{2} \dx{B(t)} \  .
\end{equation}
Hence the representation \eqref{AX2} yields $w(x,0)=\tilde{w}(x-x(0),0)$, where
\begin{equation}\label{AQ3}
\tilde{w}(y,0) \ = \ \EX\pra[\bigg]{ \frac{1}{\ve}\phi\left(\frac{Y(T)}{\ve}\right)  \ ; \  \inf_{0<s<T}[Y(s)+x(s)]>0 \  \bigg| \ Y(0)=y\  } \ .
\end{equation}

We assume now that $0<\ve<1/2$, whence $\phi(Y(T)/\ve)\ne 0$ only if $|Y(T)|<1/2$. We consider  paths $Y(s), \ 0\le s\le T,$ such that $Y(0)=y$ and $|y|<2$. Then for $T$ sufficiently small independent of $x_0\ge 3,$ the expectation on the RHS of \eqref{AQ3} can be written as a sum of expectations over paths
$Y(\cdot)$ such that $\sup_{0\le s\le T} |Y(s)|<2$ and  $\sup_{0\le s\le T} |Y(s)|\ge 2$. For the paths with  $\sup_{0\le s\le T} |Y(s)|\ge 2$ which contribute to the expectation \eqref{AQ3}, there exists a stopping time $\tau_y, \ 0<\tau_y<T$, such that $|Y(\tau_y)|=1$ and $|Y(s)|<2, \ \tau_y\le s\le T$. We can use this decomposition of paths to obtain a representation of  $\tilde{w}(y,0)$ in terms of a Dirichlet Green's function on the interval $[-2,2]$.  Thus we consider the terminal value problem
\begin{multline} \label{AR3}
\frac{\pa u(y,t)}{\pa t} +  \frac{\pa^2 u(y,t)}{\pa y^2} +\tilde{b}(y,t)\frac{\pa u(y,t)}{\pa y}=0 \ , \quad |y|<2, \ t<T, \\
{\rm where \ } \tilde{b}(y,t) \ = \ b(y+x(t),t)-b(x(t),t) \ .
\end{multline}
The solution  $u(y,t)$ to \eqref{AR3}
with terminal and boundary conditions given by
\begin{equation}\label{AS3}
u(y,T) \ = \ u_0(y), \  |y|<2, \quad u(y,t) \ = \ 0, \ |y|=2, \ t<T,
\end{equation}
has the representation
\begin{equation}\label{AT3}
u(y,t) \ = \ \int_{-2}^2 G_{D,x_0}(y,y',t,T)u_0(y') \dx{y'} \  ,
\end{equation}
in terms of the Dirichlet Green's function $G_{D,x_0}$, which depends through $\tilde{b}(\cdot,\cdot)$ on $x_0$.  Hence we have that
\begin{multline} \label{AU3}
\tilde{w}(y,0) \ = \ \int_{-2}^2 G_{D,x_0}(y,y',0,T) \frac{1}{\ve}\phi\left(\frac{y'}{\ve}\right)  \dx{y'} \\
+ \EX\pra*{   \int_{-2}^2 G_{D,x_0}(Y(\tau_y),y',\tau_y,T) \frac{1}{\ve}\phi\left(\frac{y'}{\ve}\right)  \dx{y'}  } \ .
\end{multline}

From \eqref{AK3}, \eqref{AR3} it follows that $|\tilde{b}(y,t)|\le A|y|, \ 0\le t\le T_0$, whence the series expansion \eqref{N2} for $G_{D,x_0}$ converges for $T$ small and independent of $x_0$.  This allows us to take the limit $\ve\ra 0$ in \eqref{AU3} to obtain the representation
\begin{equation}\label{AV3}
w_{x_0}(x,0) \ = \ G_{D,x_0}(x-x(0),0,0,T) + \EX\pra*{   G_{D,x_0}(Y(\tau_{x-x(0)}),0,\tau_{x-x(0)},T)   } \ ,
\end{equation}
valid for $\quad |x-x(0)|<2$. It follows from \eqref{P2}, \eqref{AV3} that $\sup_{|x-x(0)|<2} w_{x_0}(x,0)\le C_1(T)$, where $C_1(T)$ is independent of $x_0\ge 3$. Since we have an analogous  representation to \eqref{AV3} for $w_{x_0}(x,0)$  when $|x-x(0)|>2$, we conclude the first inequality of \eqref{AM3}. The second inequality of \eqref{AM3} follows by a similar argument.

To show that $\sup c(\cdot,T)<\infty$, we observe by the continuity of the function $x\ra c(x,T), \ x\ge 0,$ that it is sufficient to obtain a uniform estimate on $c(x_0,T)$ for $x_0\ge 3$. To do this we use the representation \eqref{AI2}. Thus we have that
\begin{equation}\label{AW3}
c(x_0,T) \ =  \  \int_0^\infty w_{x_0}(x,0)c(x,0) \dx{x}+\int_0^Tdt \ \frac{\pa w_{x_0}(0,t)}{\pa x} \exp[-V(0)+\theta(t)W(0)] \ .
\end{equation}
 We see from the first inequality of \eqref{AM3} and \eqref{AJ2}  that the first term on the RHS of \eqref{AW3} is uniformly bounded for $x_0\ge 3$. The second inequality of \eqref{AM3} and Lemma~\ref{lem:uniformFluxBound} imply that the second term on the RHS of \eqref{AW3} is uniformly bounded for $x_0\ge 3$.

 To prove equicontinuity of the functions $c(\cdot,t), \ t_0\le t\le T_0$ on $[0,L]$ we argue as in Lemma~\ref{lem:reg:C1}. In particular, using the notation of \eqref{D4} we have for $0<x<L+1$, $t_0<T\le T_0$ that
\begin{align}
c(x,T) &= \int_0^{L+1} G_{D,L+1}(x',x,t_0,T)c(x',t_0) \dx{x'} \notag  \\
&\quad + \int_{t_0}^T \frac{\pa G_{D,L+1}(0,x,t,T)}{\pa x'} \ e^{-V(0)+\theta(t)W(0)}\dx{t} \label{BW3}\\
&\quad - \int_{t_0}^T \frac{\pa G_{D,L+1}(L+1,x,T',T)}{\pa x'}c(L+1,T')\dx{T'} \ . \notag
\end{align}
Since $c(\cdot,t_0)$ is continuous, the equicontinuity follows from \eqref{BW3} and the properties of $G_{D,L+1}$ already established in $\S2$.
\end{proof}

\medskip

The proof of the tightness Lemma~\ref{lem:tight} is based on a more quantified version of Lemma~\ref{lem:adjoint:Wbound}.
\begin{lemma}\label{lem:adjoint:tight0}
Under the Assumption of Lemma~\ref{lem:tight}. For $M\ge0$  let $w_M(x,t), \ x>0,t<T,$ be the solution to \eqref{A2}, \eqref{B2} with $w_0(x)=0, \ 0<x<M, \ w_0(x)=W(x), \ x>M$.  Then there exists $C_\infty>0$, independent of $M,T$,   such that $w_M(x,0)\le C_\infty W(x)$ for $x\ge 0$. For any $x_0,\ve>0$ there exists $M(\ve,x_0)>0$, independent of $T$, such that
\begin{equation}\label{CE3}
w_M(x,0) \ \le \ \ve \quad\text{for } 0<x\le x_0, \ M\ge M(\ve,x_0) \ .
\end{equation}
\end{lemma}
\begin{proof}
To show that $w_M(x,0)\le C_\infty W(x), \ x>0,$ we proceed  as in Lemma~\ref{lem:adjoint:Wbound}. Choosing $\del>0$ to satisfy $\del\le|\theta_\infty|/3$, we see from \eqref{F1:2} that \eqref{BB2} holds with $x_0=x_\del$, upon modifying the first inequality of \eqref{BB2} to $\mu(y,t)\le-6\sig^2(y)\theta_\infty$.  The remaining argument of Lemma~\ref{lem:adjoint:Wbound} similarly goes through.

To prove \eqref{CE3}  we  observe   from the first inequality of \eqref{F1:2} that there exists $x_1>0$ such that $b(x,t)\le 0$ for $x\ge x_1, \ 0\le t\le T$. Now let $L>x_0$ and $\tau^L_{x,t}>t$ be the first exit time from the interval $[0,L]$ of the diffusion $X(\cdot)$ of \eqref{AW2} started at  $X(t)=x$ with $0<x<L$. It follows that for any $\nu>0$ there exists $L_\nu>x_0$ such that $\Prob\bra[\big]{X(\tau_{x,t}=L)}<\nu, \ 0<x<x_0,$ provided $L\ge L_\nu$.  Similarly to \eqref{AX2} we have that
\begin{equation}\label{CF3}
w_M(x,0) \ = \ \EX\pra*{w_M(L,\tau_{x,t}) \ ; \ \tau_{x,t}<T, \ X(\tau_{x,t})=L \ } \ , \quad 0<x<L, \ M>L \ .
\end{equation}
Since \eqref{CF3} implies that $w_M(x,0)\le \Prob(X(\tau_{x,t})=L)\sup W(\cdot)$, the inequality \eqref{CE3} follows provided  $\sup W(\cdot)<\infty$. In the case when $\sup W(\cdot)=\infty$ we use the transformation $Y=\Phi(X)$ introduced in the proof of Lemma~\ref{lem:adjoint:Wbound}. We have similarly to \eqref{BA2} that there exists $M_\del>0$ such that for $x>0$, $M\ge M_\del$ it holds
\begin{multline} \label{CG3}
w_M(x,0) \le  \EX\pra[\Big]{ Y(T);  \ \inf_{0\le s\le T} Y(s)>\Phi(0),\ Y(T)\ge \Phi(M) \ \Big| \ Y(0)=\Phi(x) \ }  , \ \ .
\end{multline}
We also have similarly to \eqref{BG2} that
\begin{align} \label{CH3}
\MoveEqLeft{\EX\pra[\Big]{ Y(T);  \ \inf_{0\le s\le T} Y(s)>\Phi(0), \ Y(T)\ge \Phi(M) \ \Big| \ Y(0)=y \ } }\\
&\le  \phi(y)\sum_{k=0}^\infty \frac{k+1+\Phi(M)}{\phi(k+\Phi(M))}  \  .
\end{align}
Since $\lim_{M\ra\infty}\Phi(M)=\lim_{M\ra\infty}W(M)=\infty$, we see from \eqref{BC2} that there exists $M(\ve,x_0)>0$ such if $M\ge M(\ve,x_0)$ then the RHS of \eqref{CH3} is less than $\ve$ provided $\Phi(0)<y<\Phi(x_0)$. The inequality\eqref{CE3} follows.
\end{proof}

\medskip

Lemma~\ref{lem:adjoint:tight0} gives us enough to control on the adjoint problem to prove \eqref{eq:tight} of Lemma~\ref{lem:tight} by means of the representation \eqref{AI2}.
\begin{proof}[Proof of Lemma~\ref{lem:tight}]
For any $M>0$ the identity \eqref{AI2} implies that
\begin{equation}\label{CJ3}
\int_M^\infty W(x)c(x,T) \dx{x} \ = \ \int_0^\infty w_M(x,0) c(x,0) \dx{x}+\int_{0}^T\frac{\pa w(0,t)}{\pa x} \; e^{-V(0)+\theta(t)W(0)} \dx{t} \ .
\end{equation}
We can bound the first term on the RHS of \eqref{CJ3} by using  Lemma~\ref{lem:adjoint:tight0}. Thus we have that
\begin{equation}\label{CK3}
\int_0^\infty w_M(x,0)c(x,0) \dx{x} \ \le \ \frac{\ve}{W(0)}\int_0^\infty W(x) c(x,0) \dx{x}+C_\infty\int_{x_0}^\infty W(x)c(x,0) \dx{x} \ ,
\end{equation}
provided $M>M(\ve,x_0)$. Evidently by choosing $x_0$ sufficiently large and $\ve$ sufficiently small the RHS of \eqref{CK3} can be made smaller than $\del/2$.  Hence there exists $M(\del)>0$ such that if $M\ge M(\del)$ the first term on the RHS of \eqref{CJ3} is less than $\del/2$.

We use Lemma~\ref{lem:adjoint:fluxBound} to estimate the second term on the RHS of \eqref{CJ3}. Thus there exists $\del_\infty>0$, independent of $M,T$, such that
\begin{equation}\label{CL3}
0 \ \le \  \frac{\pa w_M(0,t)}{\pa x} \ \le \   C_\infty e^{-\del_\infty(T-t)} \quad\text{for } 0\le t\le T-1.
 \end{equation}
 We also have from Lemma~\ref{lem:adjoint:tight0} that
 \begin{equation}\label{CM3}
 w_M(x,t) \ \le \ \ve \  , \ \  0<x<1, \  0<t<T, \quad\text{for } M>M(\ve,1) \ .
 \end{equation}
 We conclude from \eqref{CM3} and Lemma~\ref{lem:uniformFluxBound} that
 \begin{equation}\label{CN3}
 0 \ \le \  \frac{\pa w_M(0,t)}{\pa x} \ \le   \ C\ve\max\left[\frac{1}{\sqrt{T-t}},1\right] \ , \quad 0<t<T \ .
 \end{equation}
 It follows from \eqref{CL3}, \eqref{CN3} that there exists $M(\del)>0$ such that if $M\ge M(\del)$ the second term on the RHS of \eqref{CJ3} is less than $\del/2$. We have proven that the inequality \eqref{eq:tight} holds for $T\ge 0$.
\end{proof}

\bigskip

\section{Convergence to equilibrium}\label{s:Conv}
In this section we will prove convergence to equilibrium for the system \eqref{A1}, \eqref{B1}, \eqref{D1}, \eqref{E1}. Let us recall some basic notion of the convergence of measures: A sequence $\mu_n, \ n=1,2,..,$ of measures  in $\mathcal{M}(\R^+)$ converges \emph{weakly} to a measure $\mu\in \mathcal{M}(\R^+)$ if
\begin{equation}\label{A3}
\lim_{n\ra\infty} \int f \dx\mu_n  \ = \  \int f \dx\mu \quad\text{for all } f\in \cC_0(\R^+) \ .
\end{equation}
Hence weak convergence for measures corresponds to weak$^*$ convergence in $\mathcal{M}(\R^+)$. In the following we will work exclusively with absolutely continuous measure satisfying a moment condition with respect to $W$ and it is convenient to introduce the space of densities with finite $W$-moment
\begin{equation}\label{e:def:L1W}
  L^1_W(\R^+) = \set*{ f \in L^1(\R^+) : f\geq 0, \; \norm{f}_{L^1_W} \! < \infty } \quad\text{with}\quad  \norm{f}_{L^1_W} \!= \int \! f(x) \, W(x) \dx{x}.
\end{equation}
We identify functions in $L^1_W(\R^+)$ with measures in $\mathcal{M}(\R^+)$ by their densities and speak just of weak convergence in the sense of~\eqref{A3}. Especially, lower semicontinuity is understood as sequentially lower semicontinuity with respect to the weak convergence in the sense of~\eqref{A3}.

\subsection{Characterization of constrained minimizer of free energy}
We start by defining the free energy and stating its properties.
\begin{lemma}\label{lem:CharactMinEnergy}
  Suppose $V,W$ satisfy Assumption~\ref{ass:VW}. For $c\in L^1_W(\R^+)$ and $\theta\in \R$ the unconstrained free energy $\cG: L^1_W(\R^+) \times \R \to \R^+$ is defined by
  \begin{equation}\label{e:def:G}
    \cG(c,\theta) = \int c(x) \bra*{\log c(x)  -1} \dx{x}  + \int V(x) c(x) \dx{x} + \frac{1}{2} \theta^2 + \cG_0,
  \end{equation}
  where $\cG_0 = \int e^{-V(x)} \dx{x} < \infty$.
  The function $\cG$ is for any $M>0$ lower semi-continuous on  $U_M = \set[\big]{ (c,\theta) \in L^1_W(\R^+) \times \R: \norm{c}_{L^1_W(\R^+)} \leq M}$ and for any $M,N>0$ continuous on
  \begin{equation*}
   \tilde U_{M,N}\subset \set*{ (c,\theta)\in C(\R^+) \times \R: \norm{c}_{L^1_W(\R^+)} \leq M, \norm{c}_\infty \leq N } ,
  \end{equation*}
  such that $\forall L > 0: \set*{ c |_{[0,L]}: (c,\theta) \in \tilde U_{M,N}}$ is equicontinuous.\\[0.25\baselineskip]
  The unconstrained minimizer is uniquely given by
  \begin{equation}\label{e:unconstrainedMin}
    \argmin\set*{ \cG(c,\theta) : c\in L^1_W(\R^+), \theta \in \R }  = \bra*{c_{0}^\eq,0} .
  \end{equation}
  Moreover, the constrained minimizer is uniquely given by
  \begin{equation}\label{e:constrainedMin}
    \argmin\set*{ \cG(c,\theta) : c\in L^1_W(\R^+), \theta+ \int W(x) c(x) \dx{x} = \rho } = \bra*{c_{\theta_\eq}^\eq,\theta_\eq}
  \end{equation}
  where $\theta_\eq = \theta_\eq(\rho)$ is given for $\rho \geq \rho_s= \int W(x) c_{0}^\eq(x) \dx{x}$ by $\theta_\eq = 0$ and for $\rho< \rho_s$ implicitly by
  \begin{equation}\label{e:lem:constraint}
    \theta_{\eq} + \int W(x) c_{\theta_\eq}^\eq \dx{x} = \rho.
  \end{equation}
\end{lemma}
\begin{proof}
Let us first note, that $\cG(c,\theta)$ is bounded from below by $0$. We rewrite it in the following way
\[
  \cG(c,\theta) = \cH(c | c_0^\eq) + \frac{1}{2}\theta^2 , \qquad\text{with}\qquad c_{0}^\eq(x) = e^{-V(x)}
\]
and the relative entropy for $c,c_0^\eq\in L^1_W(\R^+)$ is defined by
\[
  \cH(c | c_0^\eq ) = \int \Psi\bra*{ \frac{c(x)}{c_0^\eq(x)}} c_0^\eq(x) \, \dx{x} \qquad\text{with}\qquad \Psi(r) = r \log r - r + 1 .
\]
Since $\Psi$ is non-negative, we obtain the lower bound. In addition $r\mapsto \Psi(r)$ is strictly convex with convex dual $\Psi^*(s) = e^{s} - 1$ and it holds the dual variational characterization
\[
  \cH(c | c_0^\eq ) = \sup_{g\in \cC_0(\R^+)}\set*{ \int g(x) c(x) \dx{x} - \int \bra*{e^{g(x)} -1 } c_0^\eq(x) \dx{x} } .
\]
From this representation the lower semicontinuity on $U_M$ for any $M>0$ is immediate (see also~\cite{AB88,But89}).

Next we consider the function $\cG$ restricted to the set $\tilde{U}_{M,N}$ with $M,N>0$. It follows from~\eqref{e:def:G}, the bounds~\eqref{F1:1} and the lower bound $\inf W(\cdot)>0$ that $\cG(c,\theta)<\infty$ for all $(c,\theta)\in\tilde{U}_{M,N}$.  Furthermore, the energy part $\int \bra*{ V + \log a} c \dx{x}$ is continuous and we see that in order to prove continuity of $\cG$ on $\tilde{U}_{M,N}$ it is sufficient to prove continuity of the entropy $\cS(c) = \int c(x) \log c(x) \dx{x}$ on $\tilde{U}_{M,N}$. Since \eqref{G1} implies that $\lim_{x\ra\infty} W(x)=\infty$, we have that for any $\ve>0$ there exists $L_\ve>0$ such that
\begin{equation}\label{BR3}
\forall (c,\theta)\in \tilde{U}_{M,N}: \qquad \int_{\{x\geq L_\ve : \ c(x)\geq 1\}} c(x) \log c(x) \dx{x}  \ < \ \ve  \ .
\end{equation}
Now, for $\la\in\R$ let $g_\la:(0,\infty)\ra\R$ be the convex function $g_\la(z)=z\log z+\la z, \ z>0$, which has a minimum at $z=e^{-(1+\la)}$ given by $\inf_{z>0}g(z)=-e^{-(1+\la)}$. Hence, we can estimate for any $L\geq 0$ choosing $\lambda = V(x)$
\begin{equation}\label{BS3}
\int_{\{x\geq L:c(x)\leq 1\}}c(x)\log c(x) \dx{x} \ \ge \ -\int_L^\infty  V(x)  c(x) \dx{x} -e^{-1}\int_L^\infty c_0^\eq(x) \dx{x} \ .
\end{equation}
It follows now from~\eqref{F1:1} and (e) of Assumption~\ref{ass:VW} that for any $\ve>0$ there exists $L_\ve>0$ such that
\begin{equation}\label{BT3}
\int_{\{x \geq L_\ve,     \ c(x)\leq 1\}} c(x) \log c(x) \dx{x}  \ > \ -\ve  \ , \quad (c,\theta)\in \tilde{U}_{M,N} \ .
\end{equation}
Let $(c_m,\theta_m), \ m=1,2,..,$ be a sequence in $\tilde{U}_{M,N}$ which converges in the topology of  $\mathcal{M}_{ac}(\R^+)\times\R$ to $(c,\theta)\in \tilde{U}_{M,N}$, then by the equicontinuity of the sequence, it follows
\begin{equation}\label{BU3}
\lim_{m\ra\infty}\int_0^L c_{m}(x)\log c_{m}(x) \dx{x} \ = \ \int_0^L c(x)\log c(x) \dx{x} \ .
\end{equation}
The continuity of $\cS$ on $\tilde{U}_{M,N}$ follows from \eqref{BR3}, \eqref{BT3}, \eqref{BU3}.

The unique global minimizer $(c_{0}^\eq,0)$ follows by strict convexity of the functional and the just proven lower bound~$0$. To prove the constraint minimizer, we observe that the function
\[
  m(\theta) = \cG(c_\theta^\eq, \theta) = \int_0^\infty \bra*{ \theta W(x) - 1 } c_{\theta}^\eq(x) \dx{x} + \frac{1}{2} \theta^2
\]
is well defined for any $\theta\leq 0$, since $c_{\theta}^\eq(\cdot) \leq c_{0}^\eq(\cdot)$. Moreover, we have
\begin{equation}\label{e:strict:mononton:MinEnergy}
  m'(\theta) = \theta \int_0^\infty W(x)^2 c_{\theta}^\eq(x) \dx{x} + \theta < 0 \qquad\text{for } \theta <0 \ .
\end{equation}
Let $c\in L^1_W(\R^+)$ with $\cH(c| c_0^\eq) < \infty$, set $\eta  = \rho- \int_0^\infty W(x) c(x) \dx{x}$ and define the function $h:[0,1]\to \R$ by
\begin{equation*}
  h_\theta(\lambda) = \cG\bra*{ (1-\lambda) c_\theta^\eq + \lambda c, (1-\lambda) \theta + \lambda \eta} .
\end{equation*}
By convexity of $\cG$, also $h_\theta$ is convex, and we have for all $\lambda \in (0,1]$ the secant inequality
\begin{equation}\label{e:MinEnergy:Secant}
  \frac{h_\theta(\lambda) - h_\theta(0)}{\lambda} \leq \cG(c,\eta) - \cG(c_{\theta}^\eq,\theta) .
\end{equation}
We can let $\lambda\to 0$ and obtain for the relative entropy
\begin{align*}
  \MoveEqLeft{\frac{1}{\lambda} \bra[\Big]{\cH\bra*{ (1-\lambda) c_\theta^\eq + \lambda c | c_0^\eq} - \cH\bra*{c_\theta^\eq | c_0^\eq} } }\\
  &= \frac{1}{\lambda} \int \bra*{ \Psi\bra*{ \frac{(1-\lambda) c_\theta^\eq + \lambda c}{c_0^\eq}} - \Psi\bra*{ \frac{c_\theta^\eq}{c_0^\eq}} } c_0^\eq(x) \dx{x} \\
  &\overset{\lambda\to 0}{\to} \int \Psi'\bra*{ \frac{c_\theta^\eq(x)}{c_0^\eq(x)} } \bra*{ c(x) - c_\theta^\eq(x)} \dx{x} = \theta \int W(x) \bra[\big]{c(x) - c_\theta^\eq(x)} \dx{x} \  .
\end{align*}
Likewise, by using once more the constraint, we have 
\begin{align*}
  \frac{1}{\lambda} \bra*{ \frac{1}{2} \bra*{ (1-\lambda) \theta + \lambda \eta}^2 - \frac{1}{2} \theta^2} \overset{\lambda\to 0}{\to} \theta ( \eta - \theta) =\theta\bra*{  \rho - \int W(x) c(x) \dx{x} - \theta } \ .
\end{align*}
Hence, the estimate~\eqref{e:MinEnergy:Secant} becomes after passing to the limit $\lambda \to 0$
\begin{equation}\label{e:MinEnergy:Variation1}
 \theta \bra[\bigg]{ \rho - \int W(x) c_\theta^\eq(x) \dx{x}  - \theta }  \leq \cG(c,\eta) - \cG(c_{\theta}^\eq,\theta).
\end{equation}
If $\rho \leq \rho_s = \int W(x) c_0^\eq(x) \dx{x}$, then we can choose $\theta = \theta_{\eq}(\rho)$ and obtain that the LHS of~\eqref{e:MinEnergy:Variation1} is zero and the desired inequality. Now, if $\rho > \rho_s$, we show that for any $\eps>0$ there exists $c \in L^1_W(\R^+)$ and $\theta\in \R$ satisfying the constraint and $\cG(c,\theta) - \cG(c_0^\eq,0) < \eps$. To construct $c$, we use~\eqref{F1:1} combined with~\eqref{G1} of Assumption~\ref{ass:VW}, in the integrated form
\[
  \abs{V(x)} + C \leq \delta W(x) \qquad\text{for } x \geq x_\delta .
\]
We define $y_\delta \geq x_\delta$ such that
\begin{equation}\label{e:MinEnergy:Construct}
  \int_0^{x_\delta} W(x) c_0^\eq(x) \dx{x} + \int_{x_\delta}^{y_\delta} W(x) \dx{x} = \rho
\end{equation}
and set
\[
  c(x) =
  \begin{cases}
    c_0^\eq(x) &, x\in [0,x_\delta)  \\
    1 &, x\in [x_\delta, y_\delta) \\
    0 &, x \geq y_\delta
  \end{cases}.
\]
Then, by construction $\int W(x) c(x) \dx{x} = \rho$, we can calculate
\begin{align*}
  \cH(c| c_0^\eq ) &=\int_{x_\delta}^{y_\delta} \Psi\bra*{ \frac{1}{c_0^\eq(x)}} c_0^\eq(x) \dx{x} + \int_{y_\delta}^\infty c_0^\eq(x) \dx{x} \\
  &= \int_{x_\delta}^{y_\delta} \bra*{ V(x) - 1 } \dx{x} + \int_{x_\delta}^\infty c_0^\eq(x) \dx{x} \\
  &\leq \delta \int_{x_\delta}^{y_{\delta}} \bra*{ W(x) + C } + \int_{x_\delta}^\infty c_0^\eq(x) \dx{x} .
\end{align*}
Since $c_0^\eq \in L^1_W(\R^+)$ by assumption, we obtain that the second integral goes to zero as $x_{\delta} \to \infty$, which is the case as $\delta \to 0$. The first integral is bounded by $\delta C$ for some $C>0$ by the definition of~$y_\delta$ in~\eqref{e:MinEnergy:Construct} and the growth assumption of $W$~\eqref{G1} in Assumption~\ref{ass:VW}.

To prove uniqueness of the minimizer we use the inequality~\eqref{e:MinEnergy:Secant} again. But, now we calculate $h'_\theta(\lambda)$ for $\lambda \in (0,1)$, which is given by
\begin{align*}
  h_\theta'(\lambda) &= \int \Psi'\bra*{ \frac{(1-\lambda) c_\theta^\eq + \lambda c}{c_0^\eq}} \bra*{ c- c_\theta^\eq } \dx{x} + \bra[\big]{(1-\lambda) \theta + \lambda \eta }\bra*{\eta - \theta} = I + II .
\end{align*}
We can bound the first term, by noting that $\Psi'(r) = \log r$ and the elementary inequality
\begin{align*}
  \MoveEqLeft{\log\bra*{\frac{(1-\lambda) c_\theta^\eq + \lambda c}{c_0^\eq}} \bra*{ c- c_\theta^\eq }} \\
  &\leq \frac{(1-\lambda) c_\theta^\eq + \lambda c}{\lambda} \log\bra*{\frac{(1-\lambda) c_\theta^\eq + \lambda c}{c_0^\eq}} - \frac{c_{\theta}^{\eq}}{\lambda} \log\bra*{ \frac{(1-\lambda) c_\theta^\eq}{c_0^\eq}} .
\end{align*}
We obtain the bound
\begin{align*}
  I \le \frac{\cH\bra[\big]{ (1-\lambda) c_\theta^\eq + \lambda c \, | \, c_0^\eq}}{\lambda} &+ \int \frac{ (1-\lambda) c_\theta^\eq + \lambda c - c_0^\eq}{\lambda} \dx{x} \\
  & -  \int \frac{\log(1-\lambda) + \theta W}{\lambda}  c_{\theta}^{\eq} \dx{x}  \ < \  \infty ,
\end{align*}
since $c,c_\theta^\eq \in L^1_W(\R^+) \subset L^1(\R^+)$ due to $W(x) \geq W(0)>0$ by Assumption~\ref{ass:VW} and $\theta\leq 0$ and convexity of the relative entropy. Likewise, it holds by strict monotonicity of the function $\lambda \mapsto a \log\bra*{  b+ \lambda a}$ as long as $b+\lambda a >0$ the lower bound
\begin{align*}
  \log\bra*{ \frac{(1-\lambda) c_\theta^\eq + \lambda c}{c_0^\eq}} \bra*{ c- c_\theta^\eq } \geq \log\bra*{\frac{c_\theta^\eq}{c_0^\eq}}\bra*{ c- c_\theta^\eq } = \theta W \, \bra*{ c- c_\theta^\eq } ,
\end{align*}
where the inequality is strict as long as $c\ne c_\theta^\eq$. Hence, we obtain for any $c\ne c_\theta^\eq$
\[
  h_\theta'(\lambda) > \theta \int W(x) \bra*{ c- c_\theta^\eq(x)} \dx{x} + \bra*{(1-\lambda) \theta + \lambda \eta }\bra*{\eta - \theta} = \lambda \bra*{\theta - \eta}^2 \geq 0 ,
\]
by using the constraints $\eta + \int W c = \rho = \theta + \int W c_\theta^\eq$. The uniqueness from the minimizer follows now from~\eqref{e:MinEnergy:Secant} and the mean value theorem.
\end{proof}

Lemma~\ref{lem:CharactMinEnergy} justifies that the normalized constrained free energy $\cF_\rho$ from~\eqref{e:def:FreeEnergyNorm} is well-defined for any $c\in L^1_W(\R^+)$.

\subsection{The energy--energy-dissipation principle}
Let $(c(\cdot,t),\theta(t)), \ t> 0,$ be the solution to \eqref{A1}, \eqref{B1}, \eqref{D1}, \eqref{E1}. We wish to show that the function $t\ra \cF_\rho(c(t,\cdot))$ is decreasing. For the calculations it is convenient to rewrite the set of equations \eqref{A1}, \eqref{B1}, \eqref{E1} in the form
\begin{equation}\label{e:Evo:W2GF}
  \partial_t c(x,t) = \partial_x\bra*{ c(x,t) \; \partial_x \log\frac{c(x,t)}{c_{\theta(t)}^\eq(x)}} \quad\text{with b.c.}\quad \log \frac{c(0,t)}{c_{\theta(t)}^\eq(0)} = 0
\end{equation}
Then, we obtain by formal integration by parts using the above boundary condition
\begin{multline} \label{AX3}
\pderiv{}{t} \cF(c(t,\cdot)) \ = \ \int_0^\infty \frac{\pa c(x,t)}{\pa t}\bra*{\log c(x,t)+V(x)} \dx{x}+\theta(t)\frac{\dx\theta(t)}{\dx{t}} \\
= \ \int_0^\infty \frac{\pa c(x,t)}{\pa t}\log\bra*{\frac{c(x,t)}{c_{\theta(t)}^\eq(x)}} \dx{x} \ = \ -\int_0^\infty\bra*{ \partial_x \log \frac{c(x,t)}{c_{\theta(t)}^\eq(x)}}^2 c(x,t) \dx{x} \ .
\end{multline}
We wish to justify this calculation only under the assumption that the initial data satisfies \eqref{AJ2}.
\begin{lemma}\label{lem:EED}
Let $(c(\cdot,t),\theta(t)), \ 0<t\le T_0,$ be the solution of \eqref{A1}, \eqref{B1}, \eqref{D1}, \eqref{E1} with initial data satisfying \eqref{AJ2} constructed in Proposition~\ref{prop:ShortTimeExistence}.  Then for any $t_0$ satisfying $0<t_0<T_0$, the function  $(t_0, T_0) \ni t\ra \cF(c(t,\cdot))$ is continuous, decreasing and satisfies
\begin{equation} \label{AY3}
\frac{\dx{}^+}{\dx{t}} \cF\bra*{c(t) } \leq - \cD\bra*{ c(t),\theta(t) } =  -\int_0^\infty\bra*{ \partial_x \log \frac{c(x,t)}{c_{\theta(t)}^\eq(x)}}^2 c(x,t) \dx{x} \ ,
\end{equation}
where $\frac{\dx{}^+}{\dx{t}} f(t) = \limsup_{\delta \to 0} \frac{f(t+\delta)- f(t)}{\delta}$.
\end{lemma}
\begin{proof}
By Lemma~\ref{lem:stability:sup}, we have for any $t_0>0$ that $\sup c(\cdot,t)\leq N$ and hence also $\cF(c(\cdot,t))$ is well-defined for $t\in (t_0,T_0)$.
We are going to use a cut-off at $0$ and $\infty$. Let $\phi:[0,\infty)\ra\R^+$ be a $\cC^\infty$ function which has the property that $\phi(x)=1, \ 0\le x\le 1,$ and $\phi(x)=0, \ x\ge 2$. For $0<\ve\le 1$ and $L\ge 1$ we define a function $\cG_{\ve,L}(t), \ t>t_0,$  by
\begin{equation} \label{AZ3}
\cG_{\ve,L}(t) \ = \ \int_{\ve}^\infty \phi(x/L) \Psi\bra*{ \frac{c(x,t)}{c_0^\eq(x)}} c_0^\eq(x) \dx{x} +\frac{1}{2}\bra*{\rho-\int_\ve^\infty\phi(x/L)W(x)c(x,t) \dx{x} }^2 \ .
\end{equation}
Then the function $t\ra \cG_{\ve,L}(t), \ t>t_0,$ is $\cC^1$. To calculate its time derivative, we use the identity
\begin{align*}
  \MoveEqLeft{\int_\eps^\infty \phi(x/L) \underbrace{\Psi'}_{\mathclap{=\log}}\bra*{\frac{c_{\theta}^\eq(x)}{c_0^\eq(x)}} \partial_t c(x,t) \dx{x} } \\
  &= \bra*{ \varrho - \int_0^\infty W(x) c(x,t) \dx{x}} \int_{\eps}^\infty \phi(x/L) W(x) \partial_t c(x,t) \dx{x}.
\end{align*}
Now, we calculate using this identity
\begin{align}
\pderiv{}{t} \cG_{\ve,L}(t) &= \int_{\eps}^\infty \phi(x/L) \Psi'\bra*{ \frac{c(x,t)}{c_0^\eq(x)}} \partial_t c(x,t) \dx{x} \notag \\
&\quad -  \bra*{\rho-\int_\ve^\infty\phi(x/L)W(x)c(x,t) \dx{x} } \int_\eps^\infty \phi(x/L) W(x) \partial_t c(x,t) \dx{x} \notag \\
&= \int_\eps^\infty \phi(x/L) \log\bra*{\frac{c(x,t)}{c_{\theta(t)}^\eq(x)}} \partial_t c(x,t)  \dx{x} \notag \\
&\quad -  \bra*{\int_0^\eps \!\!W(x) c(x,t) \dx{x} + \int_\ve^\infty \!\!\bra*{1-\phi(x/L)}W(x)c(x,t) \dx{x}} \times \\
 &\qquad\qquad \times \int_\eps^\infty \phi(x/L) W(x) \partial_t c(x,t)  \dx{x} \notag \\
&= - \cD_{\eps,L}(t) - I^1_{\ve,L}(t)- I^2_{\ve,L}(t)+I^3_{\ve,L}(t)\left[I^4_{\ve,L}(t)+I^5_{\ve,L}(t)\right]  \ , \label{BA3}
\end{align}
where the last step follows by using $\partial_t c = \partial_x \bra*{ c \partial_x \log\frac{c}{c_{\theta}^\eq}}$ and integration by parts with $\cD_{\eps,L}(t)$, $I^1_{\ve,L}(t)$, $I^2_{\ve,L}(t)$, $I^3_{\ve,L}(t)$, $I^4_{\ve,L}(t)$, $I^5_{\ve,L}(t)$ are given by the formulas
\begin{align*}
  \cD_{\eps,L}(t) &= \int_\ve^\infty \phi(x/L) \bra*{ \partial_x \log \frac{c(x,t)}{c_{\theta(t)}^\eq}}^2 c(x,t) \dx{x}\\
I^1_{\ve,L}(t) &=  \log\bra*{\frac{c(\ve,t)}{c_{\theta(t)}^\eq(\ve)}} \  \partial_x  \log\bra*{\frac{c(\ve,t)}{c_{\theta(t)}^\eq(\ve)}} \ c(\eps,t)  \ , \\
I^2_{\ve,L}(t) &= \frac{1}{L}\int_\ve^\infty \phi'(x/L) \log \bra*{\frac{c(x,t)}{c_{\theta(t)}^\eq(x)}} \ \partial_x  \log \bra*{\frac{c(x,t)}{c_{\theta(t)}^\eq(x)}} \ c(x,t) \dx{x}  \ ,  \\
I^3_{\ve,L}(t) &=  \int_0^\ve W(x)c(x,t) \dx{x}+\int_\ve^\infty [1-\phi(x/L)]W(x)c(x,t) \dx{x} \ , \\
I^4_{\ve,L}(t) &= W(\ve) \ \partial_x  \log\bra*{\frac{c(\ve,t)}{c_{\theta(t)}^\eq(\ve)}} \ c(\eps,t) \ ,  \\
I^5_{\ve,L}(t) &= \int_\ve^\infty \bra*{\phi(x/L)W(x)}' \ \partial_x \log\bra*{ \frac{c(x,t)}{c_{\theta(t)}^\eq(x)}} \ c(x,t)  \dx{x} \ .
\end{align*}
To prove \eqref{BA3} we assume first that  $b(x,t)=\partial_x \log c_{\theta(t)}^\eq(x) = \theta(t)W'(x)-V'(x)$ has sufficient regularity so that $c(\cdot,\cdot)$ is a classical solution to \eqref{A1}, \eqref{E1}.  It is clear in this case that the function  $t\mapsto \cG_{\ve,L}(t)$ is $\cC^1$, and upon using \eqref{D1} that
\begin{align} \label{BG3}
\pderiv{}{t} \cG_{\ve,L}(t) &\ = \ \int_\ve^\infty \phi(x/L)\frac{\pa c(x,t)}{\pa t}\log \bra*{\frac{c(x,t)}{c_{\theta(t)}^\eq(x)} } \dx{x} \\
&\qquad - \ I^3_{\ve,L}(t)\int_\ve^\infty \phi(x/L)W(x) \pa_t c(x,t) \dx{x} \ . \notag
\end{align}
From \eqref{e:Evo:W2GF} we see upon integration by parts in $x$, that the first term on the RHS of~\eqref{BG3} is identical to the sum of the first term on the RHS of \eqref{BA3} minus  $I^1_{\ve,L}(t)$ and $I^2_{\ve,L}(t)$. Again integrating by parts using \eqref{e:Evo:W2GF} we see that the coefficient of $I^3_{\ve,L}(t)$ in \eqref{BG3} is equal to $-[I^4_{\ve,L}(t)+I^5_{\ve,L}(t)]$. Once we have the formula \eqref{BA3} we can remove the extra regularity assumption on $b(\cdot,\cdot)$ since by Lemma~\ref{lem:reg:C1} we see that the RHS of \eqref{BA3} is continuous in $t$ with just the assumption \eqref{C2}.

We observe next from \eqref{AZ3}  that $\lim_{\ve\ra 0} \cG_{\ve,L}(t)=\cG_{0,L}(t)$ exists and is given by the integral on the RHS of \eqref{AZ3} with $\ve=0$. Similarly  Theorem~\ref{thm:C1} implies that we may let $\ve\ra 0$ in \eqref{BA3} to conclude that the function $t\mapsto \cG_{0,L}(t), \ t>t_0,$ is $\cC^1$ and its derivative is given by the RHS of \eqref{BA3} with $\ve=0$.  Note from \eqref{E1} that $\lim_{\ve\ra 0} I^1_{\ve,L}(t)=I^1_{0,L}(t)=0$. Let $0<t_0<t_1<t_2$ such that $c(\cdot,t)$ is a solution for $t\in (t_1,t_2)$ as constructed in Proposition~\ref{prop:ShortTimeExistence}. We shall show that for some constant $C$,
\begin{equation}\label{BH3}
\sup_{t_1\le t\le t_2} |I^j_{0,L}(t)| \ \le C \ , \quad j=2,3,4,5, \ L\ge 1 \ ,
\end{equation}
and also that
\begin{equation}\label{BI3}
\lim_{L\ra\infty} I^j_{0,L}(t) \ = \ 0 \ , \quad j=2,3, \  t_1<t<t_2 \ .
\end{equation}
The result  follows from \eqref{BA3} with $\ve =0$ and \eqref{BH3}, \eqref{BI3} since $\cG(c(\cdot,t),\theta(t)) =\lim_{L\ra\infty} \cG_{0,L}(t)$ and $\cD(c(\cdot,t),\theta(t)) = \lim_{L\ra\infty}\cD_{0,L}(t)$.

To bound $I^2_{0,L}(t), \ t_1\le t\le t_2,$ uniformly as $L\ra\infty$ we first note that for some constant $C_0$ and the support restriction for $\phi$, we have uniformly in $L$ for any $x\in [L,2L]$
\[
  \sup_{x\in[ L, 2L]}\frac{\phi(x/L)}{L^2} \leq C_0
\]
by the uniform bound on $a''$ in Assumption~\ref{ass:VW}. Now, use integration by parts such that
\begin{align*}
I^2_{0,L}(t) &= \frac{1}{L}\int \phi'(x/L) \log \bra*{\frac{c(x,t)}{c_{\theta(t)}^\eq(x)}} \ \bra*{ \partial_x c(x,t) - c(x,t) \partial_x \log c_{\theta(t)}^\eq(x)} \dx{x} \\
&= - \frac{1}{L^2} \int \phi''(x/L)   \log  \bra*{\frac{c(x,t)}{c_{\theta(t)}^\eq(x)}} c(x,t) \dx{x} - \int \phi'(x/L) \partial_x c(x,t) \dx{x} \\
&= - \frac{1}{L^2} \int \phi''(x/L)  \bra*{ \log  \bra*{\frac{c(x,t)}{c_{0}^\eq(x)}} -  c(x,t) + c_0^{\eq}(x)} \dx{x} \\
&\quad + \frac{\theta(t)}{L^2} \int \phi''(x/L) W(x) c(x,t) \dx{x} + \frac{1}{L^2} \int \phi''(x/L) c_0^\eq(x,t) \dx{x} \\
&\leq C_0 \int_L^{2L} \Psi\bra*{\frac{c(x,t)}{c_{0}^\eq(x)}} c_0^\eq(x) \dx{x} +
C_0  \int_L^{2L} W(x) c(x,t) \dx{x} +  C_0 \int_L^{2L} c_0^{\eq} \dx{x} .
\end{align*}
The first of the above integral is bounded by $\cH(c|c_0) \leq \cG(c,\theta)$, the second by the conservation law~\eqref{D1} and the third by the integrability condition of Assumption~\ref{ass:VW}, whence we see that \eqref{BH3}, \eqref{BI3} hold for $j=2$. Evidently  \eqref{D1} implies that \eqref{BH3}, \eqref{BI3} hold for $j=3$, and Theorem~\ref{thm:C1} implies that \eqref{BH3} holds for $j=4$.

We are left then to prove \eqref{BH3} for $j=5$. First, since $\log\bra*{ \frac{c(0,t)}{c_{\theta(t)}^\eq(0)}}=0$, we can integrate by parts and obtain
\[
  I_{0,L}^5(t) = - \int \underbrace{\bra*{ \bra[\big]{ \phi(x/L) W(x)}'' - \bra[\big]{ \phi(x/L) W(x)}'  \; \partial_x \log c_{\theta(t)}^\eq(x)}}_{=: I(x)} c(x,t) \dx{x} .
\]
Due to the growth bounds~\eqref{F1:0}, \eqref{F1:1}, \eqref{F1:2}, \eqref{G1} of Assumption~\ref{ass:VW}, we obtain
\[
\abs{\partial_x \log c_{\theta(t)}^\eq(x)} = \abs{- V'(x) + \theta(t) W'(x)} \leq C_0 \bra*{1+\abs{\theta(t)}} W'(x).
\]
Then, the prefactor $I$ in front of $c(x,t)$ is bounded by
\begin{align*}
 I(x) &\leq \norm{\phi''}_\infty \frac{W(x)}{L^2} + 2\norm{\phi'} \frac{W'(x)}{L} + W''(x) \\
 &\quad + C_0 \bra*{ \norm{\phi'}_\infty \frac{W(x)}{L} + W'(x)} \bra*{1+\abs{\theta(t)}} W'(x) ,
\end{align*}
which can further estimated by $W(x)$ with the help of Assumption~\ref{ass:VW}.

Hence, the $I_{0,L}^5(t)$ is bounded with the help of \eqref{D1} by a constant depending on $\norm{\theta}_\infty$, uniformly in $L\geq 1$ and $t_1 \leq t \leq t_2$.
\end{proof}
\begin{remark}\label{rem:GlobalExistence}
Lemma~\ref{lem:CharactMinEnergy} and Lemma~\ref{lem:EED} combined with the local existence result Proposition~\ref{prop:ShortTimeExistence} imply global existence of solutions to \eqref{A1}, \eqref{B1}, \eqref{D1}, \eqref{E1} with initial data satisfying \eqref{AJ2}. The reason is that Lemma~\ref{lem:EED} implies that $\theta(t), \ t>0,$ is bounded by a constant depending only on~$\theta(0)$ and~$\rho$.
\end{remark}
The last ingredient for the proof of the convergence to equilibrium is a dual variational characterization of the dissipation, which allows to prove it lower semicontinuity.
\begin{lemma}\label{lem:Dissipation:lsc}
  Let $C_1,C_2$ be positive constants and the convex set $\mathcal{X}=\{(c,\theta) \in L^1_W(\R^+) \times\R \ : \cG(c,\theta)\le C_1, \ \norm{c}_{L^1_W(\R^+)} \le C_2 \ \}$. Then the domain of $\cD$ as defined in~\eqref{AY3} extends from $\cC^1(\R^+) \times \R^+ $ to $\cX$ by defining
  \begin{equation}\label{e:def:Dissipation:variational}
    \cD(c,\theta) = \sup_{\phi\in \cC_0^\infty(\R^+)} \int \cD^*_\theta[\phi](x) \, c(x) \dx{x} ,
  \end{equation}
  where
  \[
    \cD^*_\theta[\phi](x) = 2 \partial_x\phi(x) + 2 \phi(x) \; \partial_x \log c_\theta(x) -  \abs{\phi(x)}^2.
  \]
  Moreover, $\cD$ by this definition is sequentially lower semicontinuous on $\cX$.
\end{lemma}
\begin{proof}
From the definition~\eqref{e:def:Dissipation:variational}, it is clear that we can use test function $\phi \in \cC_0^1(\R^+)$. Let us first assume, that $c\in \cC^1(\R^+)$ with $c(0) = c_\theta^\eq(0)$. Then, we can integrate by parts in the first term using $\phi(0)=0$ and obtain
\begin{align}\label{e:Dissipation:variational:IntByParts}
  \int \cD^*_\theta\pra*{ \phi}(x) \, c(x) \dx{x} = \int \bra*{ - 2 \phi(x) \, \partial_x \log \frac{c(x)}{c_{\theta}(x)} - \abs{\phi(x)}^2 } c(x) \dx{x} .
\end{align}
Let $\psi:\R\ra\R$ be a non-negative $\cC^\infty$ function with support in the interval $[-1,1]$,  with integral equal to $1$ and define for $\ve>0$ the function $\psi_\ve:\R\ra\R$ by $\psi_\ve(x)=\ve^{-1}\psi(x/\ve), \ x\in\R$. Then, we set
\[
  c_\ve(x) = c(x) \int_0^x \psi_\ve(y-\eps) \dx{y} + c_\theta^\eq(x) \bra*{ 1 - \int_0^x \psi_\ve(y-\eps) \dx{y}} .
\]
By choosing $\phi_\ve = - \partial_x \log \frac{c_\ve}{c_\theta^\eq}$, we have
\begin{align*}
  \phi_\ve(x) = \frac{c_\theta^\eq(x)}{c_\eps(x)} \biggl[ & \bra*{ \frac{c(x)}{c_\theta^\eq(x)} -1}  \psi_\ve(x-\eps) \\
  &+ \int_0^x  \psi_\ve(y-\eps) \dx{y} \ \bra*{ \frac{\partial_x c(x)}{c_\theta^\eq(x)} - \frac{\partial_x c_\theta^\eq(x) \, c(x)}{c_\theta^\eq(x)^2}}\biggr]
\end{align*}
and hence $\phi_\ve(0)=0$, since $c(0) = c_\theta^\eq(0)$. For $x\in [0,2\ve)$ by using again the boundary condition and the $\cC^1$ assumption, we obtain for some $x_1,x_2 \in (0,x)\subseteq (0,2\eps)$
\begin{equation}\label{e:phi_eps:bound}
  \abs{\phi_\ve(x)} \leq C \bra*{  \frac{\abs*{ \partial_x c(x_1)} + \abs*{\partial_x c_\theta^\eq(x_2)}}{c_\theta^\eq(x)} x  \psi_\ve(x-\eps) + 1} \leq C \bra*{ \eps \sup_{z>0} \psi_\ve(z) + 1 } \leq C,
\end{equation}
where $C$ is independent of $\eps$.
 Hence, by using $\phi_\ve$ as test function in~\eqref{e:Dissipation:variational:IntByParts}, we obtain
\begin{align*}
  \int \cD_\theta^*[\phi_ve](x) c(x) \dx{x} &\geq - \int_0^{2\eps} \bra*{2  \abs*{ \phi_\ve(x) }\abs*{\partial_x \log\frac{c(x)}{c_\theta^\eq(x)}} + \abs*{\phi_\ve(x)}^2} c(x) \dx{x} \\
  &\qquad + \int_{2\eps}^\infty \bra*{ \partial_x \log\frac{c(x)}{c_\theta^\eq(x)}}^2 c(x) \dx{x} .
\end{align*}
The bound~\eqref{e:phi_eps:bound} and $\cC^1$ assumption ensure that we can let $\eps\to 0$ and obtain a lower bound of $\cD(c,\theta)$ as defined in~\eqref{e:def:Dissipation:variational} in terms of the one in~\eqref{AY3}. On the other hand, we can interchange the integration and $\sup$ in~\eqref{e:def:Dissipation:variational}. The $\sup$ is attained for $\phi = - \partial_x \log \frac{c}{c_\theta^\eq}$ and we obtain an upper bound by~\eqref{AY3}. This shows, that the definition~\eqref{e:def:Dissipation:variational} is consistent with~\eqref{AY3} for $c\in \cC^1(\R^+)$ with $\cD(c,\theta)<\infty$.

Finally, we observe that $\cD_\theta^*[\phi]\in \cC_0(\R^+)$ due to $\phi\in \cC_c^\infty(\R^+)$ and Assumption~\ref{ass:VW}. Hence, if $(c_m,\theta_m), \ m=1,2,..,$ is a sequence in $\mathcal{X}$ such that $c_m$ converges weakly (see~\eqref{A3}) to some $c\in \cX$ and $\theta_m\to \theta$, then the statement on the lower semicontinuity follows by interchanging the $\lim$ and $\sup$ in the dual formulation of the dissipation~\eqref{e:def:Dissipation:variational}.
\end{proof}

\subsection{Convergence to equilibrium}
We can now use the LaSalle invariance principle to prove convergence of the solution of \eqref{A1}, \eqref{B1}, \eqref{D1}, \eqref{E1} to an equilibrium solution \eqref{C1} and finish the proof of Theorem~\ref{thm:ExistUniqueConvergence}. For the convenience of the reader we quote it in a concise form under our specific assumptions.
\begin{theorem}[{\cite[Theorem 4.2 in Chapter IV]{walk}}]\label{thm:walk}
  Let $\set{S(t)}_{t\geq 0}$ be a dynamical system on a compact metric space $\cX$ and let $V: \cX \to \overline{\R}$ 
  be a lower semicontinous Lyapunov function such that $\dot V(x) \leq  - W(x)$ for all $x \in \cX$, where $W: \cX \to \overline{\R}^+$ is lower semincontinuous and $V(y) > -\infty$ for all $y\in \cX$. Then for any $x\in \cX$ the trajectory $\set{S(t)x}_{t\geq 0}$ converges to the largest positive invariant subset of $\set{z\in \cX:  W(z) = 0}$.
\end{theorem}
\begin{proof}[Proof of convergence to equilibrium statements of Theorem~\ref{thm:ExistUniqueConvergence}]
We start with following preliminary observation and summary of the results. Let $C_1,C_2$ be positive  constants  and the convex set $\mathcal{X}=\{(c,\theta) \in L^1_W(\R^+) \times\R \ : \cG(c,\theta)\le C_1, \ \norm{c}_{L^1_W(\R^+)} \le C_2 \ \}$, then the superlinear growth of $\Psi$ and the growth condition~\eqref{G1} on $W$ in Assumption~\ref{ass:VW} imply uniform integrability of $\cX\subset L^1(\R^+) \times \R^+$. Indeed, since $W$ is increasing, we have $W(\cdot) \geq W(0)>0$ and we can estimate $\norm{c}_{L^1(\R^+)} \leq W(0)^{-1} \norm{c}_{L^1_W(\R^+)} \le W(0)^{-1} C_2$ for all $c\in \cX$. Hence $\cX$ is uniformly bounded in $L^1(\R^+)$. In addition, we obtain from~\eqref{BS3}, applied with $L=0$, the estimate
\begin{align*}
  \int \max\set{0, c\log c} \dx{x} &= \int c \log c \dx{x} - \int_{\set{c\leq 1}} c \log c \dx{x} \\
  &\leq \cH(c | c_0^\eq) + \int \bra*{ 1 + \log c_0^\eq + V } c \dx{x} + \bra*{e^{-1} -1} \int c_0^\eq \dx{x} \\
  &\leq \cG(c,\theta) + \int c \dx{x} \leq C_1 + W(0)^{-1} C_2,
\end{align*}
where we used that $c_0^\eq = \exp\bra*{ - V}$ and $e^{-1}-1\leq 0$.
Hence, the convex set $\cX$ is uniform integrable in $L^1(\R^+)$ (see~\cite[Theorem 4.5.9]{Bogachev2007}). Now, the set $\cX$ has uniform absolutely continuous integrals. By the uniform integrability we find for $\eps>0$ a constant $C>0$ such that
\[
  \int_{\set{c\geq C}} c \dx{x} \leq \frac{\eps}{2} \; , \qquad\text{for all } c\in \cX.
\]
Set $\delta = \eps \bra{2C}^{-1}$ and let $A\subset \R^+$ with $\abs{A}\leq \delta$. Then, for any $c\in \cX$ holds
\[
  \int_{A} c \dx{x} =  \int_{A\cap \set{c \leq C}} c \dx{x} + \int_{A\cap \set{c\geq C}} c \dx{x} \leq \delta C + \frac{\eps}{2} \leq \eps
\]
by our choice of $\delta$. Hence $\cX$ has also uniform absolutely continuous integrals. The last condition to obtain relative compactness is a uniform tightness condition. For $\eps>0$ fix $L_\eps$ such that $C_2 W(L_\eps)^{-1}\leq \eps$. Then, we obtain
\[
  \int_{L_\eps}^\infty c \dx{x} \leq \frac{1}{W(L_\eps)} \int_{L_\eps}^\infty W \, c\dx{x} \leq \frac{C_2}{W(L_\eps)}  \leq \eps.
\]
We obtain that $\cX$ is relative compact in $L^1(\R^+)$ for the weak topology by an application of \cite[Theorem 4.7.20]{Bogachev2007}). Since, the Lyapunov function $(c,\theta)\mapsto \cG(c,\theta)$ by Lemma~\ref{lem:CharactMinEnergy} and the norm $\norm{\cdot}_{L^1_W(\R^+)}$ are lower semicontinuous, we obtain that $\cX$ is a compact metric space. Finally, also the dissipation $(c,\theta) \mapsto \cD(c,\theta)$ is lower semi-continuous from $\mathcal{X}$ to $\R$ by Lemma~\ref{lem:Dissipation:lsc}.

We have already observed the Remark~\ref{rem:GlobalExistence} that the solution $(c(\cdot,t),\theta(t)), \ t>0,$ of \eqref{A1}, \eqref{B1}, \eqref{D1}, \eqref{E1} with initial data satisfying \eqref{AJ2} exists globally in time.
Furthermore, for any $t_0>0$ there exist by Lemma~\ref{lem:stability:sup} constants $C_1,C_2$ depending on $t_0$ such that $(c(\cdot,t),\theta(t))\in\mathcal{X}$ for $t\ge t_0$.
It follows now from Theorem~\ref{thm:walk} that $(c(\cdot,t),\theta(t))$ converges in $\mathcal{X}$ as $t\mapsto\infty$ to the largest invariant subset of $(c,\theta)\in\mathcal{X}$ such that $\cD(c,\theta)=0$.
Evidently this invariant set is the subset $\Omega=\set*{(c_\theta^\eq,\theta): \theta\le 0}\subset \mathcal{X}$.

Observe now from Lemma~\ref{lem:CharactMinEnergy} and Lemma~\ref{lem:stability:sup} that  if $\tilde{U}_{M,N}$ is the union of $\set*{(c(\cdot,t),\theta(t)): \ t\ge t_0}$ and $\Omega$ then $\cG$ is continuous on $\tilde{U}_{M,N}$. Hence there exists  $\theta_\infty\le 0$ such that $\inf_{t>0}\cG(c(\cdot,t),\theta(t))= \cG(c_{\theta_\infty}^\eq, \theta_\infty)$, and from \eqref{e:strict:mononton:MinEnergy} it follows that this $\theta_\infty$ is unique.  We conclude that $(c(\cdot,t),\theta(t))$ converges in $\mathcal{X}$ as $t\ra\infty$ to $(c_{\theta_\infty}^\eq,\theta_{\infty})$.  From Lemma~\ref{lem:CharactMinEnergy} we have that $\inf_{t>0} \cG(c(\cdot,t),\theta(t))\ge \cG(c_{\theta_\eq}^\eq,\theta_\eq)$ with $\theta_\eq=\theta_\eq(\rho)$, whence it follows from \eqref{e:strict:mononton:MinEnergy} that $\theta_\infty\le  \theta_\eq(\rho)$.

Now, we assume that $\theta_\infty< \theta_\eq(\rho)$ and then obtain a contradiction.
Since  $\theta_\infty< 0$ by assumption, we may apply Lemma~\ref{lem:tight} and have for any $\del>0$ there exists $M(\del)>0$ such that
\begin{equation}\label{CI3}
\int_{M(\del)}^\infty W(x) c(x,T) \dx{x} \ < \ \del \quad\text{for all } T\ge 0 \ .
\end{equation}
The proof of the proposition follows by observing from Lemma~\ref{lem:stability:sup} that $c(\cdot,t)$ converges uniformly on any finite interval $[0,L]$ as $t\ra\infty$ to $c_\theta^\eq(\cdot)$ with $\theta=\theta_\infty$. It follows from this and \eqref{CI3} that
 \begin{equation}\label{CO3}
 \theta+\int_0^\infty W(x) c_\theta^\eq(x) \dx{x} \ = \ \rho \quad\text{when } \theta=\theta_\infty \ .
 \end{equation}
The identity \eqref{CO3} yields a contradiction to our assumption that  $\theta_\infty< \theta_\eq(\rho)$.

Let us now prove \eqref{H1}. We assume first that $\theta_\eq(\rho)<0$, then~\eqref{H1} follows from Lemma~\ref{lem:stability:sup} and \eqref{CI3}.  If $\rho=\int W(x) c_0^\eq(x) \dx{x}$ then Lemma~\ref{lem:stability:sup} again implies  for any $L>0$ that
\begin{equation}\label{CP3}
\lim_{t\ra\infty}\int_0^LW(x)|c(x,t)-c^\eq_0(x)| \dx{x} \ = \ 0 \ .
\end{equation}
For any $\del>0$  there exists $L_\del>0$ such that
\begin{equation}\label{CQ3}
\int_0^{L_\del}W(x)c^\eq_0(x) \dx{x} \ > \ \rho-\del \ .
\end{equation}
In this case, we have $\lim_{t\ra\infty}\theta(t)=\theta_\eq(\rho) = 0$.  Hence using \eqref{D1} we have from \eqref{CP3}, \eqref{CQ3} that for any $\del>0$ there exists $T_\del>0$ with the property
\begin{equation}\label{CR3}
\int_{L_\del}^\infty W(x)c(x,t) \dx{x} \ \le \ 2\del \quad\text{for } t\ge T_\del \ .
\end{equation}
Now \eqref{H1} follows from \eqref{CR3} just as in the case  $\theta_\eq(\rho)<0$.
\end{proof}

\bigskip

\section{Rate of convergence to equilibrium (subcritical)}\label{s:Rate}

The proof on the rate of convergence to equilibrium is based on exploiting further the energy--dissipation relation established in Lemma~\ref{lem:EED} in more detail and give a quantitative bound of the energy in terms of the dissipation. This approach was first implemented for the classical Becker-D\"oring equation in~\cite{JN03} and recently generalized in~\cite{CEL15}. Moreover, a similar strategy was recently applied to obtain rates of convergence to equilibrium for Fokker-Planck equation with constraints~\cite{ENS17}.

\subsection{Basic estimates for the free energy and dissipation}

First, we derive the following identity and properties of the normalized free energy~\eqref{e:def:FreeEnergyNorm}.
\begin{lemma}\label{lem:FreeEnergyRelEnt}
  Let $\rho < \rho_s$. Let $\cF_\rho$ be defined as in~\eqref{e:def:FreeEnergyNorm}, then it holds for all $c\in \cM_{ac}(\R^+)$
\begin{equation}\label{e:FreeEnergyRelEnt}
 \cF_\rho(c) = \cH\bra[\big]{ c | c_{\theta_\eq}^\eq} + \tfrac{1}{2} \bra*{ \theta - \theta_\eq }^2 \qquad\text{with}\qquad \theta = \rho - \int W(x)\,c(x) \dx{x}
\end{equation}
and where $\cH\bra*{ f | g}$ is the relative entropy
\begin{equation*}
  \cH\bra*{ f | g } = \int g(x) \Psi\bra*{ \frac{f(x)}{g(x)}}\dx{x} \qquad\text{with}\qquad \Psi(r) = r \log r - r + 1 .
\end{equation*}
Moreover, for any $\theta < 0$ and any $c\in \cM_{ac}(\R^+)$ such that $\theta + \int W c = \rho$ holds
\begin{equation}\label{e:FreeEnergyRelEnt:theta}
  \cF_{\rho}(c) + \tfrac{1}{2} \bra*{ \theta - \theta_\eq }^2 \leq \cH(c| c_\theta^\eq) .
\end{equation}
\end{lemma}
\begin{proof}
For the proof we neglect the argument and integration variable, which will be always $\dx{x}$. As a preliminary step, we obtain for some $\bar\theta<0$ the identity
\begin{align*}
  \cF_{\rho}(c) &= \int c\, \bra*{\log c - 1} + \int c\, V + \frac{1}{2} \theta^2 - \int c_{\theta_\eq}^\eq \bra{\log c_{\theta_\eq}^\eq - 1} + \int c_{\theta_\eq}^\eq V - \frac{1}{2} \theta_\eq^2 \\
  &= \int c\, \log \frac{c}{c_{\bar\theta}} - \int c + \bar \theta \int c W + \frac{1}{2} \theta^2 - \theta_\eq \int c_{\theta_\eq}^\eq W + \int c_{\theta_\eq}^\eq - \frac{1}{2} \theta_\eq^2 \\
  &= \cH\bra*{ c | c_{\bar\theta}} + \int c_{\theta_\eq}^\eq - \int c_{\bar \theta} + \frac{1}{2} \theta^2 - \bar\theta \, \theta + \frac{1}{2} \theta_\eq^2 - \rho \bra*{\theta_\eq - \bar\theta} \\
  &=: \cH\bra*{ c | c_{\bar\theta}}  + I\bra*{\bar\theta, \theta_\eq , \theta} .
\end{align*}
Now, the identity~\eqref{e:FreeEnergyRelEnt} follows by setting $\bar\theta = \theta_\eq < 0$ by the assumption $\rho < \rho_s$ and hence noting that $I(\theta_\eq,\theta_\eq, \theta) = \frac{1}{2} (\theta- \theta_\eq)^2$. The estimate~\eqref{e:FreeEnergyRelEnt:theta} follows by setting $\bar\theta = \theta$ once we have shown $I(\theta,\theta_\eq,\theta)\leq -\frac{1}{2} \bra*{ \theta - \theta_\eq }^2$. First, we note that
\begin{equation*}
  I(\theta,\theta_\eq,\theta) =  \int c_{\theta_\eq}^\eq - \int c_{\theta}^\eq + \bra*{ \frac{\theta_{\eq} + \theta}{2} - \rho} \bra*{\theta_\eq - \theta} .
\end{equation*}
We introduce $m(\theta) = \int c_\theta^\eq < \infty$ having derivatives
\[
  0 < m^{(k)}(\theta) = \int W^k c_\theta^\eq < \infty \qquad\text{for } k=1,2,\dots
\]
for any $\theta < 0$ by Assumption~\ref{ass:VW}. Hence, we have from the conservation of mass $\theta_\eq + m'(\theta_\eq) = \rho$ the following identity
\begin{align*}
  I(\theta,\theta_\eq,\theta) &=  m(\theta_\eq) - m(\theta) + \bra*{ \frac{1}{2}\bra{ \theta - \theta_\eq } - m'(\theta_\eq) } \bra*{\theta_\eq - \theta} \\
  &= m(\theta_\eq) - m(\theta) - m'(\theta_\eq) \bra*{\theta_\eq - \theta} - \frac{1}{2} \bra*{\theta_\eq - \theta}^2 \\
  &\leq - \bra[\Big]{ 1 + \min\set[\big]{ m''(\theta) , m''(\theta_\eq)} } \; \frac{1}{2} \bra*{\theta_\eq - \theta}^2 ,
\end{align*}
where we used a Taylor expansion of $m(\theta_\eq)$ and the monotonicity of $m''(\theta)$, that is $0< m'''(\theta)< \infty$ for any $\theta < 0$.
\end{proof}
The next ingredient is a weighted Pinsker inequality, which in the setting of the classical Becker-D\"oring model is derived in~\cite{JN03,N03} and for probability measures is proven by similar means in~\cite{BV05}.
\begin{lemma}[Pinsker inequality]\label{lem:Pinsker}
  Let $\theta < 0$ and $V,W$ satisfy Assumption~\ref{ass:VW}. Then it holds for all $c\in \cM_{ac}(\R^+)$
  \begin{equation}\label{e:Pinsker:sub}
    \int W(x) \, \abs*{c(x) - c_\theta^\eq(x)} \dx{x}  \leq \frac{2}{\abs{\theta}} \max\set*{ \cH(c | c_\theta^\eq) , \sqrt{C_{V} \cH(c | c_\theta^\eq)}} ,
  \end{equation}
  with $C_{V} = \int e^{-V} < \infty$.\newline
  Moreover, for any $\theta\in \R$ and any $L>0$ exists $C = C(V,W,\theta,L)$ such that
  \begin{equation}\label{e:Pinsker:L}
    \int_0^L W(x) \abs*{c(x) - c_\theta^\eq(x)} \leq 2 \max\set*{ \cH(c|c_\theta^{\eq,L}) , \sqrt{C \cH(c|c_\theta^{\eq,L})} } ,
  \end{equation}
  where $c_\theta^{\eq,L}(x) = \bfOne_{[0,L]}(x) c_\theta^\eq(x)$.
\end{lemma}
\begin{proof}
  We introduce the convex non-negative function $\varphi(r) = (1+ r) \log(1 + r) - r$, which allows to rewrite
  \begin{equation*}
    \cH(c | c_\theta^\eq)  = \int c_\theta^\eq(x)\, \varphi\bra*{ \frac{c-c_\theta^\eq}{c_\theta^\eq}} \dx{x} .
  \end{equation*}
  The convex dual function $\varphi^*(s) = e^s - s - 1$ has the property $\varphi^*(\eps s) \leq \eps^2 \varphi^*(s)$ for all $\eps \in [0,1]$. Moreover, it holds $\varphi(\abs{r}) \leq \varphi(r)$. By convex duality, the bound $rs \leq \varphi(r) + \varphi^*(s)$ holds. Hence, by setting $r= \frac{\abs{c(x)-c_\theta^\eq(x)}}{c_\theta^\eq(x)}$ and $s = \eps \eta W(x)$ for some $\eta>0$ to be determined later, we arrive at
  \begin{equation*}
    \eps \eta W(x) \frac{\abs{c(x)-c_\theta^\eq(x)}}{c_\theta^\eq(x)} \leq \varphi\bra*{ \frac{c(x)-c_\theta^\eq(x)}{c_\theta^\eq(x)}} + \eps^2 e^{\eta W(x)} .
  \end{equation*}
  Dividing by $\eps$, multiplying by $c_\theta^\eq(x)$ and integrating over $[0,L]$ leads to
  \begin{equation}\label{e:Pinsker:p1}
    \eta \int_0^L W(x) \abs*{c(x)-c_\theta^\eq(x)} \dx{x} \leq \eps \int_0^L e^{\eta W(x)} c_\theta^\eq(x) \dx{x} +\frac{1}{\eps} \cH( c | c_\theta^\eq) .
  \end{equation}
  To prove~\eqref{e:Pinsker:sub}, we set $\eta = \abs{\theta}$, $L=\infty$ in~\eqref{e:Pinsker:p1} and obtain
  \begin{equation*}
    \abs{\theta} \int W(x) \abs*{c(x)-c_\theta^\eq(x)} \dx{x} \leq \eps \int_0^\infty e^{-V(x)} \dx{x} + \frac{1}{\eps} \cH( c | c_\theta^\eq) .
  \end{equation*}
  The integral $\int e^{-V(x)} \dx{x}$ is bounded by Assumption~\ref{ass:VW} and we obtain the result~\eqref{e:Pinsker:sub} by setting $\eps = \min\set*{\sqrt{\cH( c | c_\theta^\eq)}/\sqrt{C_{V}},1}$.
  \newline
  To prove~\eqref{e:Pinsker:L}, we just set $\eta=1$ in~\eqref{e:Pinsker:p1} and get the constant
  \[
   C(V,W,\theta,L) = \int_0^L e^{W(x)} c_{\theta}^\eq(x) \dx{x} = \int_0^L \exp\bra*{-V(x) + (\theta+1) W(x)} \dx{x} < \infty ,
  \]
  which by setting $\eps= \min\set*{\sqrt{\cH( c | c_\theta^\eq)}/\sqrt{C},1}$ gives again the result.
\end{proof}
We give an immediate Corollary of the above Pinsker inequality and Lemma~\ref{lem:FreeEnergyRelEnt}, which by the convergence statement of Theorem~\ref{thm:QuantLongTime} proofs Corollary~\ref{cor:Pinsker}.
\begin{corollary}\label{cor:FreeEnergyRelEnt}
  For $\rho<\rho_s$ and with Assumption~\ref{ass:VW} it holds for any $c\in L^1_W(\R^+)$ with $\theta + \norm{c}_{L^1_W(\R^+)} = \rho$ the estimate
  \begin{equation}\label{e:FreeEnergyRelEnt:est}
     \bra*{\int W(x) \abs*{c(x)- c_{\theta^\eq}^\eq(x)} \dx{x}}^2 + \bra*{ \theta - \theta^\eq}^2 \leq \bra[\Big]{ 1 + 4\abs*{\theta^\eq}^{-2} \max\set[\Big]{ \cF_\rho(c), C_{V}} } \cF_\rho(c) ,
  \end{equation}
  where $C_{V} = \int e^{-V(x)}\dx{x}$.
\end{corollary}
\begin{proof}
  By the representation~\eqref{e:FreeEnergyRelEnt} of Lemma~\ref{lem:FreeEnergyRelEnt}, it is enough to bound $\cH(c| c_{\theta^\eq}^\eq)$ from below, for which we use the Pinsker inequality of Lemma~\ref{lem:Pinsker} with $\theta= \theta^\eq$.
\end{proof}
Since, by the energy dissipation identity~\eqref{e:EDI} the free energy $\cF_\rho(c(t))$ is decreasing along a solution $c$ of~\eqref{A1}, \eqref{B1}, \eqref{D1}, \eqref{E1}, we find that whenever $c(t)$ is such that $\theta(t)\leq -\delta$, the estimate $\cF_\rho(c(t)) \leq C \cH(c(t)| c_{\theta_\eq}^\eq)$ for some constant only depending on Assumption~\ref{ass:VW} and $\delta$.

The following weighted $L^1$ estimate will help to control error terms occurring in the derivation of the dissipation inequality.
\begin{lemma}\label{lem:weightDifference}
 Suppose Assumption~\ref{ass:VW} holds.
 \begin{enumerate}[ (a) ]
  \item Let $\delta>0$. Then, for any $\theta < 0$ and all $c$ smooth enough with $\theta + \norm{Wc}_1 = \rho$ as well as $c(0) = c_\theta^\eq(0)$ holds for some $C=C(V,W,\theta)$
 \begin{equation}\label{e:weightDifference}
   \int \sqrt{W(x)}\, W'(x)\, \abs*{c(x) - c_\theta^\eq(x)}\dx{x} \leq C \sqrt{\bra{\rho-\theta} \; \cD(c,\theta)} .
 \end{equation}
  \item Let $\Theta >0$ and $L>0$. Then for all $\abs{\theta} \leq \Theta$ and all $c$ smooth enough with $\theta + \int W c = \rho$ as well as $c(0) = c_\theta^\eq(0)$ holds for some $C=C(V,W,\Theta,L)$
  \begin{equation}
    \label{e:weightDifference:L}
   \int_0^L \sqrt{W(x)}\, W'(x) \, \abs{c(x) - c_\theta^\eq(x)} \dx{x} \leq C \sqrt{\bra{\rho-\theta} \; \cD(c,\theta)} .
  \end{equation}
 \end{enumerate}
\end{lemma}
\begin{proof}
  Let us start by noting that since $c(0) = c_\theta^\eq(0)$, we can calculate
  \begin{align*}
    \int \sqrt{W(x)}\, W'(x) \, \abs{c(x) - c_\theta^\eq(x)} \dx{x} &= \int_0^\infty \sqrt{W(x)}\, W'(x) \, \abs*{ \frac{c(x)}{c_\theta^\eq(x)} - 1} c_\theta^\eq(x) \dx{x} \\
    &= \int_0^\infty \sqrt{W(x)}\, W'(x) \, \abs*{ \int_0^x \partial_y \frac{c(y)}{c_\theta^\eq(y)} \dx{y} } c_\theta^\eq(x) \dx{x} \\
    &\leq \int_0^\infty \abs*{\partial_y \frac{c(y)}{c_\theta^\eq(y)}} \int_y^\infty \!\! \sqrt{W(x)}\, W'(x) \,  c_\theta^\eq(x) \dx{x}  \dx{y} .
  \end{align*}
  The inner integral can be bounded by applying~\eqref{F1:1} of Assumption~\ref{ass:VW} leading to $c_{\theta}^{\eq}(x) \leq e^{-(\abs{\theta}-\delta)W(x)}$ for $x\geq x_\delta$. Then, we have by integration by parts for any $y\geq x_\delta$ the bound
  \begin{align*}
    \int_y^\infty \sqrt{W(x)}\, \underbrace{W'(x) \,  e^{-(\abs{\theta}-\delta)W(x)}}_{\mathclap{=-\bra*{\abs{\theta}-\delta}^{-1}\pderiv{}{x} e^{-(\abs{\theta}-\delta)W(x)}}} \dx{x} &\leq \frac{\sqrt{W(y)} e^{-(\abs{\theta}-\delta)W(y)}}{\abs{\theta}+\delta} \\
    &\quad + \frac{\int_y^\infty \sqrt{W(x)}\, W'(x) \,  c_\theta^\eq(x) \dx{x}}{2\bra{\abs{\theta}+\delta} W(y)}    .
  \end{align*}
  Hence, we obtain a bound, if we choose $x_\delta$ large enough such that $\bra{\abs{\theta}+\delta} W(y)\geq 1$. On the other hand, for $y\leq x_\theta$ the integral on $[y, x_\theta]$ is anyway bounded. Hence, for a constant $C=C(V,W,\theta)$, we can further estimate by the Cauchy-Schwarz inequality
  \begin{align*}
    \int \sqrt{W(x)} \, W'(x)\,  \abs{c(x) - c_\theta^\eq(x)} \dx{x} &\leq C \int \sqrt{W(x)} \abs*{\partial_y \log \frac{c(y)}{c_\theta^\eq(y)}} c(y) \dx{y} \\
    &\leq C \bra*{ \int W(x) c(x) \dx{x}  \ \cD(c,\theta)}^{\frac{1}{2}} .
  \end{align*}
  Finally, the estimate \eqref{e:weightDifference:L} is an immediate consequence by suitable truncating the above occurring integrals to the interval $[0,L]$.
\end{proof}

\subsection{Weighted logarithmic Sobolev inequalities}
\label{s:weightLSI}

In this section, we are going to show the following result:

\begin{theorem}
  \label{thm:epi}
  Take $\rho < \rho_{\mathrm{s}}$, $\delta > 0$ and $k > 0$. Let $-1/\delta \leq \theta \leq -\delta$ and define the weight
  \begin{equation}\label{e:def:omega}
    \omega(x) = \frac{W(x)}{\bra*{W'(x)}^2} .
  \end{equation}
  Then, for any $c\in \cM_{ac}(\R^+)$ such that $\theta + \int W c = \rho$, $c(0) = c_\theta^\eq(0)$ and
  \begin{equation}\label{eq:epi:moment}
    \bra*{\int \omega(x)^k c_\theta^\eq(x)  \, \Psi\bra*{\frac{c(x)}{c_\theta^\eq(x)}} \dx{x}}^{\frac{1}{k}} =: C_k
  \end{equation}
  there exists a constant $C_{\LSI} = C_{\LSI}(V,W, \delta)$ such that
  \begin{equation}
    \label{eq:epi}
      \cF_\rho(c)^{\frac{1+k}{k}} \leq  C_k \, C_{\LSI} \, \cD(c,\theta).
  \end{equation}
  In particular, if~\eqref{G1:alpha} holds with $\beta=0$, that is $c_0 W(x) \leq W'(x)^2$,  then there exists $C_{\LSI} = C_{\LSI}(V,W, \delta)$ such that
    \begin{equation}
    \label{eq:epi:linear}
      \cF_\rho(c) \leq  \frac{C_{\LSI}}{c_0} \, \cD(c,\theta).
    \end{equation}
\end{theorem}
A form of the above weighted entropy dissipation inequality~\eqref{eq:epi} was recently derived in~\cite{CEL15} for the classical Becker-D\"oring model with subcritical inital mass. The main ingredient of the proof is a weighted logarithmic Sobolev inequality, which we adopt to our setting with Dirichlet boundary conditions. Therefore, we slightly modify the arguments in~\cite{ABCFGMRS00,Barthe,BR08,BG99} to deduce a criterion for logarithmic Sobolev inequalities on the positive half real line incorporating functions with fixed boundary conditions at~$0$. These kind of inequalities have there origin in the Muckenhoupt criterion~\cite{Muckenhoupt}.
\begin{proposition}\label{prop:mixedDirLSI}
  Let $\nu \in \cP(\R^+)$ and $\mu\in \cM_{ac}(\R^+)$ be absolutely continuous and by abuse of notation let its density be denoted by $\mu(\dx{x})=\mu(x) \dx{x}$. Let $A$ be the smallest constant such that for any smooth $f$ on $\R^+$ with $f(0)=1$ it holds
  \begin{equation}\label{e:mixedDirLSI}
    \Ent_{\nu}(f)= \int f \log \frac{f}{\int f \dx{\nu}}  \dx{\nu} \leq A \int \abs{\partial_x \log f}^2  f \dx{\mu}.
  \end{equation}
  Then, it holds $B/4\leq A \leq B$ where
  \begin{equation}\label{e:mixedDirLSI:B}
    B = \sup_{x > 0 } B(x) \qquad\text{with}\qquad B(x)= \nu\bra[\big]{ [x,\infty] } \log\bra*{ 1 + \frac{e^2}{\nu\bra[\big]{ [x,\infty] }}} \int_0^x \frac{\dx{y}}{\mu(y)}
  \end{equation}
\end{proposition}
\begin{proof}
  For the proof it will be convenient to proof the equivalent formulation of~\eqref{e:mixedDirLSI} with~$f$ replaced by~$f^2$, which does not change the boundary value $f(0)=f^2(0)=1$. Hence, we want to prove the inequality
  \begin{equation}\label{e:mixedDirLSI2}
    \Ent_{\nu}(f^2) = \int f^2 \log \frac{f^2}{\int f^2 \dx{\nu}}  \leq 4 A \int \abs{f'}^2  \dx{\mu}.
  \end{equation}
  Let us write $\Phi(r) = e^r - 1$. Therewith, we can define an Orlicz type of norm, by setting for $K>0$
  \begin{equation}\label{e:mixedDirLSI:Orclicz}
    \norm{ f }_{\nu, \Phi, K} = \sup\set*{ \int \abs{f} g \dx{\nu} : g\geq 0 , \int \Phi(g) \dx{\nu} \leq K} .
  \end{equation}
  Let us assume, for a moment, that the following facts hold true
  \begin{align}
    \Ent_\nu( f^2 ) &\leq \norm*{ (f-1)^2 }_{\nu,\Phi,e^2}  \label{e:mixedDirLSI:Ent} \\
    \forall I\subset \R^+ : \qquad \norm*{ \bfOne_I}_{\nu,\Phi,K} &=  \nu[I] \log\bra*{1 + \frac{K}{\nu[I]} } \label{e:mixedDirLSI:1}
  \end{align}
  The identity~\eqref{e:mixedDirLSI:Ent} shows, that it is enough to show that the best constant $A$ in
  \begin{equation*}
    \norm{ (f-1)^2}_{\nu,\Phi,e^2} \leq A \int \abs*{f'}^2 \dx{\mu}
  \end{equation*}
  satisfies $B\leq A \leq 4 B$ with $B$ as in~\eqref{e:mixedDirLSI:B}. This result follows by an application of \cite[Proposition 2]{Barthe} or \cite[Corollary 5.2]{BG99}, which yields the following expression for $B$
  \begin{equation*}
    B = \sup_{x > 0} \norm{ \bfOne_{[x,\infty)} }_{\nu, \Phi, e^2} \int_{0}^x \frac{\dx{y}}{\mu(y) } .
  \end{equation*}
  Hence, the identity~\eqref{e:mixedDirLSI:1} gives the conclusion. It is left to show~\eqref{e:mixedDirLSI:Ent} and \eqref{e:mixedDirLSI:1}.

  \medskip
  For the proof of~\eqref{e:mixedDirLSI:Ent}, let us start from the following observation by~\cite[Lemma 9]{Rothaus} for any $a\in \R$
  \[
    \Ent_\nu\bra[\big]{f^2} \leq \Ent_{\nu}\bra[\big]{ (f-a)^2 } + 2 \int (f-a)^2 \dx{\nu} .
  \]
  On the other hand by the variational characterization of the entropy follows for any $f\geq 0$
  \begin{align*}
    \Ent_{\nu}(f) + 2 \int f \dx{\nu} &= \sup_g\set*{ \int f (g+2) \dx{\nu} : \int e^g \dx{\nu} \leq 1 } \\
    &\leq \sup_g\set*{ \int f g \bfOne_{\set{g\geq 0}} \dx{\nu} : \int e^g \dx{\nu} \leq e^2 } \\
    &\leq \sup_g\set*{ \int f g \dx{\nu} : g \geq 0, \int \bra*{e^g - 1 } \dx{\nu} \leq e^2 } = \norm{ f }_{\nu, \Phi,e^2 } .
  \end{align*}
  The last step is a consequence of $\int e^{g \bfOne_{\set{g \geq 0}}} \dx{\nu} \leq \int_{\set{g\geq 0}} e^g \dx{\nu} + \int_{\set{g< 0}} e^g \dx{\nu} \leq e^2 +1$. A combination of the above two estimates yields~\eqref{e:mixedDirLSI:Ent}.

  \medskip
  One direction of the proof of~\eqref{e:mixedDirLSI:1} from using the function $g(x) = \bfOne_I(x)\, \log(1+K/\nu[I])$ in the definition of~\eqref{e:mixedDirLSI:Orclicz}.  Using the fact that $s\mapsto \log(1+s)$ is concave, the estimate in the other direction follows from an application of the Jensen inequality
  \begin{align*}
    \int_I g \dx{\nu} &= \nu[I] \int_I \log\bra*{1+ \Phi(g)} \frac{\dx{\nu}}{\nu[I]}
    \leq \nu[I] \log\bra*{ 1 + \int_I \Phi(g) \frac{\dx{\nu}}{\nu[I]}} .
  \end{align*}
  Taking finally the supremum over all $g$ with $\int \Phi(g) \dx{\nu} \leq K$ concludes~\eqref{e:mixedDirLSI2}.
\end{proof}
We derive the following consequence of Proposition~\ref{prop:mixedDirLSI} as a version of the entropy production
inequality with a weight, which does not need any additional
exponent. Again, the need for the weight is directly related to the
fact that logarithmic Sobolev inequalities do not hold for an
exponentially decaying measure such as $c_\theta^\eq$, but in general need a
Gaussian decay to be valid.
\begin{proposition}\label{prop:epi0}
  Suppose Assumption~\ref{ass:VW} holds and let $\delta>0$. Then, there exists a constant
  $C_{\LSI} = C_{\LSI}(V,W,\delta)$ depending only on $\delta$ and the constants in Assumption~\ref{ass:VW}, such that for all $c$ with $\theta + \norm{W c}_1 = \rho$ and $-1/\delta \leq \theta \leq  -\delta$ as well as $c(0) = c_\theta^\eq(0)$ holds for $\omega$ given by~\eqref{e:def:omega}
  \begin{equation}
    \label{eq:epi0}
    \int_0^\infty \frac{c_\theta^\eq(x)}{\omega(x)}\; \Psi\bra*{\frac{c(x)}{c_\theta^\eq(x)}} \dx{x} \leq C_{\LSI} \cD(c,\theta).
  \end{equation}
  Likewise, let $\Theta>0$ and $L>0$, then there exists $C_{\LSI} = C_{\LSI}(V,W,\Theta,L)$ such that for any $\abs{\theta}\leq \Theta$ and any $c\in \cM_{ac}(\R^+)$ such that $\theta+ \int Wc = \rho$ and $c(0) = c_\theta^\eq(0)$ holds
    \begin{equation}
    \label{eq:epi0:L}
    \int_0^L \frac{c_\theta^\eq(x)}{\omega(x)}\; \Psi\bra*{\frac{c(x)}{c_\theta^\eq(x)}} \dx{x} \leq C_{\LSI} \cD(c,\theta).
  \end{equation}
\end{proposition}
\begin{proof}
 Let us set $\nu(\dx{x}) = \frac{c_{\theta}^\eq(x)}{\omega(x) Z_{\theta,W}} \dx{x}$ with $Z_{\theta,W} = \int\frac{c_{\theta}^\eq(x)}{\omega(x)} \dx{x}$, $\mu(\dx{x})=\frac{c_\theta^\eq(x)}{Z_{\theta,W}}$ and $f(x) = \frac{c}{c_\theta^\eq}(x)$.
 Then, it holds $\cD(c,\theta) = Z_{\theta,W} \int \abs*{\partial_x \log f}^2 f \dx{\mu}$ and we can rewrite
\begin{align*}
  \MoveEqLeft{\int \frac{c_\theta^\eq(x)}{Z_{\theta,W}\,\omega(x)}  \Psi\bra*{\frac{c(x)}{c_\theta^\eq(x)}} \dx{x}}\\
  &= \int \frac{c_\theta^\eq(x)}{Z_{\theta,W}\,\omega(x)} \,\frac{c(x)}{c_\theta^\eq(x)}\,\log\bra*{ \frac{c(x)}{c_\theta^\eq(x)}} \dx{x} - \int  \frac{c(x)\dx{x}}{Z_{\theta,W}\,\omega(x)} + 1 \\
  &= \int f \log f \dx{\nu} - \int f \dx{\nu}\; \log\bra*{ \int f\dx{\nu}} + \Psi\bra*{ \frac{c(x)\dx{x}}{Z_{\theta,W}\,\omega(x)}} .
\end{align*}
 The first term is exactly $\Ent_\nu(f)$. Hence, we can apply Proposition~\ref{prop:mixedDirLSI}, since $f(0)=1$ by definition and the assumption $c(0) = c_\theta^\eq(0)$. The second term can be estimated by using $\Psi(x) \leq \bra*{ x- 1}^2$ and by noting that $1= \int \frac{c(x)\dx{x}}{Z_{\theta,W}\,\omega(x)}$ from the definition of $Z_{\theta,W}$. Then, we can apply Lemma~\ref{lem:weightDifference} to it, where we note that $\frac{1}{w(x)} \leq C \sqrt{W} W'$ by~\eqref{G1} of Assumption~\ref{ass:VW}. By doing so, we arrive at the bound
 \begin{align*}
   \int \frac{c(x)}{Z_{\theta,W}\,\omega(x)} \Psi\bra*{\frac{c(x)}{c_\theta^\eq(x)}}  \dx{x} &\leq \frac{A}{Z_{\theta,W}} \cD(c,\theta) + \bra*{ \frac{1}{Z_{\theta,W}} \int \frac{1}{\omega(x)} \abs{ c(x) - c_\theta^\eq(x)}\dx{x}}^2  \\
   &\leq \bra*{\frac{ A}{Z_{\theta,W}} + \frac{C}{Z_{\theta,W}^2}}  \cD(c,\theta).
 \end{align*}
  Now, we multiply by $Z_{\theta,W}$ and note that
  \[
    \pderiv{}{\theta} Z_{\theta,W} = \int a(x) c_\theta^\eq(x) \dx{x} > 0
  \]
  and hence it holds $\infty > Z_{0,W} \geq Z_{\theta,W} \geq Z_{-\delta^{-1},W} >0$.

  For the constant $A$, we have that $A \leq B$ with $B$ defined in~\eqref{e:mixedDirLSI:B} of Proposition~\ref{prop:mixedDirLSI}. First, we obtain an upper and lower bound on $\nu([x,\infty))$. We introduce the function $b_{\theta}(x) = V(x) - \theta W(x)$ and have thanks to~\eqref{F1:1} and~\eqref{G1} for $y\geq x_{\bar\delta}$ the estimate
  \begin{equation}\label{e:comp:b}
    b'_\theta(x) \geq \bra*{\abs{\theta}-\bar\delta} W'(x) \geq \bra*{\abs{\theta}-\bar\delta} \sqrt{c_0} .
  \end{equation}
  Hence, for $\bar\delta$ sufficiently small is $b_\theta$ monotone increasing. We assume from now on that $2\bar\delta\leq \abs{\theta}$
  To do so, we write
  \begin{align}
    \int_x^\infty \frac{1}{\omega(y)} e^{-b_{\theta}(y)}\dx{y} &= - \int_x^\infty \frac{1}{\omega(y)\, b'_\theta(y)} \pderiv{}{y}  e^{-b_\theta(y)} \dx{y} \notag \\
    &= \frac{e^{-b_\theta(x)}}{\omega(x) W'(x)} + \int_x^\infty \bra*{\frac{1}{\omega(x) \, b'_\theta(x)}}' e^{-b_{\theta}(x)} \dx{x} \label{e:nu:integral}.
  \end{align}
  In addition, we have
  \begin{align*}
    \bra*{\frac{1}{\omega(x) \, b'_\theta(x)}}' &= \frac{2 W'(x) W''(x)}{W(x) b'_\theta(x)} - \frac{W'(x)^3}{W(x)^2 b'_\theta(x)} - \frac{W'(x)^2 \, b_\theta''(x)}{W(x) b_\theta'(x)^2 }\\
    &= \frac{1}{\omega(x)} \bra*{ \frac{2 W''(x)}{W'(x)^2} \frac{W'(x)}{b'_\theta(x)} - \frac{1}{W(x)} \frac{W'(x)}{b'_\theta(x)} - \frac{b_\theta''(x)}{b_\theta'(x)^2}}
  \end{align*}
  Each term in the bracket can be made arbitrary small. Indeed, we have by~\eqref{F1:2} that $\abs{W''(x)}\leq \bar\delta W'(x)^2$ for $x\geq \bar\delta$. Then, from~\eqref{e:comp:b}, we have $W'(x)/b'_\theta(x) \leq \bra*{\abs{\theta}-\bar\delta}^{-1}$ and likewise by using once more~\eqref{F1:1} also $W'(x)/b'_\theta(x) \geq \bra*{\abs{\theta}+\bar\delta}^{-1}$. Moreover, again by~\eqref{F1:2}, we have $\abs*{b_\theta''(x)} = \abs{V''(x) - \theta W''(x)} \leq \bar\delta(1+\abs{\theta}) W'(x)^2$. Finally, the strict monotonic growth from~\eqref{G1} implies that $W(x)$ becomes small for $x\geq x_{\bar\delta}$ large. In total, we find by choosing a sufficiently small~$\bar\delta$ and sufficiently large~$x_{\bar\delta}$ that for all $x\geq x_{\bar\delta}$
  \[
     -\frac{1}{2\omega(x)} \leq \bra*{\frac{1}{\omega(x) \, b'_\theta(x)}}' \leq \frac{1}{2\omega(x)} .
  \]
  By plugging this estimate into~\eqref{e:nu:integral} and rearrange, we have obtained the bound for $x\geq x_{\bar\delta}$
  \begin{equation}\label{e:epi:p1}
     \frac{e^{-b_\theta(x)}}{2 \omega(x) W'(x)} \leq Z_{\theta,W} \nu([x,\infty)) \leq \frac{2 e^{-b_\theta(x)}}{\omega(x) W'(x)}.
  \end{equation}
  By very similar arguments, we can estimate for $x\geq x_{\bar\delta}$
  \begin{equation}\label{e:epi:p2}
    Z_{\theta,W}\int_{x_{\bar\delta}}^x \frac{\dx{y}}{\mu(y)} = \int_{x_{\bar\delta}}^x e^{b_{\theta}(x)} \dx{y}  \leq \frac{2 e^{b_{\theta}(x)}}{W'(x)}  ,
  \end{equation}
  By the smoothness of $V$ and $W$ from Assumption~\ref{ass:VW} it is enough to estimate $B(x)$ as defined in \eqref{e:mixedDirLSI:B} of Proposition~\ref{prop:mixedDirLSI} for $x\geq x_{\bar\delta}$. Hence, we can bound by using the estimates~\eqref{e:epi:p1} and~\eqref{e:epi:p2}, recalling that $\omega(x)=W(x)/W'(x)^2$, for $x\geq x_{\bar\delta}$
  \begin{align*}
    B(x) &= \nu\bra*{[x,\infty)} \log\bra*{ 1 + \frac{e^2}{\nu([x,\infty))}} \int_0^x \frac{\dx{y}}{\mu(y)} \\
    &\leq \frac{4}{W(x)} \log\bra*{ 1 + \frac{2 W(x)}{W'(x)} \, e^{V(x) - \theta W(x)}}\\
    &\leq \frac{4}{W(x)} \bra*{ \log\bra*{ 1 + \frac{2 W(x)}{\sqrt{c_0}} } + \log\bra*{ 1+ e^{(\delta^{-1} + C_V)W(x)}}},
  \end{align*}
  where we used the bound $\log(1+ab) \leq \log(1+a)+\log(1+b)$ for any $a,b>0$ and the growth conditions on $V$ and $W$ from Assumption~\ref{ass:VW} implying $V(x)\leq C_V W(x)$ for some $C_V>0$. We conclude with the help of the estimate $\log(1+e^x)\leq x+\log(2)$ for $x\geq 0$ that $\limsup_{x\to \infty} B(x) \leq C < \infty$ and hence also $B=\sup_{x> 0} B(x) \leq C < \infty$ for some $C=C(V,W,\delta)$.

  The proof of~\eqref{eq:epi0:L} follows by the same arguments, suitable truncation of the integrals as well as~\eqref{e:weightDifference:L} from Lemma~\ref{lem:weightDifference}.
\end{proof}
\medskip
The previous result essentially contains Theorem \ref{thm:epi}. The
proof can now be finished by an interpolation argument based on the moment assumption on $c$~\eqref{eq:epi:moment}.
\begin{proof}[Proof of Theorem \ref{thm:epi}]
  By~\eqref{e:FreeEnergyRelEnt:theta} from  Lemma~\ref{lem:FreeEnergyRelEnt}, we have $\cF_\rho(c) \leq \cH(c | c_\theta^\eq )$. Next, we are going to use~\eqref{eq:epi0} of Proposition~\ref{prop:epi0} by an interpolation argument. Set $\frac{1}{p} = \frac{k}{1+k}$ and $\frac{1}{q} = 1-\frac{1}{p} = \frac{1}{1+k}$. Then, we have $(\omega)^\frac{k}{q} = (\omega)^\frac{1}{p}$ and hence by H\"older's inequality
  \begin{align*}
    \cF_\rho(c) &\leq \cH(c | c_\theta^\eq)
    = \int c_\theta^\eq(x) \Psi\bra*{\frac{c(x)}{c_\theta^\eq(x)}} \dx{x}\\
    &\leq
    \left(
      \int \frac{c_\theta^\eq(x)}{\omega(x)} \Psi\bra*{\frac{c(x)}{c_\theta^\eq(x)}} \dx{x}
    \right)^{\frac{1}{p}}
    \left(
      \int \omega(x)^k c_\theta^\eq(x) \Psi\bra*{\frac{c(x)}{c_\theta^\eq(x)}} \dx{x}
    \right)^{\frac{1}{q}}
    \\
    &\leq
    C_k^{\frac{k}{q}} C_{\LSI}^{\frac{1}{p}} \cD(c,\theta)^{\frac{1}{p}},
  \end{align*}
  which finishes the proof of the first part. For the second part, we note that~\eqref{G1:alpha} with $\beta=0$ implies $\omega(x)\leq \frac{1}{c_0}$ for all $x\in \R^+$ and the estimate~\eqref{eq:epi0} becomes already inequality~\eqref{eq:epi:linear}.
\end{proof}
The final ingredient from the static investigation of the entropy and dissipation is a lower bound on the dissipation in the case where $\theta$ is not strictly negative.
\begin{theorem}\label{thm:dissipation:lb}
  Let $\delta >0$ and $\Theta>0$. Then there exists $\eps=\eps(V,W,\delta,\Theta)>0$ such that for all $\delta \leq \abs{ \theta -\theta_\eq}\leq \Theta$ and all $c\in \cM_{ac}(\R^+)$ with $\theta+\int W c = \rho$ and $c(0)=c_\theta^\eq(0)$ holds
  \begin{equation}\label{e:dissipation:lb}
    \cD(c,\theta) \geq \eps .
  \end{equation}
\end{theorem}
\begin{proof}
  Let us show, that there exists $L=L(\delta)$ such that
  \begin{equation}\label{e:WcL:delta}
    \int_0^L W(x) \abs*{ c(x) - c_\theta^\eq(x)}\dx{x} \geq \abs*{\int_0^L W(x) c_\theta^\eq(x) \dx{x} - \int_0^L W(x) c(x)\dx{x} } \geq \frac{\delta}{2}
  \end{equation}
  Indeed, if $\theta>0$, it holds $\int_0^L W(x) c_\theta^\eq(x)\dx{x} \to \infty$ as $L\to \infty$ and $\int_0^L W(x) c(x) \to \rho -\theta < \infty$ and the statement holds. For $\theta_\eq + \delta \leq \theta<0$, we have by the monotonicity of $\theta \mapsto \int W(x) c_\theta^\eq(x)\dx{x}$
  \[
    \int W(x) c_\theta^\eq(x)\dx{x} - \int W(x) c(x) \geq \int W(x) c_{\theta_\eq}^\eq(x) \dx{x} - \int W(x) c(x)\dx{x} = \theta - \theta_\eq \geq \delta
  \]
  and the statement \eqref{e:WcL:delta} also holds. The argument for $\theta_\eq - \Theta\leq \theta \leq \theta_\eq - \delta$ is similar. Now, we can use the truncated logarithmic Sobolev inequality~\eqref{eq:epi0:L} of Proposition~\ref{prop:epi0}, the truncated Pinsker inequality~\eqref{e:Pinsker:L} from Lemma~\ref{lem:Pinsker} as well as the Assumption~\ref{ass:VW} to estimate
  \begin{align*}
    \cD(c,\theta)&\geq \frac{1}{C_{\LSI}} \int_0^L \frac{c_\theta^\eq(x)}{\omega(x)} \Psi\bra*{ \frac{c(x)}{c_\theta^\eq(x)}} \dx{x} \\
    &\geq \frac{c_0}{2 W(L)\,C_{\LSI} } \int_0^L W(x) \abs*{ c(x) - c_\theta^\eq(x)}\dx{x} \geq \frac{\delta}{C} =: \eps,
  \end{align*}
  for some $C=C(V,W,\Theta,L)$.
\end{proof}

\subsection{Quantitative long-time behavior: Proof of Theorem~\ref{thm:QuantLongTime}}

The last ingredient are stability estimates for solutions, such that the Assumption~\eqref{eq:epi:moment} of Theorem~\ref{thm:epi} is preserved uniformly in time by the evolution. Let us summarize the bounds obtained from the existence and regularity so far.
\begin{corollary}\label{cor:stability:sup}
  Let $c$ be the solution to~\eqref{A1}, \eqref{B1}, \eqref{D1}, \eqref{E1} with initial value satisfying~\eqref{AJ2}. Then for any $t_0>0$ exists $C_0<\infty$ such that it holds
  \begin{equation}\label{e:c:F:infty:bound}
    \sup_{t\geq 0} \abs{\theta(t)} \leq C_0 \; , \qquad \sup_{t\geq t_0} \norm{c(\cdot,t_0)}_\infty \leq C_0 \qquad\text{and}\qquad \sup_{t\geq t_0} \cF(c(\cdot,t)) \leq C_0 .
  \end{equation}
\end{corollary}
\begin{proof}
    The bound on $\norm{c(\cdot,t_0)}_\infty$ on any time interval $[t_0,T_0]$ is a consequence of Lemma~\ref{lem:stability:sup}. Therewith, we can proof the bound $\cF(c(\cdot,t_0))\leq C_0$. Indeed, we calculate using the representation~\eqref{e:FreeEnergyRelEnt} and Assumption~\ref{ass:VW}
    \begin{align*}
      \MoveEqLeft{\cF(c(\cdot,t_0)) = \cH\bra*{c| c_{\theta_{\eq}}^\eq} + \tfrac{1}{2} \bra*{ \theta(t_0) - \theta_\eq}^2} \\
      &\leq \int c \log c \, \dx{x} + \int \bra*{ V - \theta(t_0) W} c \dx{x} - \int c \dx{x} + \int c_{\theta_{\eq}^\eq} \dx{x} +  \tfrac{1}{2} \bra*{ \theta(t_0) - \theta_\eq}^2 \\
      &\leq C  \bra*{ 1 +  (1+\abs{\theta(t_0)}) \int W \,c \dx{x} } +\tfrac{1}{2} \bra*{ \theta(t_0) - \theta_\eq}^2 \\
      &\leq C \bra*{ 1 + \abs{ \theta(t_0)}^2}
    \end{align*}
    where we in addition also used the conservation law $\theta(t) + \int W \, c\dx{x} = \rho$. To make this bounds uniform in time, we have from~\eqref{e:FreeEnergyRelEnt} and the previous estimates that $\frac{1}{2}(\theta(t_0) - \theta_\eq)^2 \leq \cF_{\rho}(c(t_0))\leq C_0$. This implies, that $\sup_{t\geq 0} \abs{\theta(t)} \leq \abs{\theta_{\eq}} + \sqrt{2 C_0}$ and the global in time bound~\eqref{e:c:F:infty:bound}, since the bound in Lemma~\ref{lem:stability:sup} only depends on $\sup_{t_0\leq t \leq T_0}\abs{\theta(t)}$.
\end{proof}
From the bound~\eqref{e:c:F:infty:bound} and especially~$\cF(c(t_0,\cdot))\leq C_0$, we can conclude the following:
For any $\delta > 0$ there exists by the identity~\eqref{e:FreeEnergyRelEnt}, the energy--energy-dissipation principle Lemma~\ref{lem:EED} and the lower bound on the dissipation in Theorem~\ref{thm:dissipation:lb} a constant $C=C(V,W,\delta, \Theta,C_0)>0$ such that
\begin{equation}\label{e:bound:bad:times}
    \abs{\set{ t \geq t_0:  \delta \leq \abs{\theta(t) -\theta_\eq}\leq \Theta}}  \leq C  .
\end{equation}
Moreover, the bound~\eqref{e:c:F:infty:bound} allows to estimate for any $t\geq t_0$
\begin{align}
    \MoveEqLeft{\int \omega(x)^k c_\theta^\eq(x) \Psi\bra*{\frac{c(x,t)}{c_\theta^\eq}}} \dx{x} = \int \omega(x)^k c(x,t) \log c(x,t) \dx{x} - \int \omega(x)^k c(x,t) \dx{x} \\
    &\qquad + \int \omega(x)^k \bra*{ V(x) - \theta(t) W(x)}  c(x,t)\dx{x} + \int \omega(x)^k c_{\theta}^\eq(x) \dx{x} \\
    &\leq \bra*{\log C_0 -1} \int \omega(x)^k c(x) \dx{x} +  \int \omega(x)^k \bra{C_{V} + C_0} C_0 W(x)  c(x)  \dx{x}  + C_{k,\theta} \\
    &\leq C \bra*{1 + \norm{c}_\infty + \int \omega(x)^k \, W(x) \; c(x,t) \dx{x} }, \label{e:epi:moment}
\end{align}
Hence, it is left to show the uniform propagation in time of the moment
\[
  \int \omega(x)^k \, W(x) \; c(x,t) \dx{x} .
\]
However, to avoid assumptions on third derivatives on $W$, we use~\eqref{G1:alpha} with $\beta\in (0,1]$ of Theorem~\ref{thm:QuantLongTime} to bound
\[
  \omega(x) \leq c_0^{-1} W(x)^{\beta} .
\]
Hence, it is sufficient under assumption~\eqref{G1:alpha} to obtain uniform propagation in time of the moment
\begin{equation}\label{e:def:Wmoment}
   \int W(x)^p c(x,t) \dx{x}
\end{equation}
for $p=1 +k \beta$. The result is a consquence of the following Lemma, which itself is a refinement of  Lemma~\ref{lem:adjoint:Wbound} and~\ref{lem:adjoint:fluxBound} using $w_0(x) = W(x)^p$ as terminal data for the adjoint problem and the representation~\eqref{AI2} for solutions.
\begin{lemma}\label{lem:moment:bound}
  Let $V,W$ satisfy Assumption~\ref{ass:VW} and take $p>1$. Let $c$ be the solution to~\eqref{A1}, \eqref{B1}, \eqref{D1}, \eqref{E1} with initial value satisfying for some $p>1$
  \begin{equation*}
    \int W(x)^p\, c(x,0) \dx{x} \leq C_0 .
  \end{equation*}
  For $T,\delta>0$ let $\Theta$ and $m(\delta)$ satisfy the inequalities    $\sup_{t\in [0,T]} \theta(t)\le \Theta$ and $m\{t\in[0,T]:  \ \theta(t)\ge-\delta\}\le m(\delta)$.
  Then there exists $M_p=M_p(\Theta,\delta, m(\delta))$ independent of $T$, such that
  \begin{equation}\label{e:moment:bound}
    \sup_{t\in [0,T]} \int W(x)^p\, c(x,t) \dx{x} \leq M_p(C_0+1).
  \end{equation}
\end{lemma}
\begin{proof}
  The proof follows by a modification of Lemma~\ref{lem:adjoint:Wbound} and~\ref{lem:adjoint:fluxBound} using $w_0(x) = W(x)^p$ as terminal data for the adjoint problem and the representation~\eqref{AI2} for solutions.
  
  Thus to modify Lemma~\ref{lem:adjoint:Wbound} we let $w_\la(x,t)=\exp\pra*{\la \int_t^TH(\theta(s)+\delta) \dx{s} } \ W(x)^p$, where $H(\cdot)$ is the Heaviside function.  We have then similarly to ~\eqref{E2} that
\begin{align*} 
 &\exp\bra*{-\la \int_t^T  H(\theta(s)+\delta) \dx{s} } W(x)^{p-1}\mathcal{L}w_\la(x,t) =  -\la H(\theta(t)+\delta) W(x)+pW''(x) \\
 &\qquad\qquad\qquad\qquad\qquad +p(p-1)W'(x)^2W(x)^{-1}+\bra*{\theta(t)W'(x)-V'(x)} p W'(x) \  . 
\end{align*}
From Assumption~\ref{ass:VW} it follows that there exists $x_\del, \ \la_\Theta>0$ such that if $\la\ge \la_\Theta,$   then $\mathcal{L}w_\la(x,t)\le 0$ for $x>x_\del, \ 0<t<T$. Letting $w(x,t)$ be the solution of ~\eqref{A2}, ~\eqref{B2}, ~\eqref{A*2} with $w_0(x)=W(x)^p$, we have by the maximum principle that $w(x,t)\le [\sup_{0<t<T}w(x_\del,t)]\exp[\la_\Theta m(\delta)][W(x)/W(x_\del)]^p$ for $x>x_\del, \ 0<t<T$.  To see  that $\sup_{0<t<T}w(x_\del,t)\le C_\del$, where the constant $C_\del$ is independent of $T$, we set  $Y(t)=\exp\pra*{\la \int_t^TH(\theta(s)+\delta) \dx{s} } \Phi(X(t))$ in  Lemma~\ref{lem:adjoint:Wbound}.  Then one sees if $\la$ is sufficiently large that $Y(t)$ satisfies  ~\eqref{BJ2}, and ~\eqref{BB2} continues to hold up to some changes in the constants. In particular,
 we have now that $\mu(y,t)\le -\del/2 \, \exp[-\la m(\delta)]\sig(y,t)^2$.  With these bounds one can argue as before to estimate the expectation of $Y(T)^p$, and so we conclude that $w(x,t)\le C_pW(x)^p$ for $x>0, \ 0<t<T$, where the constant  $C_p$ is independent of $T$. 

For the modification of Lemma~\ref{lem:adjoint:fluxBound} we need  to show that ~\eqref{BL2} of Lemma~\ref{lem:adjoint:fluxBound} holds when $w_0(x)=W(x)^p$.  The main issue is to prove there exist constants $C,\nu>0$ independent of $T$ such that the exit time $\tau_{x,t}$ for the diffusion $X(\cdot)$ with initial condition $X(t)=x\le 1$ and dynamics ~\eqref{AW2}  satisfies the bound $\Prob(\tau_{x,t}>T)\le Ce^{-\nu(T-t)}, \ t<T$. We can establish this by moving to the variable $Y(t)$ defined in the previous paragraph. The key point is that the lower bound condition on $W'(\cdot)$ implies there are positive constants $C_0,c_0$ such that $\sig(y,t)\ge c_0$ for $y\ge C_0$.   We choose now a barrier $y_{\rm min}$ such that $Y(t)>y_{\rm min}$ for all $t<T$ provided $X(t)>0$ for all $t<T$. Evidently $y_{\rm min}$ depends on $m(\delta)$. Letting $\tau^*_{y,t}$ be the exit time from the interval $(y_{\rm min},\infty)$ for the diffusion $Y(\cdot)$ with $Y(t)=y>y_{\rm min},$ it is clear that if $x\le 1$ one has $\Prob(\tau_{x,t}>T)\le \Prob(\tau^*_{y,t}>T)$ for some $y$ close to $y_{\rm min}$. We may estimate $\Prob(\tau^*_{y,t}>T)$ by considering the function
\begin{equation} \label{A*5}
u(y,t) \ = \ \EX\pra*{  \exp\bra*{\delta\int_t^{\tau^*_{y,t}\wedge T}\sig(Y(s),s)^2+1 \dx{s}   } \ } \ , \qquad y>y_{\rm min}, \ t<T \ .
\end{equation}
Then $u(y,t)$ is a solution to the PDE
\begin{equation} \label{B*5}
\frac{\pa u(y,t)}{\pa t}+\frac{\sig(y,t)^2}{2}\frac{\pa^2 u(y,t)}{\pa y^2}+\mu(y,t)\frac{\pa u(y,t)}{\pa y}
+\delta [\sig(y,t)^2+1]u(y,t) \ = \ 0 \ ,
\end{equation}
with terminal and boundary conditions given by
\begin{equation}
u(y_{\rm min},t) \ = \ 1, \   \ t<T, \quad u(y,T) \ = \ 1 \ ,  \ \ y>y_{\rm min} \ .
\end{equation}
We make now a change of variable $y\ra z$ as in ~\eqref{BP2} so that the transformed diffusion and drift coefficients $\tilde{\sig}, \ \tilde{\mu}$ satisfy an inequality $\tilde{\mu}(z,t)\le -c_0\tilde{\sig}(z,t)^2, \ \tilde{\sig}(z,t)^2\ge c_1$ for $z>z_{\rm min}, \ t<T$, where $c_0,c_1$ are positive constants independent of $T$. We can now construct a time independent super-solution of the transformed PDE ~\eqref{B*5} by  finding a super-solution to a PDE similar to ~\eqref{BV2}. We obtain from this as in the proof of  Lemma~\ref{lem:adjoint:fluxBound} an exponential decay bound on  $\Prob(\tau^*_{y,t}>T)$, and hence on $\Prob(\tau_{x,t}>T)$.
\end{proof}
Hence, we can combine the result~\eqref{e:moment:bound} with the estimate~\eqref{e:bound:bad:times}, which gives the desired uniform in time moment bound, which was the last ingredient for the proof of Theorem~\ref{thm:QuantLongTime}

\begin{proof}[Proof of Theorem~\ref{thm:QuantLongTime}]
  Let $\delta>0$ be such that $\theta_\eq + 2\delta \leq 0$ and choose $t_0>0$. To apply Theorem~\ref{thm:epi}, we have to verify~\eqref{eq:epi:moment}. We obtain from the estimate~\eqref{e:epi:moment} together Corollary~\ref{cor:stability:sup} and the~\eqref{G1:alpha} that for any $t\geq T_\delta>0$
  \begin{align*}
    \int \omega(x)^k c_\theta^\eq(x) \Psi\bra*{\frac{c(x,t)}{c_\theta^\eq(x)}} \dx{x}
    \leq C \bra*{ 1 +  \int W(x)^{1+k\beta} c(x,t) \dx{x}} .
  \end{align*}
  Hence, the estimate~\eqref{e:moment:bound} of Lemma~\ref{lem:moment:bound} together with~\eqref{e:bound:bad:times} yields that Assumption~\eqref{eq:epi:moment} is satisfied for a constant $C_k$ uniform in time. An application of Theorem~\ref{thm:epi} yields for any $t\geq t_0$
  \begin{align*}
   \pderiv{}{t} \cF_{\rho}(c(t)) &\leq -\lambda \cF_{\rho}(c(t))^{\frac{1+k}{k}} \qquad\text{with} \qquad \lambda = C^{-\frac{k}{k+1}} C_{\LSI}^{-1},
  \end{align*}
  which integrates for any $t\geq t_0$ to
  \[
    \cF_\rho(c(t)) \leq \frac{1}{\bra[\big]{ \cF_{\rho}(c(\cdot,t_0))^{-\frac{1}{k}} + \frac{1}{k} \lambda (t-T_\delta)}^{k}} .
  \]
  Since $t_0>0$, we have $\cF_\rho(c(t_0))\leq C_0$ by Corollary~\ref{cor:stability:sup}, we conclude the proof for the algebraic decay. The arguments for the proof of the exponential decay follow the same lines, but do not need any uniform in time moment bounds, by using the linear energy dissipation estimate~\eqref{eq:epi:linear}.
\end{proof}

\bigskip

\section{A family of discrete Fokker-Planck equations and their gradient structures}\label{s:Interpol}
In this section, we want to motivate the connection of the continuum model~\eqref{A1}, \eqref{B1}, \eqref{D1}, \eqref{E1} and the Becker-D\"oring model~\eqref{P1}, \eqref{Q1} by interpolating between both of them with a family of discrete Fokker-Planck equations. In §\ref{s:discFP}, we introduce the family of discrete Fokker-Planck equations and show its formal connection to the continuum and the Becker-D\"oring equation model. In §\ref{s:discFP:GF}, we show that this family possesses a gradient flow structure.

\subsection{The family and its free energy}\label{s:discFP}
We consider functions $F:\ve\Z^+_0\ra\R$, where $\Z^+_0 = \set{0,1,2,3,\dots}$ and define the difference operator $D_\ve$ on such functions by  $D_\ve F(x)=[F(x+\ve)-F(x)]/\ve$. Denoting by $D_\ve^*$ the formal adjoint of $D_\ve$, we consider the partial difference equation on $\Z^+ = \set{1,2,3,\dots}$
\begin{align}
\partial_t c_\ve(x,t) -D^*_\ve\left[b_\ve(x,t)c_\ve(x,t)\right] &= -D_\ve^*D_\ve\left[a_\ve(x)c_\ve(x,t)\right] \ , x\in\ve\Z^+, \ t>0, \notag \\
{\rm where \ } b_\ve(x,t) &= a_\ve(x)\{\theta_\ve(t)\La_\ve(x)-\Ga_\ve(x)\} \ .  \label{A5}
\end{align}
Hereby, the functions $\Lambda_\ve, \Gamma_\ve : \ve\Z^+ \to \R$ are given. If the function $\theta_\ve(\cdot)$ is known, then \eqref{A5} is uniquely solvable with given Dirichlet condition $c_\ve(0,t), \ t\ge 0$.  An equilibrium  $c_{\ve,\theta}(x)$ for \eqref{A5} satisfies the identity
\begin{equation}\label{B5}
D_\ve[a_\ve(x)c_{\ve,\theta}(x)] \ = \ \{\theta\La_\ve(x)-\Ga_\ve(x)\} a_\ve(x)c_{\ve,\theta}(x)  \ ,
\end{equation}
whence
\begin{equation}\label{C5}
c_{\ve,\theta}(x) \ = \ a_\ve(x)^{-1}a_\ve(0)c_{\ve,\theta}(0)\prod_{0\le z< x, \ z\in\ve\Z}\left[1+\ve \{\theta\La_\ve(z)-\Ga_\ve(z)\}\right] \ .
\end{equation}
Let us introduce the relative entropy with respect to the equilibrium $\cF_\ve(c_\ve(\cdot),\theta)$ defined by
\begin{equation}\label{D5}
\cF_\ve(c_\ve(\cdot),\theta) \ = \ \ve\sum_{x\in \ve\Z^+} \left\{\log\pra*{\frac{c_\ve(x)}{c_{\ve,\theta}(x)}}-1\right\}c_\ve(x) \ .
\end{equation}
If $c_\ve(\cdot,t)$ is a solution to \eqref{A5} then
\begin{align} \label{E5}
\pderiv{}{t} \cF_\ve(c_\ve(\cdot,t),\theta_\ve(t)) &\ = \ \ve\sum_{x\in \ve\Z^+} \pa_t c_\ve(x,t) \log\pra*{\frac{c_\ve(x,t)}{c_{\ve,\theta_\ve(t)}(x)}} \\
&\qquad - \ve\sum_{x\in \ve\Z^+} c_\ve(x,t) \pderiv{}{t} \log c_{\ve,\theta_\ve(t)}(x) \ . \notag
\end{align}
Next we impose the  Dirichlet condition
\begin{equation}\label{F5}
c_\ve(0,t) \ = \ c_{\ve,\theta_\ve(t)}(0) \ , \quad t>0 \ ,
\end{equation}
on the solution  to \eqref{A5}, which is convenient to omit boundary terms in the summation by parts of the first term on the RHS of \eqref{E5}
\begin{multline} \label{G5}
\ve\sum_{x\in \ve\Z^+} \pa_t c_\ve(x,t) \log\left[\frac{c_\ve(x,t)}{c_{\ve,\theta_\ve(t)}(x)}\right] \ = \\
-\ve\sum_{x\in \ve\Z^+_0} \left\{D_\ve\left[a_\ve(x)c_\ve(x,t)\right]-b_\ve(x,t)c_\ve(x,t) \right\}D_\ve \log\left[\frac{c_\ve(x,t)}{c_{\ve,\theta_\ve(t)}(x)}\right] \ .
\end{multline}
Observe now that
\begin{align} \label{H5}
D_\ve \log\left[\frac{c_\ve(x,t)}{c_{\ve,\theta_\ve(t)}(x)}\right] &= \ve^{-1}\log\left[\frac{a_\ve(x+\ve)c_\ve(x+\ve,t)}{a_\ve(x)\left[1+\ve \{\theta_\ve(t)\La_\ve(x)-\Ga_\ve(x)\}\right] c_\ve(x,t)}\right]  \\
&=  \ve^{-1}\log\left[1+\ve\frac{D_\ve\left[a_\ve(x)c_\ve(x,t)\right] -b_\ve(x,t)c_\ve(x,t)}{a_\ve(x)\left[1+\ve \{\theta_\ve(t)\La_\ve(x)-\Ga_\ve(x)\}\right] c_\ve(x,t)}\right] \  .\notag
\end{align}
From \eqref{H5} we see that the RHS of \eqref{G5} is negative provided
\begin{equation}\label{I5}
1+\ve \{\theta_\ve(t)\La_\ve(x)-\Ga_\ve(x)\}    \ > \ 0 \ , \quad x\ge0, \ x\in\ve\Z, \ t\ge 0 \ .
\end{equation}

We consider next the second term on the RHS of \eqref{E5}. We have from \eqref{C5}  that
\begin{equation} \label{J5}
 \pderiv{}{\theta} \log c_{\ve,\theta}(x) \ = \  \pderiv{\log c_{\ve,\theta}(0)}{\theta}+\ve\sum_{0\le z< x, \ z\in\ve\Z}\frac{\La_\ve(z)}{1+\ve \{\theta\La_\ve(z)-\Ga_\ve(z)\}}
 =: \ W_\ve(x,\theta) \ ,
\end{equation}
which defines the function $W_\ve$.  It follows that
\begin{equation}\label{K5}
\ve\sum_{x\in \ve\Z^+} c_\ve(x,t)\pderiv{}{t} \log c_{\ve,\theta_\ve(t)}(x) \ = \ \pderiv{\theta_\ve(t)}{t} \; \ve\sum_{x\in \ve\Z^+} W_\ve(x,\theta_\ve(t))c_\ve(x,t) \ .
\end{equation}
We now introduce the conservation law
\begin{equation}\label{L5}
\ve\sum_{x\in \ve\Z^+} W_\ve(x,\theta_\ve(t))c_\ve(x,t) \ = \ H'_\ve(\theta_\ve(t)) \ ,
\end{equation}
where $H_\ve \in \cC^1(\R; \R)$ is given. The conservation law determines the dependency of $c_{\eps,\theta}(0)$ in~\eqref{C5} on $\theta$ via the identity~\eqref{J5}.

Defining now the function $\cG_\ve(c_\ve(\cdot),\theta)$ by
\begin{equation}\label{M5}
\cG_\ve(c_\ve(\cdot),\theta) \ = \ \cF_\ve(c_\ve(\cdot),\theta)+H_\ve(\theta) \ ,
\end{equation}
we see that by construction the function $t\ra \cG_\ve(c_\ve(\cdot,t),\theta_\ve(t))$ is decreasing for solutions~$c_\ve(\cdot,t)$ and~$\theta_\ve(t)$ of \eqref{A5}, \eqref{F5}, \eqref{L5}.

\subsection{The Becker-D\"oring model \texorpdfstring{for $\eps=1$}{}}\label{s:Interpolate:BD}
The Becker-D\"oring model \eqref{P1}, \eqref{Q1} is a particular case of \eqref{A5}, \eqref{F5}, \eqref{L5}.  Thus we take $\ve =1$ and $c_1(x,t)=c_{x+1}(t), \ x\ge 0, \ x\in\Z$, where $c_\ell(t), \ \ell\ge 1,$ is the solution to \eqref{P1}, \eqref{Q1}. Comparing \eqref{AB1}, \eqref{A5} we have that
\begin{equation}\label{N5}
a_1(x) \ = \ b_{x+1} \ , \quad  \La_1(x) \ = \  \frac{a_{x+1}}{b_{x+1}} \ , \quad     \Ga_1(x) \ = \ 1-z_s\frac{a_{x+1}}{b_{x+1}} \ , \ x\ge 0 \ .
\end{equation}
The boundary condition \eqref{AC1}, \eqref{F5} becomes
\begin{equation}\label{O5}
c_1(0,t) \ = c_{1,\theta_1(t)}(0) \ = \ z_s+\theta(t) \ .
\end{equation}
From \eqref{J5}, \eqref{N5}, \eqref{O5} we then have that
\begin{equation}\label{P5}
W_1(x,\theta) \ = \ \frac{1}{z_s+\theta}+\sum_{r=0}^{x-1} \frac{1}{z_s+\theta} \ = \ \frac{x+1}{z_s+\theta} \ .
\end{equation}
The conservation law \eqref{L5} becomes then
\begin{equation}\label{Q5}
\sum_{x=1}^\infty (x+1)c_1(x,t) \ = \  [z_s+\theta(t)]H_1'(\theta(t)) \ ,
\end{equation}
which can be rewritten as
\begin{equation}\label{R5}
\sum_{\ell=1}^\infty \ell c_\ell(t) \ = \ [z_s+\theta(t)][H_1'(\theta(t))+1] \ .
\end{equation}
Comparing \eqref{Q1}, \eqref{R5} we conclude that for some constant $C$
\begin{equation}\label{S5}
H_1(\theta) \ = \ \rho\log(z_s+\theta)-\theta + C \  .
\end{equation}
We have also from \eqref{C5}, \eqref{N5}, \eqref{O5} that 
\begin{equation}\label{T5}
c_{1,\theta}(x) \ = \ \frac{b_1}{b_{x+1}}(z_s+\theta)^{x+1}\prod_{\ell=1}^x \frac{a_\ell}{b_\ell} \ = \ (z_s+\theta)^{x+1}Q_{x+1} \ .
\end{equation}
We conclude from \eqref{D5}, \eqref{M5}, \eqref{S5}, \eqref{T5} that
\begin{align*}
\cG_1(c_1(\cdot),\theta) \ &= \ \cG(c(\cdot))-c_1[\log c_1-1]-\log(z_s+\theta)\sum_{\ell=2}^\infty \ell c_\ell+H_1(\theta) \\
&= \  \cG(c(\cdot))-c_1[\log c_1-1]-(\rho-c_1)\log c_1+H_1(\theta) \ = \ \cG(c(\cdot))+C \ ,
\end{align*}
 where $\cG(c(\cdot))$ is given by \eqref{Y1} and $C$ is a constant.

\subsection{The Fokker-Planck model \texorpdfstring{for $\eps\to 0$}{}}\label{s:Interpolate:FP}
Likewise the continuum model~\eqref{A1}, \eqref{B1}, \eqref{D1}, \eqref{E1} can be obtained from \eqref{A5}, \eqref{F5}, \eqref{L5} by taking the formal limit $\ve\to 0$. Therefore, we choose $\Lambda_0(x) = W'(x)$ and $\Gamma_0(x) = V'(x)$. Then, \eqref{C5} becomes~\eqref{C1} after choosing $c_{0,\theta}(0) = a_0(0) e^{-V(0) + \theta W(0)}$, which also identifies the boundary condition~\eqref{E1}. In addition, \eqref{J5} reads
\[
  W(0) + \int_0^x W'(z) \dx{z} = W(x) = W_0(x,\theta) .
\]
Hence, to obtain the conservation law~\eqref{D1}, we have to choose $H'_0(\theta) = \rho - \theta$ in~\eqref{L5}, whence $H_0(\theta) = -\frac{1}{2} \theta^2 + \rho \theta$. Therewith, the Lyapunov function \eqref{M5} becomes \eqref{e:FreeEnergy:FP} after some manipulations.

\subsection{Gradient flow structures}\label{s:discFP:GF}

The state manifold is given by all possible states satisfying the constraint~\eqref{L5}
\begin{equation}\label{GF:state_mfd}
  \cM_\eps = \set*{ (c,\theta) :  \eps \sum_{x\in \eps\Z^+} W_\eps(x,\theta) c(x) = H'(\theta)} .
\end{equation}
Therewith, possible variations of $(c,\theta)$ denoted by $(s,t)$ satisfy
\begin{equation}\label{GF:tangent_space}\begin{split}
  T_{(c,\theta)}\cM_\eps = \bigg\{ (s,t) : &\ s(0) =\partial_\theta c_{\eps,\theta}(0) t \\
  &\ \text{and}\quad \eps \sum_{x\in\eps\Z^+} \bra*{ t \partial_\theta W(x,\theta) c(x) W(x,\theta) s(x) } \\
  &\qquad\qquad +\eps \sum_{x\in\eps\Z^+} W_\eps(x,\theta) s(x) = H''(\theta) t \bigg\}.
\end{split}\end{equation}
Therewith, the variation of the free energy $\cG_\ve$~\eqref{M5}, denoted by $\delta \cG_\ve$ calculates as
\begin{align*}
  \delta \cG_\ve(c,\theta) \cdot (s,t) &= \eps \sum_{x\in\eps\Z^+} \log\bra*{ \frac{c(x)}{c_{\eps,\theta}(x)}} s(x) - t \eps \sum_{x\in\eps\Z^+} \partial_\theta c(x) \log c_{\eps,\theta}(x) + H'_\eps(\theta) t \\
  &=\eps \sum_{x\in\eps\Z^+}  \log\bra*{ \frac{c(x)}{c_{\ve,\theta}(x)}} s(x) - \eps t \sum_{x\in\eps\Z^+} W_\ve(x,\theta) c(x) + H'_\ve(\theta) t .
\end{align*}
Hence, from the conservation law~\eqref{L5} follows
\begin{equation}\label{GF:DG}
  \delta \cG_\ve(c,\theta) \cdot (s,t) = \eps \sum_{x\in\eps\Z^+}  \log\bra*{ \frac{c(x)}{c_{\ve,\theta}(x)}} s(x) ,
\end{equation}
which is independent of the variation in $t$ and hence we identify $\delta \cG_\eps(c,\theta)$ with the function $\eps \Z^+ \ni x \mapsto \log\frac{c(x)}{c_{\eps,\theta}(x)}$. The gradient flow formulation of~\eqref{A5} can be obtained by rewriting it in flux form
\begin{equation}\label{GF:FluxForm}
 \partial_t c_\eps(x,t) = \tfrac{1}{\eps} \bra[\big]{ J_\eps(x-\eps,t) - J_\eps(x,t)} = D_\eps^* J_\eps(x,t) ,
\end{equation}
where the net flux $J_\eps(x,t)$ from state $x$ to $x+\eps$ is given by
\begin{equation}\label{GF:flux}
  \eps J_\eps(x,t) = a_\eps(x) \bra*{ 1+ \eps\bra*{\theta(t) \Lambda_\eps(x) - \Gamma_\eps(x)}} c_{\eps}(x,t) - a_\eps(x+\eps) c_\eps(x+\eps,t) .
\end{equation}
Let us rewrite the identity~\eqref{B5} determining the equilibrium states~\eqref{C5} as a detailed balance condition
\begin{equation}\label{GF:DBC}
  a_\eps(x+\eps) c_{\eps,\theta}(x+\eps) = a_\eps(x) \bra*{ 1+ \eps\bra*{ \theta(t) \Lambda_\eps(x) - \Gamma_\eps(x)}} c_{\eps,\theta}(x) = k_{\eps,\theta}(x) .
\end{equation}
Therewith, we can rewrite the flux as
\begin{equation}\label{GF:flux:D}
  J_\eps(x,t) = k_{\eps,\theta}(x) \,\frac{1}{\eps} \bra*{ \frac{c_\eps(x,t)}{c_{\eps,\theta}(x)} - \frac{c_\eps(x+\eps,t)}{c_{\eps,\theta}(x+\eps)} } = - k_{\eps,\theta}(x) \, D_\eps\bra*{\frac{c_\eps(x,t)}{c_{\eps,\theta}(x)}} .
\end{equation}
This can be brought into the gradient of the energy~\eqref{GF:DG} by introducing the quantity
\begin{equation}\label{GF:hatc}
  \frac{D_\eps\bra*{\frac{c_\eps(x,t)}{c_{\eps,\theta}(x)}} }{D_\eps\bra*{\log\frac{c_\eps(x,t)}{c_{\eps,\theta}(x)}} } = \Lambda\bra*{\frac{c_\eps(x,t)}{c_{\eps,\theta}(x)},\frac{c_\eps(x+\eps,t)}{c_{\eps,\theta}(x+\eps)}} =: \widehat c_{\eps,\theta}(x,t) ,
\end{equation}
where the function $\Lambda: \R^+ \to \R^+$ is the logarithmic mean defined by
\begin{equation*}
 \Lambda(a,b) = \int_0^1 a^s b^{1-s} \dx{s} = \begin{cases}
                                                \frac{a-b}{\log a - \log b} &, a\ne b \\
                                                a &, a= b
                                              \end{cases}.
\end{equation*}
Therewith, we finally obtain the identity
\begin{equation}\label{GF:flux:grad}
\begin{split}
  J_\eps(x,t) &= - k_{\eps,\theta}(x) \; \widehat c_{\eps,\theta}(x,t) \; D_\eps\bra*{ \log\frac{c_\eps}{c_{\eps,\theta}}}(x,t) \\
  &=  - k_{\eps,\theta}(x) \; \widehat c_{\eps,\theta}(x,t) \; D_\eps \delta \cG_\eps(c,\theta)(x,t) ,
\end{split}
\end{equation}
which yields in~\eqref{GF:FluxForm}
\begin{equation}\label{GF:eq:deltaG}
  \partial_t c(x,t) = - D_\eps^*\bra[\big]{ k_{\eps,\theta}\; \widehat c_{\eps,\theta}\; D_\eps \delta \cG_\eps(c,\theta) }(x,t)
\end{equation}
To conclude that~\eqref{GF:eq:deltaG} is a gradient flow, we show that the operator given by $D_\eps^*\bra[\big]{ k_{\eps,\theta} \, \widehat c_{\eps,\theta} \, D_\eps\varphi}$ is non-negative inducing a metric. Therewith, we justify its identification as a gradient. Also, this forces us to introduce the boundary condition~\eqref{F5}, which we rewrite as
\begin{equation*}
  \log \frac{c(0,t)}{c_{\eps,\theta}(0)} = 0 \ .
\end{equation*}
Therewith, let us introduce for a non-negative measure $\nu_\eps \in \cM(\eps \Z^+)$ weighted Sobolev spaces
\begin{align}\label{GF:H1:eps}
  H^1_\eps(\nu_\eps) &= \set[\Big]{ \varphi : \eps\Z^+ \to \R \;\Big|\; \varphi(0) = 0 \quad\text{and}\quad  \eps \sum_{x\in\eps\Z^+} \abs*{ D_\eps \varphi(x)}^2 \nu_\eps(x) < \infty } \\
  H^{-1}_\eps(\nu_\eps) &= \bra*{H^1_\eps(\nu_\eps)}^* \quad\text{with}\quad \norm{ s }_{H^{-1}_\eps(\nu_\eps)} = \sup_{\varphi\in H^1_\eps(\nu_\eps)} \frac{\skp{\varphi,s}}{\norm{ \varphi }_{H^{1}_\eps(\nu_\eps)}}  .
  \label{GF:H-1:eps}
\end{align}
By these definitions, the linear operator
\begin{equation}\label{GF:def:K}
  K_\eps[\nu_\eps] : H^1_\eps(\nu_\eps) \to H^{-1}_\eps(\nu_\eps) \quad\text{given by}\quad K_\eps[\nu_\eps] \varphi = D^*_\eps\bra*{ \nu_\eps D_\eps \varphi}
\end{equation}
is well defined. It is non negative, since $\skp{\varphi , K_\eps[\nu_\eps] \varphi} = \eps \sum_{x\in \Z^+} \abs*{ D_\eps \varphi(x)}^2 \nu_\eps(x) \geq 0$, where the boundary condition $\varphi(1)=0$ ensures no extra  terms in the summation by parts. Therewith, the gradient flow formulation \eqref{GF:eq:deltaG} becomes
\begin{equation}\label{GF:eq:K}
  \partial_t c = - K_\eps[k_{\eps,\theta} \widehat c_{\eps,\theta}] \, \delta \cG_\eps(c,\theta) .
\end{equation}
Solutions then satisfy the condition $\delta G_\eps(c,\theta)\in H^1_\eps(k_{\eps,\theta} \widehat c_{\eta,\theta})$ and in particular the boundary condition $\delta \cG_\eps(c,\theta)(0,t) = 0$, which translates to
\begin{equation*}
  \log\frac{c(0,t)}{c_{\eps,\theta}(0)} = 0 \qquad \text{whence}\qquad c(0,t) = c_{\eps,\theta}(0)  .
\end{equation*}
Together, with the constraint~\eqref{L5}, which implicitly defines $\theta$ in terms of $c$ the system is closed.
Note, that we immediately recover, that $\cG_\eps$ is a Lyapunov function
\begin{align*}
 \pderiv{}{t} \cG_\eps(c,\theta) &= \skp{ \delta \cG_\eps(c,\theta), \partial_t c} = - \skp{ \delta \cG_\eps(c,\theta) , K_\eps[k_{\eps,\theta} \widehat c_{\eps,\theta}] \, \delta \cG_\eps(c,\theta)} \\
 &=  - \eps \sum_{x\in \Z^+} k_{\eps,\theta}(x) \; \widehat c_{\eps,\theta}(x,t) \;\abs*{ D_\eps \log\frac{c}{c_{\eps,\theta}}(x,t)}^2
\end{align*}

\subsubsection{Gradient flow structure for \texorpdfstring{$\eps=1$}{the Becker-Döring model}}

Let us denote $n_\ell(t) = c_1(\ell-1,t)$ for $\ell =1,2,\dots$ to omit notational confusion.
The state manifold~\eqref{GF:state_mfd} becomes by using~\eqref{P5}, \eqref{S5} and \eqref{O5}
\begin{equation*}
  \cM_1 = \set*{ (n,\theta) : \sum_{\ell=1}^\infty \ell n_\ell = \rho , n_1 = z_s + \theta} .
\end{equation*}
Since, $n_1 = z_s + \theta$, we only consider the reduced state manifold
\begin{equation*}
  \cM_1 = \set*{ n : \sum_{\ell=1}^\infty \ell n_\ell = \rho } .
\end{equation*}
Then, it is easy to check, that the possible variations are given by
\begin{equation*}
  T_{n,\theta} \cM_1 = \set*{ s : \sum_{\ell=1}^\infty \ell s_\ell = 0 } ,
\end{equation*}
which can be checked to be consistent with~\eqref{GF:tangent_space}.

Now, let us compare the gradient structure, which we obtain for $\eps=1$ with the one contained in~\cite{Schlichting}. In §\ref{s:Interpolate:BD}, the free energy is already identified and can be written for any $z\in [0,z_s]$ as
\begin{equation*}
  \cF_z(n) = \sum_{\ell =1}^\infty z^\ell Q_\ell \; \lambda_{\text{B}}\bra*{\frac{n_\ell}{z^\ell Q_\ell}} + C_z \quad\text{with}\quad \lambda_{\text{B}}(a) = a \log a - a + 1 ,
\end{equation*}
for some constant $C_z$ only depending on $z$ and $Q_\ell$ as defined in~\eqref{T1}. By~\eqref{O5} it holds $n_1(t) - z_s = \theta(t)$ and therefore the local equilibrium state $c_{1,\theta}(x)$~\eqref{T5} is given by
\begin{equation}\label{GF:BD:omega}
   c_{1,\theta}(\ell-1) = n_1^\ell Q_\ell =: \omega_\ell(n_1).
\end{equation}
Hence, from~\eqref{GF:DBC}, we obtain
\begin{equation}\label{GF:BD:DBC}
  k_{1,\theta}(\ell-1) =  b_{\ell+1} c_{1,\theta}(\ell) = n_1 a_\ell  c_{1,\theta}(\ell-1)
\end{equation}
Therewith, it follows
\begin{align*}
  k_{1,\theta}(\ell-1)\widehat c_{1,\theta}(\ell-1) &= \Lambda\bra*{ a_\ell n_1 n_\ell, b_{\ell+1} n_{\ell+1}} =: \nu_\ell .
\end{align*}
The semi-norm introduced by $K_1$~\eqref{GF:def:K} is given for $\phi \in H^1_1(\nu)$ by
\begin{equation*}
  \skp{\varphi , K_1[\nu] \psi} = \sum_{\ell=1}^\infty \nu_\ell \bra*{ \varphi_{\ell+1} - \varphi_\ell} \bra*{\psi_{\ell+1} - \psi_\ell} .
\end{equation*}
To identify~\cite[Equation (1.14)]{Schlichting}, we note that there no restriction on the boundary values of the function space is assumed, but the $0$ boundary value for functions is implemented by using the modified gradient $\bra[\big]{\tilde D\varphi}_\ell= \varphi_{\ell+1} - \varphi_\ell - \varphi_1$. Moreover, one can check again using the explicit definition of the local equilibrium states~\eqref{GF:BD:omega} and the detailed balance condition~\eqref{GF:BD:DBC}
\begin{equation*}
  \nu_\ell = a_\ell \omega_\ell(z) \Lambda\bra*{ \frac{n_1 n_\ell}{\omega_1(z) \omega_\ell(z)} ,  \frac{n_{\ell+1}}{\omega_{\ell+1}(z)}} .
\end{equation*}
Once, the driving energy functional and metric in terms of the operator $K_1$ are identified, the gradient structures agree.

\subsubsection{Gradient flow structure for \texorpdfstring{$\eps \to 0$}{the Fokker-Planck model}}\label{s:GF:FP}

In the setting of §\ref{s:Interpolate:FP}, let us take the limit $\eps\to 0$ in the definition of the state manifold~\eqref{GF:state_mfd} to find
\begin{equation*}
  \cM_0 = \set*{ (c,\theta) : \int W(x) c(x) \dx{x} = \rho - \theta , c(0) = c_{\theta}^\eq(0) }
\end{equation*}
with possible variations
\begin{equation*}
  T_{(c,\theta)}\cM_0 =\set*{ (s,t) : \int W(x) s(x) \dx{x} = - t , s(0) = W(0) c_\theta^\eq(0) t }
\end{equation*}
The weighted Sobolev $H_0^1(\nu)$ is now defined for a measure $\nu \in \cM(\R^+)$ by taking the closure of function with $\varphi \in \cC^\infty(\R^+)$ with $\varphi(0)=0$ with respect to the weighted homogeneous Sobolev norm defined by
\begin{equation}\label{GF:H1:0}
  \norm{\varphi}_{H^1(\nu)}^2 = \int \abs{\partial_x \varphi}^2 \dx{\nu}
\end{equation}
with dual space $H^{-1}_0(\nu) = \bra*{ H^1_0(\nu)}^*$.
Therewith, the operator $K_0$~\eqref{GF:def:K} is defined as
\begin{equation*}
 K_0[\nu] : H^1_0(\nu) \to H^{-1}_0(\nu) \quad\text{with}\quad  \skp{\psi, K_0[\nu] \varphi } = \int \partial_x \psi \, \partial_x \varphi \dx{\nu} .
\end{equation*}
Finally, we obtain from~\eqref{GF:DBC} and~\eqref{GF:hatc} the identity
\begin{equation*}
  \lim_{\eps \to 0} k_{\eps,\theta} \widehat c_{\eps,\theta}(x,t) = a(x) c_{0}(x,t).
\end{equation*}
The driving free energy was already identified in §\ref{s:Interpolate:FP} and is given by~\eqref{e:FreeEnergy:FP}, which first variation is identified in analog to~\eqref{GF:DG} with the function $\log \frac{c_0(\cdot,t)}{c_{0,\theta}(\cdot)}$, where $c_{0,\theta}$ is given by~\eqref{C1}. Thus, the gradient flow formulation of the system \eqref{A1}, \eqref{B1}, \eqref{E1} becomes
\begin{equation*}
  \partial_t c_0(\cdot,t) = K_0[a(\cdot)\, c_0(\cdot,t)] \log\frac{c_0(\cdot, t)}{c_{0,\theta}(\cdot)} (x,t)  = \partial_x \bra*{ a(\cdot) c_0(\cdot,t) \partial_x \log \frac{c_0(\cdot,t)}{c_{0,\theta(t)}(\cdot)}} .
\end{equation*}

\bigskip

\appendix

\section{Auxiliary results and calculations}

\begin{lemma}\label{lem:Ass:intro}
Suppose $a,V,W$ satisfy Assumption~\ref{ass:VW:intro}, then $\tilde V$ and $\tilde W$ defined in~\eqref{L1} satisfy Assumption~\ref{ass:VW}. 
\end{lemma}
\begin{proof}
  We go though the individual points of Assumption~\ref{ass:VW}:
  \begin{enumerate}[ (a) ]
   \item First the second derivatives of $\tilde V$ and $\tilde W$ satisfy by~\eqref{tildeVW:d2} and~\eqref{tildeVW:d4} the estimates
   \begin{align*}
     \abs*{\pderiv[2]{}{z} \tilde W(z)} + \abs*{ \pderiv[2]{}{z} \tilde V(z)} &\leq \bra[\big]{ \abs*{V''(x(z))}+ \abs*{W''(x(z))} } a(x(z)) + \frac{\abs{ a''(x(z))}}{2} \\
     &\qquad + \frac{1}{2} \bra[\big]{ \abs{W'(x(z))} + \abs{V'(x(z))}} \abs{ a'(x(z))}  + \frac{\abs*{ a'(x(z))}^2}{4 a(x(z))}  \\
     &\leq C_0 + \frac{\abs*{ a'(x(z))}^2}{4 a(x(z))} \ , 
   \end{align*}
   by using Assumption~\ref{ass:VW:intro}~(a). From Assumption~\ref{ass:VW:intro}~(c), we obtain for $x\geq x_\delta$
   \[
     \frac{\abs*{ a'(x)}^2}{a(x)} \leq \delta W'(x) \abs{ a'(x)} \leq \delta C_0 ,
   \]
   where the last step is again Assumption~\ref{ass:VW:intro}~(a). Now, the smoothness and positivity of $a$ implies that $\max_{0\leq x \leq x_\delta} \frac{\abs*{ a'(x)}^2}{a(x)} < \infty$ concluding the desired result. 
   \item This is immediate from~\eqref{tildeVW:d3}.
   \item By using Assumption~\ref{ass:VW:intro}~(c), we have from~\eqref{tildeVW:d1} and~\eqref{tildeVW:d3} for $z(x) \geq z(x_\delta)$
   \begin{align*}
     \abs*{ \pderiv{}{z} \tilde V(z(x))} &\leq \bra*{ \abs{V'(x(z))} + \frac{1}{2} \abs*{\bra*{\log a(x(z))}'} }\sqrt{a(x(z))}  \\
     &\leq \delta W'(x(z)) \sqrt{a(x(z))} = \delta \pderiv{}{z} \tilde W(z(x))  \ .
   \end{align*}
   Likewise, we can estimate the second derivatives by using both conditions in (c) of Assumption~\ref{ass:VW:intro} and choosing $z(x) \geq z(x_\delta)$. Moreover, for the sake of presentation, we omit the arguments $z(x)$, $x(z)$ and use the representations of the derivatives~\eqref{tildeVW:d2}--\eqref{tildeVW:d4}
   \begin{align*}
     &{\abs*{ \pderiv[2]{}{z} \tilde V} +\abs*{ \pderiv[2]{}{z} \tilde W}}\\
     &\leq \bra*{ \abs{V''} + \abs{W''} } \, a + \frac{1}{2} \bra*{ \abs{V'} + \abs{W'}}  \sqrt{a} \, \abs*{\bra*{\log a}'} \sqrt{a}  +{ \frac{1}{2}\abs*{\bra*{\log a}''} \, a + \frac{1}{4}\abs*{\bra*{ \log a}'}^2 \,a }  \\
     &\leq \frac{1}{2}\bra*{3+\delta} \,\delta \bra*{ W' \sqrt{a}}^2   
     = \frac{1}{2}\bra*{3  +\delta} \, \delta\, \bra*{\pderiv{}{z} \tilde W}^2 \ .
   \end{align*}
  \item This is immediate from~\eqref{tildeVW:d3} and Assumption~\ref{ass:VW:intro}~(d).
  \item This follows by the definition of the coordinate change~\eqref{J1} from the identity
  \begin{align*}
    \int \tilde W(z) \exp\bra*{ - \tilde V(z)} \dx{z} &= \int W(x) \exp\bra*{ - V(x) - \frac{1}{2} \log a(x)} \pderiv{z(x)}{x} \dx{x} \\
    &= \int \frac{W(x)}{a(x)} \exp\bra*{ - V(x) } \dx{x} < \infty
  \end{align*}
  by Assumption~\ref{ass:VW:intro}~(e). 
  \end{enumerate}
\end{proof}

\bigskip

\section*{Acknowledgments} The authors would like to thank José Cañizo for fruitful discussions on the quantified convergence to equilibrium in the subcritical case. The authors thank the referees whose incisive and detailed comments have substantially improved the final version of the paper.
A.S.~gratefully acknowledges support by the German Research Foundation through the Collaborative Research Center 1060 \emph{The Mathematics of Emergent Effects}. J.G.C. ~gratefully acknowledges support through a Simons Foundation travel grant.

\end{document}